\documentclass[paper=a4, fontsize=11pt]{scrartcl} 

\usepackage[T1]{fontenc} 
\usepackage{fourier} 
\usepackage[english]{babel} 
\usepackage{amsmath,amsfonts,amsthm} 
\usepackage{amsmath, calligra, mathrsfs}
\usepackage{amssymb}
\usepackage{relsize} 
\usepackage{mathtools}
\usepackage{hyperref}
\usepackage{mathdots}
\usepackage{array}
\usepackage[mathscr]{euscript}
\usepackage{verbatim}
\usepackage{enumitem}
\setlist{nolistsep} 
\usepackage[dvipsnames]{xcolor}
\usepackage{mdframed} 
\usepackage{soulutf8}
\usepackage{stmaryrd}
\usepackage{tikz-cd}
\usepackage{hyperref}
\usepackage{lipsum} 

\usepackage{sectsty} 
\allsectionsfont{\centering \normalfont\scshape} 

\usepackage{fancyhdr} 
\usepackage{xurl}
\newcommand*{\sheafhom}{\mathcal{H}\kern -.5pt om}
\numberwithin{equation}{section} 
\numberwithin{figure}{section} 
\numberwithin{table}{section} 

\newtheorem{thm}{Theorem}[section]
\newtheorem{cor}[thm]{Corollary}
\newtheorem{prop}[thm]{Proposition}

\newtheorem{lem}[thm]{Lemma}

\newtheorem{quest}[thm]{Question}
\theoremstyle{definition}
\newtheorem{defn}[thm]{Definition}

\newtheorem{exmp}[thm]{Example}

\theoremstyle{remark}
\newtheorem{rem}[thm]{Remark}

\DeclareMathOperator{\St}{st}

\DeclareMathOperator{\lk}{lk}

\DeclareMathOperator{\Cone}{Cone}

\DeclareMathOperator{\Span}{Span}

\DeclareMathOperator{\Int}{int}

\DeclareMathOperator{\N}{\mathbb{N}}

\newcommand{\horrule}[1]{\rule{\linewidth}{#1}} 

\title{	
	\normalfont \normalsize 
	\textsc{} \\ [25pt] 
	\horrule{0.5pt} \\[0.4cm] 
	\huge Signature 0 toric varieties, wall crossings, and cross polytope-like structures

	\horrule{2pt} \\[0.5cm] 
}

\author{Soohyun Park } 

\author{Soohyun Park \\ \href{mailto:soohyun.park@mail.huji.ac.il}{soohyun.park@mail.huji.ac.il} } 

\date{\normalsize October 6, 2025} 

\begin{document}
	
	\maketitle 
	
	\begin{abstract}
		
		\noindent We describe the structure of simplicial locally convex fans associated to even-dimensional complete toric varieties with signature 0. They belong to the set of such toric varieties whose even degree Betti numbers yield a top gamma vector component equal to 0. The gamma vector is an invariant of palindromic polynomials whose nonnegativity lies between unimodality and real-rootedness. It is known (and expected more generally) that the cases where this top component is 0 are among the ``building blocks'' of those where it is nonnegative. This means minimality with respect to a certain restricted class of blowups. However, this equality to 0 case is currently poorly understood. In the course of addressing this situation, we find that this interpretation encodes \emph{intrinsic} combinatorial information on the fan in addition to earlier compatibility with existing natural combinatorial examples. \\
		
		\noindent Our main method uses wall crossings. The links of the fan come from a repeated suspension of the maximal linear subspace in its realization in the ambient space of the fan. Conversely, the centers of these links containing any particular line form a cone or a repeated suspension of one. The intersection patterns between these ``anchoring'' linear subspaces come from how far certain submodularity inequalities are from equality and parity conditions on their dimensions. This involves linear dependence and containment relations between them which are connected to optimization. We obtain these relations by viewing the vanishing of certain mixed volumes from the perspective of the exponents. Finally, these wall crossings yield a simple method of generating induced 4-cycles covering the minimal objects described above. We intersect rational equivalence relations with 2-dimensional orbit closures instead of 1-dimensional ones as in most combinatorial applications. \\

	\end{abstract}

	\section*{Introduction}

	Our main objective is to determine structural properties of toric varieties associated to simplicial locally convex fans (p. 255 and 259 of \cite{LR}) whose signature is equal to 0. These contain the set of such toric varieties where the top component of the gamma vector of the vector recording its even degree Betti numbers are equal to 0. The nonnegativity of the gamma vector of a palindromic polynomial is a positivity property lying between unimodality and real-rootedness. It appears in a wide range of combinatorial and geometric contexts ranging from permutation statistics to the Euler characteristics of nonpositively curved piecewise Euclidean manifolds (\cite{Athgam}, \cite{CD}). Note that the geometry of flag simplicial spheres ties many of these examples together. We are interested in an instance of the following question: \\

	\begin{quest}
		Given a positivity property of a (palindromic) polynomial (e.g. unimodality, gamma positivity, real-rootedness) and some initial constraints, what can be said about the starting polynomial or ``boundary conditions''? If we started with the $h$-polynomial of some simplicial complex, what can be said about the underlying simplicial complex? \\
	\end{quest}
	
	\color{black}
	
	Among flag simplicial spheres with nonnegative gamma vector such as simplicial complexes associated to fans of the toric described above, the ones with top gamma vector component equal to 0 are expected to be among the ``building blocks'' of the others by work of Lutz--Nevo \cite{LN}. To be more specific, this means being ``minimal'' with respect to edge subdivisions (Main Theorem 1.2 on p. 70 and 77 of \cite{LN} and (99) on p. 36 of \cite{Athgam}). In the setting of toric varieties, this means minimality with respect to a restricted class of blowups. \\

	However, the case where the top gamma vector component is equal to 0 is poorly understood and it is difficult to find results on it in the literature. One interesting piece of recent work where the equality case has been studied is by D'Al\`i--Juhnke-Kubitze--K\"ohne--Venturello \cite{DJJKKV} in the case of symmetric edge polytopes. In the subcritical regime of the Erd\H{o}s--R\'enyi model applied to the input graph, they find that the graph is a forest and the gamma vector is 0 with probability 1. The associated polytopes are free sums (see p. 1 of \cite{CD}) of cross polytopes coming from forests (Theorem 5.4 and Remark 5.5 on p. 507 -- 508 of \cite{DJJKKV}). Note that cross polytopes themselves are free sums of lines. By the work of Lutz--Nevo \cite{LN} mentioned above, they are generally expected to come from flag simplicial spheres where every edge of the associated simplicial complex is contained in an induced 4-cycle (Theorem 1.2 on p. 70 and 77 and p. 80 of \cite{LN}). That being said, it is difficult to describe when these 4-cycles (or ones containing every vertex) occur in an explicit manner. \\
	
	In order to extract concrete structural properties of cases where the top gamma vector component is equal to 0, we study the only known case where a component of the gamma vector has an algebro-geometric interpretation. More specifically, it was shown by Leung--Reiner \cite{LR} that the top component of the gamma vector is the signature of a toric variety $X_\Sigma$ of even dimension $d$ when the $h$-vector input records the even degree Betti numbers of the toric variety $X_\Sigma$ associated to a simplicial locally convex fan $\Sigma$ (p. 255 and 259 of \cite{LR}) that is the normal fan of a simple rational polytope (Theorem 1.1 on p. 255 of \cite{LR}). This signature being equal to 0 is equivalent to all degree $d$ monomials in the Chow ring of $X_\Sigma$ that have even degree on each ray divisor that and are supported on $\le \frac{d}{2}$ ray divisors being equal to 0 (Lemma 3.1 on p. 262, Lemma 3.2 and proof of Theorem 1.2(i) on p. 264 of \cite{LR}). In this setting, the same earlier recursive argument ((99) on p. 36 of \cite{Athgam}) implies that the signature 0 objects are minimal with respect to the restricted class of blowups mentioned above. \\
	
	Here, it is more natural to view these monomials as terms involving restrictions of conormal bundles to intersections of ray divisors $D_{i_1} \cap \cdots \cap D_{i_p}$ for $1 \le p \le \frac{d}{2}$ with odd exponents on each conormal bundle restriction. This is because the local convexity condition implies that these conormal bundle restrictions are globally generated (which was the main property used to show nonnegativity in \cite{LR}) and can be expressed as degree $d - p$ mixed volumes involving their divisor polytopes. In addition, we would like to have a more ``intrinsic'' understanding of how these geometric properties are linked to the combinatorial fan structure of $\Sigma$ beyond the compatibility with combinatorial examples described in Section 4 on p. 269 -- 278 of \cite{LR}. \\
	
	Our main tool for determining when these mixed volume terms vanish by studying the orthogonal complements of the divisor polytopes of these conormal bundle restrictions using wall crossings (Corollary \ref{mixvol0toconorm0}). Aside from yielding combinatorial structural properties of the underlying fan (Theorem \ref{sig0fanstr}), it will give a simple method of generating induced 4-cycles using intersections of conormal bundle restrictions (Lemma \ref{zeroconorm4cycle}) and connections to algebraic properties (Corollary \ref{suspalg}, Remark \ref{suspspecnon}). \\
	
	The motivation for applying these results to construct suspension structures was that vanishing of all conormal bundles on a complete fan implying a cross polytope structure of the fan from a self-product of copies of $\mathbb{P}^1$ (Proposition \ref{vanconormcross}). Afterwards, we combine this information with a generalization (Remark \ref{extaffequiv}, Corollary \ref{spanflatlink}, Remark \ref{flatambvar}) of the interpretation of graphs of supports of conormal bundles of these toric varieties by Leung--Reiner (Proposition A.1 on p. 282 -- 283 of \cite{LR}) in terms of maximal linear subspaces contained in lifts of links in certain realization spaces. As a result, we find that vanishing of conormal bundle restrictions generally yields a suspension structure around the ``center'' ray of the conormal bundle (Proposition \ref{conormrestrsusp}). \\
	
	Using this information, we can describe the realizations of the links as suspensions over linear subspaces (Corollary \ref{suspspecnon}, Proposition \ref{sign0specialcov}). In fact, the way these cones fit together can be described in terms of individual cones or repeated suspensions of them (Proposition \ref{conormrestrsusp}, Proposition \ref{sign0specialcov}). Finally, we view (non)vanishing of mixed volumes of polytopes from the perspective of the \emph{exponents} to find linear dependence and containment relations among the maximal linear subspaces inside realizations of the links that ``anchor'' these suspensions (Corollary \ref{oddtupzerovoltot}). We note that they come from equality cases of submodularity inequalities or cases where we are not far from equality and certain parity conditions. While these results apply to a rather general class of generalized permutohedra (see conditions for Proposition \ref{oddalgoutput}), our focus is eventually on a particular specialization of the polytopes. On the other hand, there is a general framework for dual mixed volumes (which we can view in the context of both the exponents and polar duals) in recent work of Gao--Lam--Xue \cite{GLX}. The results described are summarized below: \\
	
	\begin{thm} \textbf{(Structural properties of signature 0 toric varieties) \\} \label{sig0fanstr}
		Suppose that $X_\Sigma$ is a complete toric variety of even dimension $d$ associated to a simplicial locally convex fan $\Sigma$ (p. 255 and 259 of \cite{LR}) and has signature $\sigma(X_\Sigma) = 0$. \\
		\begin{enumerate}[topsep=0pt]
			\item The support of the link $|\lk_\Sigma(\rho)| \subset N_{\mathbb{R}}$ has the structure of a repeated suspension of the maximal linear subspace in its realization in $N_{\mathbb{R}}$. Conversely, the centers of links $\lk_\Sigma(\alpha)$ containing a particular line in their realizations in $N_{\mathbb{R}}$ form a cone or a repeated suspension of one. (Corollary \ref{suspspecnon}, Proposition \ref{sign0specialcov}) \\
			
			\item The results of Part 1 can be generalized to linear subspaces contained in realizations of links over $p$-cones $\sigma_{i_1, \ldots, i_p} \in \Sigma(p)$ in lifts from $N(\sigma_{i_1, \ldots, i_p})_{\mathbb{R}}$ to $N(\omega)_{\mathbb{R}}$ for a cone $\omega \subset \sigma_{i_1, \ldots, i_p}$ generated by a subcollection of the minimal generators $\rho_{i_1}, \ldots, \rho_{i_p}$ of $\sigma_{i_1, \ldots, i_p}$ for $1 \le p \le \frac{d}{2}$. An example on realizations of $p$-cones in different spaces is given in Remark \ref{extaffequiv}. (Proposition \ref{sign0specialcov}, Corollary \ref{suspspecnon}) \\ 
			
			\item Note that the mixed volume 0 conditions equivalent to $\sigma(X_\Sigma) = 0$ are trivial unless the rays involved are pairwise non-special (i.e. not in maximal linear subspaces of realizations of respective links -- see Definition \ref{specnonspecdef} and Corollary \ref{mixvol0toconorm0}). Assuming this, there is a ``dichotomy'' of structures of realizations of links $\lk_\Sigma(\sigma_{i_1, \ldots, i_p})$ in $N(\omega)_{\mathbb{R}}$ for a cone $\omega \subset \sigma_{i_1, \ldots, i_p}$ generated by a subcollection of the minimal generators $\rho_{i_1}, \ldots, \rho_{i_p}$ of $\sigma_{i_1, \ldots, i_p}$ between repeated suspensions and cones or repeated suspensions of a cone around a linear subspace in the realization when $1 \le p \le 
			\frac{d}{2}$. This is based on structural properties of the $p = 2$ case that apply to general $1 \le p \le \frac{d}{2}$. (Corollary \ref{p2conorm}, Corollary \ref{pconelikep2}) \\
			\pagebreak 
			\item To make full use of the vanishing mixed volumes resulting from $\sigma(X_\Sigma) = 0$ for $1 \le p \le \frac{d}{2}$ in general, we view the (non)vanishing of mixed volumes from the perspective of the \emph{exponents}. In the $p = 2, 3,$ and $\frac{d}{2}$ relations between maximal linear subspaces contained in link along with some parity conditions. Note that the former follow from equality cases of submodularity inequalities involving dimensions of Minkowski sums of (divisor) polytopes. \\
			
			\noindent The results here are specializations of ones that apply to more general classes of generalized permutohedra. Steps involved are connected to induction on dimension (Proposition \ref{minbaselift}) and optimization-related properties (Remark \ref{optpov}). Vanishing of all monomials supported on $p$ variables with odd exponents on each nef divisor involved for $1 \le p \le \frac{d}{2}$ generally seems to involve a mixture of the properties observed for the values $p = 2, 3, \frac{d}{2}$ that ``interpolates'' between them (see Corollary \ref{oddtupzerovoltot} for more details). \\
			
			\noindent For example, consider a pair of polytopes $P$ and $Q$ with $P \supset Q$ where $Q$ the Minkowski sum of a subcollection of the polytopes whose Minkowski sum is $P$. Denote by $\omega_X$ the orthogonal complement of a polytope $X$. In our setting, the orthogonal complements correspond to vector subspaces (i.e. ``flat parts'') of realizations of links in the spaces that we lift to. Relevant identities include those of the form $\omega_Q \subset \omega_P + \omega_R$ for a new polytope $R$ from equality cases of the submodularity inequality (e.g. $p = 2, 3$) and dependence properties of the form $\omega_P \subset \omega_R$ from the $p = \frac{d}{2}$ case. (Proposition \ref{p2mvodd}, Example \ref{submodineqeqgeom}, Proposition \ref{p3algsubmodex}, Proposition \ref{pdover2}, Corollary \ref{oddtupzerovoltot}) \\
		\end{enumerate}
	\end{thm}

	Finally, studying ``flat'' wall crossings induced by conormal bundles that are nef but not ample (the $p = 1$ case $D_\rho^d = D_\rho \cdot (-D_\rho)^{d - 1} = 0$ of the vanishing identities implied by $\sigma(X_\Sigma) = 0$) gives an easy way to generate induced 4-cycles (Lemma \ref{zeroconorm4cycle}). For example, they can be used to show that every ray of $\Sigma$ is contained in an induced 4-cycle if $\Sigma$ is a complete simplicial locally convex fan and has signature $\sigma(X_\Sigma) = 0$ (Corollary \ref{sig04cycle}). \\
	
	We note that they involve an application of rational equivalence relations intersected with 2-dimensional torus orbit closures instead of 1-dimensional ones as is usual in combinatorial applications (e.g. see \cite{ACEP}). In addition, the induced 4-cycles generated using this more algebraic method yield potential algebraic explanations for the ``other side'' of the suspensions studied combinatorially in earlier sections (Corollary \ref{suspalg}, Remark \ref{suspspecnon}, Corollary \ref{4cycleconstralg}). This involves a common off-wall ray of the wall crossing inducing a ``flat'' wall crossing with respect to the ``center'' ray. The objects involved are connected to the contraction theorem in the minimal model program for toric varieties (Theorem 14-1-9 on p. 422 of \cite{Mat}).  \\

	\section{Wall crossings and fan structures} \label{wallfan}

	We first review background on signature 0 toric varieties in Remark \ref{sign0background}. The main point is that having the signature equal to 0 is equivalent to vanishing of a collection of mixed volumes involving conormal bundle restrictions (which are nef/globally generated by local convexity condition). To study properties of their divisor polytopes, we use wall crossings (Corollary \ref{mixvol0toconorm0}). This is the main tool for narrowing down structural properties of the fan along with maximal linear subspaces contained in realizations of links in lifts to different spaces parametrized by ``centers'' of conormal bundle restrictions added (Remark \ref{extaffequiv}). It yields repeated suspension structures in local parts as well as how they are glued together. \\
	
	To apply the wall crossing results on vanishing conormal bundles, we start by showing at cross polytope structure implied by vanishing of all conormal bundles (Proposition \ref{vanconormcross}). The combinatorial properties are then generalized to obtain repeated suspension structures from vanishing of conormal bundle restrictions (Proposition \ref{conormrestrsusp}). Resulting ``covers'' by suspension structures are given in Corollary \ref{suspspecnon} and Proposition \ref{sign0specialcov}. A key point in the generalizations to the $p$-cones mentioned is that all maximal cones in links over $p$-cones contain a special ray of a conormal bundle restriction (Proposition \ref{specialcovpconext}). \\

	As mentioned above, we start with some background on known properties of toric varieties associated to simplicial locally convex fans whose signatures are 0. It includes some comments on our general approach to the problem. \\

	\begin{rem} \textbf{(Signature 0 background and setting up ``special ray covers'') \\} \label{sign0background}
		
		Let $P$ be a simple polytope with a locally convex normal fan (p. 255 and p. 259 of \cite{LR}). Since we are studying toric varieties $X_\Sigma$ of even (complex) dimension $d$ such that the signature $\sigma(X_\Sigma) = 0$, all monomials of degree $d$ in the ray divisors supported on at most $\frac{d}{2}$ ray divisors that have even exponents on each ray divisor must be equal to 0 (Lemma 3.1 on p. 262 and the proof of Theorem 1.2(i) on p. 264 of \cite{LR}). This follows from the expression of $\sigma(X_\Sigma)$ as a \emph{strictly} positive linear combination of nonnegative intersection numbers (Lemma 3.1 on p. 262 of \cite{LR} and Euler's formula on p. 286 and Problem B-4 on p. 287 of \cite{MS}). \\
		
		Locally, this corresponds to vanishing of top degree (degree $d - p$) products of conormal bundle restrictions on $D_{i_1} \cap \cdots \cap D_{i_p}$ for $1 \le p \le \frac{d}{2}$ with an odd exponent on each conormal bundle restriction. The local convexity of the fan $\Sigma$ implies that these divisors are nef. Thus, we will mainly focus on mixed volumes associated to odd tuples of exponents. In particular, we have that $D_\rho^d = D_\rho \cdot (-D_\rho)^{d - 1} = 0$ for all rays $\rho \in \Sigma(1)$. Converting this into information on the divisor polytope $P_{-D_\rho}^{D_\rho}$, this is equivalent to the statement that $\dim P_{-D_\rho}^{D_\rho} \le d - 2$ (i.e. that it is \emph{not} full-dimensional). Note that the ambient space of the polytope of a divisor on $D_\rho$ is $\rho^\perp \cap M_{\mathbb{R}}$ (6.1.1 on p. 262 and Lemma 3.2.4 on p. 118 of \cite{CLS}). \\
		
		Since $-D_\rho|_{D_\rho}$ is generated by global sections/nef by local convexity (see Lemma 3.2 on p. 264 of \cite{LR}, equivalent when the fan has full-dimensional convex support by Theorem 6.3.12 on p. 291 of \cite{CLS}), we can also think about polar duals to look at properties of the support function of $-D_\rho|_{D_\rho}$. In particular, the divisor $-D_\rho|_{D_\rho}$ being nef implies that the divisor polytope $P_{-D_\rho}^{D_\rho}$ is the convex hull of points $m_\sigma(-D_\rho|_{D_\rho})$ from the Cartier data of $-D_\rho|_{D_\rho}$ (Theorem 6.1.7 on p. 266 -- 267 of \cite{CLS}) and the ``slopes'' of the divisor polytope are of the form $(m_\sigma(-D_\rho|_{D_\rho}), -1)$ (p. 265 of \cite{CLS}). The polar dual connection is made more concrete in Proposition 1.2.8 on p. 26 and Proposition 1.2.10 on p. 26 -- 27 of \cite{CLS}. \\

		From this polar dual perspective, the maximal linear subspace in the realization of $\lk_\Sigma(\rho)$ in $N_{\mathbb{R}}$ (see Proposition A.1 on p. 282 -- 283 of \cite{LR}) corresponds to the orthogonal complement of $P_{-D_\rho}^{D_\rho}$. These are the ``special rays'' of $\lk_\Sigma(\rho)$ from the $p = 1$ case of Corollary \ref{mixvol0toconorm0} (see Definition \ref{specialcovpconext}). The condition that $\dim P_{-D_\rho}^{D_\rho} \le d - 2$ (i.e. that it is \emph{not} full-dimensional) is equivalent to every maximal cone of $\lk_\Sigma(\rho)$ containing a special ray (Definition \ref{specnonspecdef}). In particular, this implies that every ray $\rho \in \lk_\Sigma(\rho)$ is adjacent to a ray $\gamma$ that is a special ray of $\lk_\Sigma(\rho)$. Note that this carries over to sums of conormal bundle restrictions with the orthogonal complement of the divisor polytope of $\sum_{j \in A} (-D_{i_j}|_{D_{i_1} \cap \cdots \cap D_{i_p}})$ for $A \subset [p]$ corresponding to the maximal linear subspace in the realization of $\lk_\Sigma(\rho_{i_1}, \ldots, \rho_{i_p})$ lifted from $N(\rho_{i_1}, \ldots, \rho_{i_p})_{\mathbb{R}}$ to $N(\rho_{i_j} : j \in [p] \setminus A)_{\mathbb{R}}$. \\

	\end{rem}

	Since we repeatedly use the affine equivalence and polar duals mentioned above for conormal bundle restrictions mentioned in Remark \ref{sign0background}, we will explain this in more detail below after defining some notation. \\
	
	\begin{defn} \textbf{(Cone notation) }
		\begin{enumerate}
			\item Given a collection of rays $\rho_{i_1}, \ldots, \rho_{i_p}$, we will use $\sigma_{i_1, \ldots, i_p}$ to denote the cone spanned by them. \\
			
			\item Given a fan $\Sigma \subset N_{\mathbb{R}}$ and a cone $\omega \in \Sigma$, the vector space space $N(\omega)_{\mathbb{R}}$ denotes the one constructed from the quotient of $N$ by $N \cap \omega$ (see Lemma 3.2.4 on p. 118 of \cite{CLS}).  \\
		\end{enumerate}
	\end{defn}

	\begin{rem} \textbf{(More on conormal bundle restrictions and lifts of realizations of links) \\} \label{extaffequiv}
		
		When we have vanishing of conormal bundle restrictions given by \[ -D_{i_j}|_{D_{i_1} \cap \cdots \cap D_{i_p}} = 0, \] this means that $|\lk_\Sigma(\sigma_{i_1, \ldots, i_p})| \subset N(\sigma_{i_1, \ldots, \widehat{i_j}, \ldots, i_p})_{\mathbb{R}}$ is a codimension 1 vector subspace after being lifted from $N(\sigma_{i_1, \ldots, i_p})_{\mathbb{R}}$. Similarly, having \[ \sum_{j \in A} (-D_{i_j}|_{D_{i_1} \cap \cdots \cap D_{i_p}}) = 0 \] means that the lift of $\lk_\Sigma(\sigma_{i_1, \ldots, i_p})$ from $N(\sigma_{i_1, \ldots, i_p})_{\mathbb{R}}$ to $N(\rho_{i_j} : j \in [p] \setminus A)_{\mathbb{R}}$ contains a vector subspace of dimension at least 1. While this is referred to implicitly in Remark \ref{sign0background}, we will go into more detail since it is used repeatedly. The most direct approach would be to replace $\Sigma$ and $N_{\mathbb{R}}$ on in Proposition A.1 on p. 282 -- 283 of \cite{LR} by $\lk_\Sigma(\sigma_{i_1, \ldots, \widehat{i_j}, \ldots, i_p})$ and $N(\sigma_{i_1, \ldots, \widehat{i_j}, \ldots, i_p})_{\mathbb{R}}$ respectively. For sums of conormal bundles, we note that $P_{D + E} = P_D + P_E$ for nef divisors $D$ and $E$ (see p. 69 of \cite{Ful} for nef divisor equality and general inclusion on p. 192 of \cite{CLS}) and that orthogonal complements of sums are intersections of the orthogonal complements of the individual parts. Then, we can use the polar duals discussed in Remark \ref{sign0background}. To compare the conditions $-D_{j_1}|_{D_{j_1} \cap D_{j_2}} = 0$ and $-D_{j_1}|_{D_{j_1}} = 0$, we can think about the picture on p. 9 of \cite{Ful} and an ``actual'' plane in $N_{\mathbb{R}}$. \\

		Alternatively, we can generalize the proof of Proposition A.1 on p. 282 -- 283 of \cite{LR} on Weil divisor expansions and affine equivalence. Using the argument on p. 283 of \cite{LR}, we see that there is a scalar multiple we need to take into account while writing down restrictions of divisors on $X_\Sigma$ to $X_{\St_\Sigma(\sigma_{i_1, \ldots, i_p})}$, where $\sigma_{i_1, \ldots, i_p} \coloneq \Cone(\rho_{i_1}, \ldots, \rho_{i_p})$. For each $\alpha \in \Sigma(1)$, we can decompose the ray generator $u_\alpha$ of $\alpha$ (Definition \ref{raygen}) as \[ u_\alpha = \sum_{j = 1}^p b_j u_{i_j} + c_{\alpha, i_1, \ldots, i_p} u_\alpha^{i_1, \ldots, i_p} \] for some $b_j \in \mathbb{Z}$ and $c_{\alpha, i_1, \ldots, i_p} \in \mathbb{Z}_{\ge 0}$ via a direct sum decomposition $N \cong \bigoplus_{j = 1}^p \mathbb{Z} \rho_{i_j} \oplus N'_{i_1, \ldots, i_p}$ for some $N'_{i_1, \ldots, i_p} \cong  N/\bigoplus_{j = 1}^p \mathbb{Z} \rho_{i_j}$. This comes from an extension $B$ of $u_{i_1}, \ldots, u_{i_p}$ to a basis of $N$ using $d - p$ additional elements (p. 107 of \cite{Ful}, Lemma 5.3 on p. 288 in Ch. VII of \cite{Ew}, p. 300 and Proposition 11.1.8 on p. 519 of \cite{CLS}). Note that the dual of this extended basis is only rational instead of integral since our fans are only assumed to be simplicial. \\

		Substituting in the relation above, we have the following:

		\begin{lem} \label{conormrestr}
			Given $m \in M_{\mathbb{R}}$, we have that \[ \hspace{-12mm} -\sum_{j = 1}^p \langle m, u_{i_j} \rangle D_{i_1} \cdots D_{i_{j - 1}} D_{i_j}^2 D_{i_{j + 1}} \cdots D_{i_p} = \sum_{\alpha \in \lk_\Sigma(\Cone(\rho_{i_1}, \ldots, \rho_{i_p}))(1)} \langle m, u_\alpha \rangle D_\alpha  D_{i_1} \cdots D_{i_p} \] in $A^\cdot(X_\Sigma)$. \\
			
			On the restriction to $D_{i_1} \cap \cdots \cap D_{i_p}$, we have 
			
			\[ \sum_{j = 1}^p \langle m, u_{i_j} \rangle (- D_{i_j}) = \sum_{\alpha \in \lk_\Sigma(\Cone(\rho_{i_1}, \ldots, \rho_{i_p}))(1)} \frac{\langle m, u_\alpha \rangle}{c_{\alpha, i_1, \ldots, i_p}} D_\alpha. \]

		\end{lem}

		\begin{proof}
			From the rational equivalence relation, we get 
			
			\begin{align*}
				- \langle m, u_{i_1} \rangle D_{i_1} &= \sum_{\alpha \in \Sigma(1) \setminus \rho_{i_1}} \langle m, u_\alpha \rangle D_\alpha 
			\end{align*}
			
			for each $m \in M_{\mathbb{R}}$ after moving one variable to one side. \\
			
			After multiplying by $D_{i_1} \cdots D_{i_p}$ and separating the squarefree terms on the right hand side from the others, this implies that 
			
			\begin{align*}
				- \langle m, u_{i_1} \rangle D_{i_1}^2 \cdots D_{i_p} &= \sum_{\alpha \in \lk_\Sigma(\Cone(\rho_{i_1}, \ldots, \rho_{i_p}))(1)} \langle m, u_\alpha \rangle D_\alpha D_{i_1} \cdots D_{i_p} \\ 
				&+ \sum_{j = 2}^p \langle m, u_{i_j} \rangle D_{i_1} \cdots D_{i_{j - 1}} D_{i_j}^2 D_{i_{j + 1}} \cdots D_{i_p}.
			\end{align*}
			
			The division by terms $c_{\alpha, i_1, \ldots, i_p}$ comes from the relation \[ u_\alpha = \sum_{j = 1}^p b_j u_{i_j} + c_{\alpha, i_1, \ldots, i_p} u_\alpha^{i_1, \ldots, i_p} \] for some $b_j \in \mathbb{Z}$ and $c_{\alpha, i_1, \ldots, i_p} \in \mathbb{Z}_{\ge 0}$ coming from the direct sum decomposition $N = \bigoplus_{j = 1}^p \mathbb{Z} u_{i_j} \oplus N'$ for some $N' \cong N/\bigoplus_{j = 1}^p \mathbb{Z} u_{i_j}$. \\
			
		\end{proof}
		
		Then, we can substitute a dual element of a basis for the conormal bundle restrictions we are adding as in p. 282 -- 283 of \cite{LR} and p. 300 of \cite{CLS}. \\
		
	\end{rem}
	
	\color{black}

	Our main approach towards obtaining structural information on locally convex fans $\Sigma$ such that $\sigma(X_\Sigma)$ is to use wall crossings to analyze orthogonal complements to divisor polytopes of conormal bundle restrictions. This will then be further refined by a study of general conditions on polytopes where odd tuples fail to produce any top degree terms with nonzero mixed volumes. The key idea in this second part is to consider the exponents as indicator functions for the mixed volume and study how far submodularity inequalities are from equality. \\
	
	We start studying the structure of such fans with an alternate characterization of rays of the ambient fan of a toric variety that are perpendicular to the divisor polytope $P_D$ of a basepoint free divisor $D$. Note that every basepoint free divisor is nef and that the converse holds if the fan has convex support of full dimension and we are considering a Cartier divisor (Proposition 6.3.11 and Theorem 6.3.12 on p. 291 of \cite{CLS}). We can also compare the statement below with an alternate characterization of the generalized fan $\Sigma_{P_D}$ formed by cones of the normal ``fan'' of the divisor polytope $P_D$ (Proposition 6.2.5 on p. 279 of \cite{CLS}). \\

	\begin{defn} (\textbf{Ray generator in a lattice}, p. 29 of \cite{CL}) \\
		\label{raygen}
		Give a ray $\rho \in \Sigma(1)$ and $\Sigma \subset \N_{\mathbb{R}}$, we define $u_\rho$ is the unique generator of $\rho \cap N$. This will be used in proofs of the following results.
	\end{defn}

	\color{black}
	
	\begin{prop} \label{altpdperp}
		Suppose that $\Sigma$ is a simplicial fan such that $|\Sigma|$ is convex of full dimension. Let $D$ is basepoint free Cartier divisor of $X_\Sigma$ (nef by former assumption). Then, the rays $\alpha \in \Sigma(1)$ that are perpendicular to all of $P_D$ are exactly those in \[ \bigcap_{ \substack{ \tau \in \Sigma(d - 1) \text{ a wall } \\ D \cdot V_\Sigma(\tau) \ne 0 }  } \Span(\tau). \]
	\end{prop}
	
	\begin{rem} ~\\
		\vspace{-3mm}
		\begin{enumerate}
			
			\item Since $\Sigma$ is simplicial (which we will assume throughout), every Weil divisor on $X_\Sigma$ has a positive multiple that is a Cartier divisor (Proposition 4.2.7 on p. 180 of \cite{CLS}). So, we will use this condition on a positive multiple without explicitly stating that we are working with a Cartier divisor. \\
			
			\item In the setting of locally convex fans that we will mainly work in, we will substitute links $\lk_\Sigma(\omega)$ in place of $\Sigma$ as a simplicial complex structure on the cones (or stars $\St_\Sigma(\omega)$ if working in the quotient lattice $N(\omega)$ in place of $N$) itself. Also, we have that $\Span(\tau) = \tau - \tau$ and the span itself can be considered to be a cone in a refinement of the original fan including $-\rho$ for each ray $\rho$ (see p. 28 and Theorem 6.1.18 on p. 275 of \cite{CLS}). \\
			
		\end{enumerate}
	\end{rem}
	
	\begin{proof}
		
		The rays perpendicular to the entire divisor polytope are exactly those in every (maximal) normal cone to $P_D$. Note that its normal cones form a generalized fan (Definition 6.2.2 on p. 278 of \cite{CLS}). \\ 
		
		Since $D$ is basepoint free, the vertices of $P_D$ are given by the Cartier divisor data $m_\sigma(D)$ for maximal cones $\sigma \in \Sigma(d)$ (Theorem 6.1.7 on p. 266 -- 267 of \cite{CLS}). The rays we would like to look into are those where $\langle m_\sigma(D) - m_{\sigma'}(D), \alpha \rangle = 0$ for all pairs of distinct maximal cones $\sigma$ and $\sigma'$ (also see p. 272 of \cite{CLS}). Since the maximal cones can be traversed via a series of wall crossings, it suffices to consider maximal cones $\sigma, \sigma' \in \Sigma(d)$ such that $\sigma \cap \sigma'$ is a wall. \\

		The connection between $\langle m_\sigma(D) - m_{\sigma'}(D), \alpha \rangle$ and $D \cdot V_\Sigma(\tau)$ comes from the fact that $D \cdot V_\Sigma(\tau) = \langle m_\sigma(D) - m_{\sigma'}(D), u_{\sigma'/\tau} \rangle$ (Proposition 6.3.8 on p. 289 of \cite{CLS}). Since $\Sigma$ is simplicial, the rays of any maximal cone of $\Sigma$ form a basis for $N_{\mathbb{R}}$. In particular, we can write $\alpha$ as a (unique) linear combination of the rays $\rho \in \sigma'(1)$. We can rescale the basis elements to take them to be the minimal lattice elements $u_\rho$ in place of $\rho$. Recall that $\langle m_\sigma(D), u_\rho \rangle = -a_\rho(D)$ for any ray $\rho \in \sigma(1)$ (p. 262 of \cite{CLS}). In particular, this implies that $u_\rho$-components of $\alpha$ with respect to the rescaled $\sigma'(1)$ basis that lie in $\tau(1) = \sigma(1) \cap \sigma'(1)$ do not contribute anything to the difference $\langle m_\sigma(D) - m_{\sigma'}(D), \alpha \rangle = \langle m_\sigma(D), \alpha \rangle - \langle m_{\sigma'}(D), \alpha \rangle$. This implies that 
		
		\[ \hspace{-5mm} \langle m_\sigma(D) - m_{\sigma'}(D), \alpha \rangle = c_{\sigma'/\tau, \sigma'(1)}(\alpha) \langle m_\sigma(D) - m_{\sigma'}(D), u_{\sigma'/\tau} \rangle = c_{\sigma'/\tau, \sigma'(1)}(\alpha) (D \cdot V_\Sigma(\tau)),  \]
		
		where $c_{\sigma'/\tau, \sigma'(1)}(\alpha)$ denotes the coefficient of $u_{\sigma'/\tau}$ in the rescaled $\sigma'(1)$-basis expansion of $\alpha$. \\

		Note that $D \cdot V_\Sigma(\tau) = 0$ if and only if $m_\sigma(D) = m_{\sigma'}(D)$ (Proposition 6.3.15 on p. 292 of \cite{CLS}). Assuming that $m_\sigma(D) \ne m_{\sigma'}(D)$, this product is zero if and only if $\alpha \in \Span(\tau)$. Repeating this procedure for each pair of maximal cones of $\Sigma$ sharing a wall yields the intersection in the statement.

	\end{proof}

	We can use Proposition \ref{altpdperp} to understand the structure of rays relevant to the conormal bundle restrictions in the mixed volume 0 condition from Theorem \ref{zerovolcrit}. They can be phrased in terms of conormal bundles that restrict to 0. \\
	
	\begin{lem} \label{convlink}
		If $\Sigma \subset N_{\mathbb{R}}$ is locally convex, then the realization $|\St_\Sigma(\omega)| \subset N_{\mathbb{R}}$ has convex support for any cone $\omega \in \Sigma$. \\
	\end{lem}
	
	\begin{proof}
		
		Since $\Sigma$ is simplicial and locally convex, its cones form a flag simplicial complex (Proposition 5.3 on p. 279 of \cite{LR}). This means that $\St_\Sigma(\omega) = \bigcap_{ \rho \in \omega(1) } \St_\Sigma(\rho)$. Each of the fans intersected has convex support since $\Sigma$ is locally convex. Since a finite intersection of convex sets is convex, the fan $\St_\Sigma(\omega)$ must also have convex support. \\
	\end{proof}
	
	We now introduce some notation that will be used frequently. \\

	\begin{defn} \textbf{(Special/non-special rays of a link with respect to a given set) \\} \label{specnonspecdef}
		A ray of $\lk_\Sigma(\sigma_{i_1, \ldots, i_p})$ is called \textbf{special with respect to $A \subset [p]$} if it generates the orthogonal complement of the divisor polytope $P_{\sum_{j \in A} (-D_{i_j})}^{D_{i_1} \cap \cdots \cap D_{i_p}}$ in its ``usual'' ambient space. Otherwise, a ray of $\lk_\Sigma(\sigma_{i_1, \ldots, i_p})$ is called \textbf{non-special with respect to $A$}. If $A = \{ j \}$, we may simply say that a ray of $\lk_\Sigma(\sigma_{i_1, \ldots, i_p})$ is special/non-special with respect to $\rho_{i_j}$. \\
		
		Dually, we can consider the special rays of $\lk_\Sigma(\sigma_{i_1, \ldots, i_p})$ to be generators of the maximal linear subspace in the realization of $\lk_\Sigma(\sigma_{i_1, \ldots, i_p})$ in its lift from $N(\sigma_{i_1, \ldots, i_p})_{\mathbb{R}}$ to $N(\rho_{i_j} : j \in [p] \setminus A)_{\mathbb{R}}$ (see Remark \ref{sign0background} and Remark \ref{extaffequiv}). Note that $N(\omega)_{\mathbb{R}}$ implicitly makes a choice of a codimension $\dim \omega$ vector subspace of $N_{\mathbb{R}}$ whose intersection with the vector space generated by the rays of $\omega$ is 0. \\
	\end{defn}

	Using wall crossings to analyze the classification in Definition \ref{specnonspecdef} is our main approach to narrowing down the general structure of the fan $\Sigma$. \\

	\begin{cor} \label{mixvol0toconorm0}
		Suppose that $\Sigma$ is a (complete) simplicial locally convex fan. Let $D = \sum_{j \in A} (-D_{i_j})$. For conormal bundle restrictions on links/stars over cones of $\Sigma$, the intersection from Proposition \ref{altpdperp} yields generators of the intersection given by ``special'' rays of $\lk_\Sigma(\sigma_{i_1, \ldots, i_p})$ with the following properties. The rest of the rays of $\lk_\Sigma(\sigma_{i_1, \ldots, i_p})$ will be termed ``non-special''. \\
		
		\begin{itemize}
			\item $D_{i_j} \cdot V_{\lk_\Sigma(\sigma_{i_1, \ldots, i_p})}(\sigma \setminus \gamma) = 0$ for all $j \in A$ and every maximal cone $\sigma$ of the link containing $\gamma$. Equivalently, we have that $D_{i_j}|_{D_\gamma \cap D_{i_1} \cap \cdots \cap D_{i_p}} = 0$ for all $j \in A$. \\
			
			\item $D_\delta \cdot V_{\lk_\Sigma(\sigma_{i_1, \ldots, i_p})}(\sigma \setminus \gamma) = 0$ for every maximal cone $\sigma$ of the link containing $\delta$ and $\gamma$. Equivalently, we have that $D_\delta|_{D_\delta \cap D_\gamma \cap D_{i_1} \cap \cdots \cap D_{i_p}} = 0$ for all non-special rays $\delta$ and special rays $\gamma$ of $\lk_\Sigma(\sigma_{i_1, \ldots, i_p})$. \\
		\end{itemize}
		
		The special rays generate the linear subspace of vectors perpendicular to all of $P_D$. Each maximal cone has the same number of special rays. 
		
	\end{cor}
	
	\begin{proof}
		
		Part 3 of Theorem \ref{zerovolcrit} implies that the mixed volumes involving conormal bundle restrictions are determined by dimensions of polytopes associated to divisors that are sums of conormal bundle restrictions given by restrictions of $-D_{i_j}$ for $j \in A \subset [p]$ to $D_{i_1} \cap \cdots \cap D_{i_p}$. Let $D = \sum_{j \in A} (-D_{i_j})$. \\

		Since the maximal cones can be traversed via wall crossings, we study the intersection from Proposition \ref{altpdperp} in 3 stages: \\

		\begin{enumerate}\bfseries
			\item Within a fixed maximal cone \\
			
			\item Wall crossings where the proposed span intersection generator/special ray is \emph{on} the wall \\
			
			\item  Wall crossings where the proposed span intersection generator/special ray is \emph{off} the wall \\
		\end{enumerate}
		
		We will now start studying the span intersection using these steps. \\

		\textbf{Step 1: Walls inside a fixed maximal cone \\}

		Suppose that we are working with a fixed maximal cone $\sigma$ of $\lk_\Sigma(\sigma_{i_1, \ldots, i_p})$. Since $\Sigma$ is simplicial, the linear independence of the rays implies that the intersection of spans of subsets of the rays in a fixed cone is given by the span of the intersection of the subsets. For example, consider 4 linearly independent vectors $v_1, v_2, v_3$, and $v_4$. If $a_1 v_1 + a_2 v_2 = b_2 v_2 + b_3 v_3 + b_4 v_4$, this is equivalent to $a_1 v_1 + (a_2 - b_2) v_2 - b_3 v_3 - b_4 v_4 = 0$. Since $v_1, v_2, v_3$, and $v_4$ are linearly independent, this occurs if and only if $a_1 = 0, a_2 - b_2 = 0, b_3 = 0$, and $b_4 = 0$. Note that the second condition is equivalent to $a_2 = b_2$. The same reasoning holds for comparisons linear combinations of arbitary pairs of subcollections of a set of linearly independent vectors. This means that

		\[ \mathlarger{\bigcap}_{ \substack{ \tau \subset \sigma \\ \tau \in \lk_\Sigma(\sigma_{i_1, \ldots, i_p})(d - p - 1) \text{ a wall } \\ D \cdot V_\Sigma(\tau) \ne 0 }  } \Span(\tau) = \Span \left( \mathlarger{\bigcap}_{ \substack{ \tau \subset \sigma \\ \tau \in \lk_\Sigma(\sigma_{i_1, \ldots, i_p})(d - p - 1) \text{ a wall } \\ D \cdot V_\Sigma(\tau) \ne 0 }  } \tau(1)  \right). \]

		Rescaling using multiplicities (e.g. see p. 108 of \cite{Ful}), this is equivalent to 
		
		\[ \mathlarger{\bigcap}_{ \substack{ \tau \subset \sigma \\ \tau \supset \sigma_{i_1, \ldots, i_p} \\ \tau \in \Sigma(d - 1) \text{ a wall } \\ D \cdot V_\Sigma(\tau) \ne 0 }  } \Span(\tau) = \Span \left( \mathlarger{\bigcap}_{ \substack{ \tau \subset \sigma \\ \tau \supset \sigma_{i_1, \ldots, i_p} \\ \tau \in \Sigma(d - 1) \text{ a wall } \\ D \cdot V_\Sigma(\tau) \ne 0 }  } \tau(1)  \right) \]
		
		since $\Sigma$ is simplicial. \\
		
		Thus, the question reduces to understanding when $D \cdot V_\Sigma(\tau) \ne 0 $. Since $-D_{i_j}$ is nef on $D_{i_1} \cap \cdots D_{i_p}$ for every $j \in [p]$ by local convexity, we have that $-D_{i_j} \cdot V_\Sigma(\tau) \ge 0$ for each wall $\tau$ containing $\rho_{i_j}$. This means that $D \cdot V_\Sigma(\tau) = 0$ if and only if $D_{i_j} \cdot V_\Sigma(\tau) = 0$ for all $j \in A$. More generally, we are taking the orthogonal complement of $P_{D + E} = P_D + P_E$ for nef divisors $D$ and $E$ (see p. 69 of \cite{Ful}). In other words, we have that $D \cdot V_\Sigma(\tau) \ne 0$ if and only if $D_{i_j} \cdot V_\Sigma(\tau) \ne 0$ for some $j \in A$. This can be split into 2 cases: \\
		\vspace{-3mm}
		\begin{itemize}
			\item If $\rho_{i_j} \notin \tau$, then this means that $(\rho_{i_j}, \tau) \in \Sigma$. However, this case is not relevant since we are assuming that $\tau \supset \sigma_{i_1, \ldots, i_p} = (\rho_{i_1}, \ldots, \rho_{i_p})$. \\

			\item If $\rho_{i_j} \in \tau$, we can write $D_{i_j} \cdot V_\Sigma(\tau) \ne 0 \Longleftrightarrow -\langle u_{\rho_{i_j}}^*, u_\gamma \rangle  \ne 0$, where the dual is taken with respect to the $\sigma'(1)$-basis of the ambient vector space ($\sigma'$ being the other maximal cone of $\Sigma$ containing $\tau$) and $\gamma \coloneq \sigma / \tau$. Note that the dual depends on the choice of basis and that $V_\Sigma(\rho_{i_j}) \cdot V_\Sigma(\sigma \setminus \gamma) = - \langle u_{i_j}^*, u_\gamma \rangle V_\Sigma(\gamma) \cdot V_\Sigma(\tau)$, where the dual is taken with respect to $\sigma'(1)$ and multiplicities yield rescaling to ray divisors (p. 421 of \cite{Mat} and p. 302 of \cite{CLS}). This point of view is expanded from the perspective of convexity/flatness of wall crossings in Remark \ref{picloconv} (see p. 420 -- 421 of \cite{Mat}). \\
		\end{itemize}
		Since the intersection is obtained as the \emph{complement} of rays $\gamma$ from the second case, the intersection for suitable walls in a fixed maximal cone $\sigma$ can be rewritten as
		\begin{align*}
			\begin{split}
				\mathlarger{\bigcap}_{ \substack{ \tau \subset \sigma \\ \tau \supset \sigma_{i_1, \ldots, i_p} \\ \tau \in \Sigma(d - 1) \text{ a wall } \\ D \cdot V_\Sigma(\tau) \ne 0 }  } \Span(\tau) &= \Span \left( \mathlarger{\bigcap}_{ \substack{ \tau \subset \sigma \\ \tau \supset \sigma_{i_1, \ldots, i_p} \\ \tau \in \Sigma(d - 1) \text{ a wall } \\ D \cdot V_\Sigma(\tau) \ne 0 }  } \tau(1)  \right) \\
				&= \Span \left(  \mathlarger{\bigcap}_{\substack{ \gamma \in \sigma(1) \\ \gamma \notin \{ \rho_{i_1}, \ldots, \rho_{i_p} \} \\ D_{i_j} \cdot V_\Sigma(\sigma \setminus \gamma) \ne 0 \text{ for some } j \in A } } \sigma(1) \setminus \gamma \right) \\
				&= \Span \biggl \{ \rho_{i_1}, \ldots, \rho_{i_p} \} \cup  \{ \gamma \in \sigma(1) : \gamma \notin \{ \rho_{i_1}, \ldots, \rho_{i_p} \} \\ 
				&\quad \text{ and } D_{i_j} \cdot V_\Sigma(\sigma \setminus \gamma) = 0 \text{ for all } j \in A \} \biggr. \\
			\end{split}
		\end{align*}

		As mentioned above, we note that the dual is taken with respect to the rays of the maximal cone $\sigma'$ ``opposite'' to $\sigma$ that shares a wall with it (which we write as $\tau$). \\
		
		\textbf{Step 2: Wall crossings with a generator lying on the wall \\}
		
		Now suppose that $\gamma \in \sigma'(1)$ for another maximal cone $\sigma' \supset \sigma_{i_1, \ldots, i_p}$ of $\Sigma$. For example, this includes maximal cones $\sigma' \ni \gamma$ such that $\sigma \cap \sigma'$ is a wall. Note that such wall crossings traverse all the maximal cones $\sigma' \supset \sigma_{i_1, \ldots, i_p}$ of $\Sigma$ that contain $\gamma$. Suppose that $\gamma \notin \{ \rho_{i_1}, \ldots, \rho_{i_p} \}$. In order to have

		\begin{align*}
			\begin{split}
				\gamma &\in \mathlarger{\bigcap}_{ \substack{ \tau \subset \sigma' \\ \tau \supset \sigma_{i_1, \ldots, i_p} \\ \tau \in \Sigma(d - 1) \text{ a wall } \\ D \cdot V_\Sigma(\tau) \ne 0 }  } \Span(\tau) \\ 
				&= \Span \left(  \mathlarger{\bigcap}_{\substack{ \gamma' \in \sigma'(1) \\ \gamma' \notin \{ \rho_{i_1}, \ldots, \rho_{i_p} \} \\ D_{i_j} \cdot V_\Sigma(\sigma' \setminus \gamma') \ne 0 \text{ for some } j \in A } } \sigma'(1) \setminus \gamma' \right) \\ 
				&= \Span (\{ \rho_{i_1}, \ldots, \rho_{i_p} \} \cup \{ \gamma' \in \sigma'(1) : \gamma' \notin \{ \rho_{i_1}, \ldots, \rho_{i_p} \} \\
				&\quad \text{ and } D_{i_j} \cdot V_\Sigma(\sigma' \setminus \gamma') = 0 \text{ for all } j \in A \} ) \\
			\end{split}
		\end{align*}
		
		for \emph{any} such maximal cone $\sigma'$ (noting that the duals depend on the choice of basis and may change depending on which one is across a wall from the maximal cone we're on), we need to have $D_{i_j} \cdot V_\Sigma(\tau) = 0$ for all $j \in A$ and \emph{any} wall $\tau \supset \sigma_{i_1, \ldots, i_p}$ such that $\tau \ni \gamma$. After rescaling, this condition is equivalent to having $D_{i_j} \cdot V_{\lk_\Sigma(\sigma_{i_1, \ldots, i_p}, \gamma)}(\tau \setminus (\sigma_{i_1, \ldots, i_p}, \gamma)) = 0$ for all $j \in A$ for all walls $\tau \supset (\sigma_{i_1, \ldots, i_p}, \gamma)$. This covers all walls of $\lk_\Sigma(\sigma_{i_1, \ldots, i_p}, \gamma)$. The convexity of the support of the full-dimensional fan $\lk_\Sigma(\sigma_{i_1, \ldots, i_p}, \gamma)$ (Lemma \ref{convlink}) then implies that $D_{i_j}|_{D_\gamma \cap D_{i_1} \cap \cdots \cap D_{i_p}} = 0$ for all $j \in A$ (Proposition 6.3.15 on p. 292 of \cite{CLS}). \\
		
		\textbf{Step 3: Wall crossings with a generator off the wall \\}

		Finally, we consider the case of wall crossings where $\gamma$ is an off-wall ray. In Step 1 and Step 2, we imposed conditions that ensure that a ray $\gamma$ lies in the intersection of spans of walls $\tau$ such that $D \cdot V_\Sigma(\tau)$ for a wall $\tau$ of \emph{any} maximal cone that contains $\gamma$. In the course of a traversal of the maximal cones via wall crossings, the only way for the dimension of the intersection to decrease is for a ray $\gamma$ to be used as an off-wall ray of a wall crossing and be sent to a ray $\gamma'$ that no longer lies in the intersection of the spans of all the relevant walls. The rays that remain are those rays $\gamma$ that are still passed on to rays lying in the overall intersection after such a wall crossing. In particular, we cannot have a further decrease in dimension (which would occur since rays forming a cone of a simplicial fan $\Sigma$ are linearly independent). \\ 
		
		If we start with a collection of rays in the current maximal cone that lie in \[ \bigcap_{ \substack{ \tau \in \Sigma(d - 1) \text{ a wall } \\ D \cdot V_\Sigma(\tau) \ne 0 }  } \Span(\tau), \]
		
		we need the span including both the subcollection of the intersection rays in the wall and $\gamma$ to be preserved. The main tool we will use will be the wall relation (p. 301 of \cite{CLS}). Note that ``opposite'' rays $\gamma'$ across $\gamma$ after crossing a wall (containing $\sigma_{i_1, \ldots, i_p}$ in our setting) still corresponds to the intersection of $D$ and the same torus-invariant curve $V_\Sigma(\tau)$ since $\sigma \setminus \gamma = \sigma' \setminus \gamma'$. \\

		Let $\tau = \sigma \setminus \gamma$ and write $\tau = \sigma \cap \sigma'$, where $\sigma'$ is the other maximal cone of $\Sigma$ containing $\tau$. In order for the span of $\gamma$ and a subcollection of the rays in $\tau(1)$ to be carried over to $\sigma'$ in a nontrivial way, there is a nonzero vector $\alpha$ in the span intersection with linear ($\sigma(1)$ and $\sigma'(1)$ basis) expansions
		
		\begin{align*}
			\alpha &= a_\gamma u_\gamma + \sum_{\omega \in \tau(1)} a_\omega u_\omega \\
			&= b_{\gamma'} u_{\gamma'} + \sum_{\omega \in \tau(1)} b_\omega u_\omega 
		\end{align*}
		
		where $a_\gamma, b_{\gamma'} \ne 0$ and $a_\omega \ne 0 \Longleftrightarrow b_\omega \ne 0$. Here, the rays $\omega \in \tau(1)$ where $a_\omega = b_\omega = 0$ are those outside the span intersection that we are considering (e.g. from combining Step 1 and Step 2). \\
		
		The equality between the relations can be rewritten as 
		
		\[ a_\gamma u_\gamma + \sum_{\omega \in \tau(1)} (a_\omega - b_\omega) u_\omega - b_{\gamma'} u_{\gamma'} = 0. \]
		
		Comparing this to the wall relation \[ \sum_\rho (D_\rho \cdot V(\tau)) u_\rho = 0, \] 
		
		the relations above imply that
		
		\[ \frac{a_\omega - b_\omega}{a_\gamma} = \frac{D_\omega \cdot V_\Sigma(\tau)}{D_\gamma \cdot V_\Sigma(\tau)} \Longrightarrow b_\omega = a_\omega - \frac{D_\omega \cdot V_\Sigma(\tau)}{D_\gamma \cdot V_\Sigma(\tau)} a_\gamma. \]
		
		If $a_\omega = b_\omega = 0$, then \[ D_\omega \cdot V_\Sigma(\tau) = 0 \Longrightarrow D_\omega \cdot V_{\lk_\Sigma(\sigma_{i_1, \ldots, i_p}, \gamma, \omega)}(\tau \setminus (\sigma_{i_1, \ldots, i_p}, \gamma, \omega)) = 0 \] since $a_\gamma \ne 0$. Since this relation must hold for \emph{any} wall $\tau$ of $\Sigma$ containing $(\sigma_{i_1, \ldots, i_p}, \gamma, \omega)$, the convexity of $|\lk_\Sigma(\sigma_{i_1, \ldots, i_p}, \gamma, \omega)|$ implies that \[ D_\omega|_{D_\omega \cap D_\gamma \cap D_{i_1} \cap \cdots \cap D_{i_p}} = 0 \] by Proposition 6.3.15 on p. 292 of \cite{CLS}. \\
		
		In order for a ray $\gamma$ to remain in the intersection of spans \[ \bigcap_{ \substack{ \tau \in \Sigma(d - 1) \text{ a wall } \\ D \cdot V_\Sigma(\tau) \ne 0 }  } \Span(\tau) \] after \emph{all} wall crossings, its ``opposite ray'' $\gamma'$ under any wall crossing with $\gamma$ as an off-wall ray must have $D_{i_j}|_{D_\gamma \cap D_{i_1} \cap \cdots \cap D_{i_p}} = 0$ for all $j \in A$ and $D_\omega|_{D_\omega \cap D_\gamma \cap D_{i_1} \cap \cdots \cap D_{i_p}} = 0$ for any ray $\omega$ of a maximal cone containing $\gamma$ that doesn't lie in the span intersection at that point. The image of $\gamma'$ itself under such wall crossings must have the same properties. This cycles after a finite number of steps if the intersection is nonempty. The required properties are written in the statement.

	\end{proof}

	We can also think about special/non-special rays as describing links that are flat in the realization spaces that we lift to. \\

	\begin{cor} \textbf{(Wall crossings maintaining spans and flat links) \\} 
		\label{spanflatlink}
		\begin{enumerate}
			\item \textbf{(Maintaining spans of ray complements after wall crossings) \\} Let $\sigma$ be a maximal cone of $\Sigma$ containing a ray $\alpha$. Suppose that $D_\alpha \cdot V_\Sigma(\sigma \setminus \gamma) = 0$. Consider the wall crossing starting with $\sigma$ using $\gamma$ as an off-wall ray (i.e. ray to be replaced) that is sent to (i.e. replaced by) a ray $\gamma'$ forming a new maximal cone $\sigma = \Cone(\sigma/\gamma, \gamma')$. Then, we have that $\Span(\sigma \setminus \alpha) = \Span(\sigma' \setminus \alpha)$. \\
			
			\item \textbf{(Flat links in given realization spaces) \\} Consider a $p$-cone $\sigma_{i_1, \ldots, i_p}$ and a maximal cone $\sigma$ of $\Sigma$ containing $\sigma_{i_1, \ldots, i_p}$. Given a subset $A \subset [p]$, let $M$ be a maximal collection of non-special rays forming a cone of $\lk_\Sigma(\sigma_{i_1, \ldots, i_p})$ with respect to $A$ that is contained in $\sigma$. Then, we have that $|\lk_\Sigma(\sigma_{i_1, \ldots, i_p}, M)| \subset N(\rho_{i_j} : j \in [p] \setminus A)_{\mathbb{R}}$ is a vector subspace of $N(\rho_{i_j} : j \in [p] \setminus A)_{\mathbb{R}}$. This is actually a link of $\lk_\Sigma(\sigma_{i_1, \ldots, i_p})$ of maximal dimension among those whose supports lift to a vector subspace of $N(\rho_{i_j} : j \in [p] \setminus A)_{\mathbb{R}}$. \\
		\end{enumerate}
	\end{cor}
	
	\begin{proof}
		\begin{enumerate}
			
			\item By the wall relations, we have that \[ a_\gamma u_\gamma + \sum_{\omega \in \tau(1)} a_\omega u_\omega + a_{\gamma'} u_{\gamma'} = 0 \]
			
			for some $a_\gamma, a_{\gamma'} > 0$, where $\tau = \sigma \cap \sigma'$ is the wall used in the wall crossing. Recall that this is a positive multiple of the relation \[ \sum_{\rho} (D_\rho \cdot V(\tau)) u_\rho = 0. \]
			
			Since $D_\alpha \cdot V_\Sigma(\sigma \setminus \gamma) = 0$, we have that $a_\alpha = 0$ and \[ a_\gamma u_\gamma + \sum_{\omega \in \tau(1) \setminus \alpha} a_\omega u_\omega + a_{\gamma'} u_{\gamma'} = 0. \]
			
			This implies that \[ a_{\gamma'} u_{\gamma'} = - \left( a_\gamma u_\gamma + \sum_{\omega \in \tau(1) \setminus \alpha} a_\omega u_\omega \right). \]
			
			Since $a_{\gamma'} > 0$, this means that $\gamma' \in \Span(\gamma, \tau(1) \setminus \alpha)$. We note that $\Span(\tau(1) \setminus \alpha) \subsetneq \Span(\gamma', \tau(1) \setminus \alpha)$ since $\Sigma$ is simplicial. Comparing dimensions then implies that $\Span(\gamma', \tau(1) \setminus \alpha) = \Span(\gamma, \tau(1) \setminus \alpha)$. This can be rewritten as $\Span(\sigma' \setminus \alpha) = \Span(\sigma \setminus \alpha)$. \\

			\item This follows from repeated application of Part 1 and Corollary \ref{mixvol0toconorm0}. By ``special'' and ``non-special'', we will mean those with respect to the initial set $A \subset [p]$. Splitting the rays of the maximal cone $\sigma$ into special and non-special rays, a wall crossing with a special ray as an off-wall ray sends the special ray to a special ray of the new maximal cone that is formed and the non-special rays remain on the wall. Such wall crossings can be used to traverse the maximal cones of $\Sigma$ containing $\sigma_{i_1, \ldots, i_p}$ and a maximal \emph{fixed} collection of non-special rays contained in $\sigma$. \\
			
			\noindent Combining Corollary \ref{mixvol0toconorm0} with Part 1, we find that we are taking repeated intersections with $N(\rho_{i_j})_{\mathbb{R}}$ for $j \in A$ and $N(\delta)_{\mathbb{R}}$ for the non-special rays $\delta$ of the maximal collection of non-special rays forming a cone. Here, the vector space $N(\omega)_{\mathbb{R}}$ denotes a \emph{fixed} codimension $\dim \omega$ vector subspace of $N_{\mathbb{R}}$ whose intersection with $\Span(\omega)$ is 0. Alternatively, we note that the only rays left to consider in the maximal cones of the link $\lk_\Sigma(\sigma_{i_1, \ldots, i_p}, M)$ are the special rays of each maximal cone involved. These all span the same linear subspace, which is the maximal vector subspace contained in the realization $|\lk_\Sigma(\sigma_{i_1, \ldots, i_p})| \subset N(\rho_{i_j} : j \in [p] \setminus A)_{\mathbb{R}}$. We note that such a link must contain all the non-special rays in a given maximal cone in order for a linear subspace to be formed. 
		\end{enumerate}
	\end{proof}

	\begin{rem} \textbf{(Flat links in a given ambient space from different values of $p$) \\} \label{flatambvar}
		We note that there is some redundancy among the flat links discussed in Corollary \ref{spanflatlink}. More specifically, suppose that $|\lk_\Sigma(\sigma)| \subset N(\omega)_{\mathbb{R}}$ for some cone $\omega \subset \sigma$ generated by a subcollection of the rays of $\sigma$. This can come from ``flat links'' constructed out of $\dim \omega \le p \le \dim \sigma$ using $p$-cones $\sigma_{i_1, \ldots, i_p}$ with $\omega \subset \sigma_{i_1, \ldots, i_p} \subset \sigma$. The sets $A \subset [p]$ would be the subcollection of rays of $\sigma_{i_1, \ldots, i_p}$ \emph{not} coming from $\omega$. \\
	\end{rem}

	The consequences above are for local properties of the fan that are ``flat'' in a realization space based on information from Corollary \ref{mixvol0toconorm0}. However, we can also say more about the possible  fan structures resulting in the properties in Corollary \ref{mixvol0toconorm0}. The key point is to think about restrictions of conormal bundles to intersections of ray divisors being 0 as being analogous to cross polytopes. \\

	\begin{prop} \textbf{(Vanishing conormal bundles and cross polytopes) \\}
		\label{vanconormcross}
		Note that the $d$-fold self-product $\mathbb{P}^1 \times \cdots \times \mathbb{P}^1$ yields the cross polytope (Proposition 2.4.9 on p. 90 of \cite{CLS}) and that the conormal bundle $-D_{\rho}|_{D_\rho} = 0$ for all rays $\rho \in \Sigma(1)$. \\
		
		Conversely, we can show that a complete fan yielding a toric variety with $-D_{\rho}|_{D_\rho} = 0$ for all rays $\rho \in \Sigma(1)$ is the normal fan of a cross polytope. \\
		
	\end{prop}
	
	\begin{proof}
		Since $-D_{\rho}|_{D_\rho} = 0$, we have that the realization $|\lk_\Sigma(\rho)| \subset N_{\mathbb{R}}$ is a codimension 1 vector subspace of $N_{\mathbb{R}}$ that does \emph{not} contain $\rho$. Now consider a ray $\alpha \ne \rho$ of $\Sigma$. Its support is a codimension 1 vector subspace $|\lk_\Sigma(\alpha)| \subset N_{\mathbb{R}}$. \\

		Suppose that $|\lk_\Sigma(\alpha)| \ne |\lk_\Sigma(\rho)|$. Then, the fan $\lk_\Sigma(\alpha)$ must contain a ray $\mu$ that is on the same side of $|\lk_\Sigma(\rho)|$ as $\rho$. However, the only such ray is $\rho$ itself. This is because cones of a fan must intersect on faces (Definition 3.1.2 on p. 106 of \cite{CLS}), which are in the boundaries of the starting cones. Suppose that $\mu \ne \rho$. Write $\mu$ as a linear combination of elements of a basis of $N_{\mathbb{R}}$ constructed from extending a basis of $|\lk_\Sigma(\rho)|$ by $\rho$. If $\mu$ is not a multiple of $\rho$, both components are nonzero. The former (which gives a point inside a maximal cone of $\lk_\Sigma(\rho)$) is nonzero due to $\mu$ not being a multiple of $\rho$ and the latter nonzero since $\mu \notin |\lk_\Sigma(\rho)|$. This would mean that $\mu$ is the interior point of a maximal cone of $\Sigma$ containing $\rho$, which would contradict our point about intersection on faces. Thus, we have that $\mu = \rho$. In particular, this implies that $\rho \in \lk_\Sigma(\alpha)$ and $\alpha \in \lk_\Sigma(\rho)$. This also means that $\beta \notin \lk_\Sigma(\rho)$ implies that $|\lk_\Sigma(\beta)| = |\lk_\Sigma(\rho)|$. \\
		
		Now consider the case where $|\lk_\Sigma(\alpha)| = |\lk_\Sigma(\rho)|$. Since $\alpha \notin \lk_\Sigma(\alpha)$ and cones must intersect on faces, we have that $\alpha \notin |\lk_\Sigma(\alpha)|$. The argument above implies that $\alpha$ is on the opposite side of $|\lk_\Sigma(\rho)|$ from $\rho$ since $\alpha \ne \rho$. In addition, the same argument with $\alpha$ replacing $\rho$ implies that $\alpha$ itself is the only ray of $\Sigma$ on the opposite side of $|\lk_\Sigma(\rho)|$ from $\rho$. We note that such a ray must exist by completeness of $\Sigma$. \\
		
		Repeating this analysis for each ray $\alpha \in \lk_\Sigma(\rho)$ implies that the cones of $\Sigma$ have the \emph{combinatorial} structure of a cross polytope. In fact, we can actually go further and show that $\Sigma$ consists of a collection of antipodal rays. Note that the dimension 1 case is trivial since we are forced to have $\mathbb{P}^1$ if the support is a 1-dimensional vector space. Now suppose that $\dim \Sigma = 2$. Take a ray $\rho \in \Sigma(1)$. If $-D_\rho|_{D_\rho} = 0$, the realization $|\lk_\Sigma(\rho)| \subset N_{\mathbb{R}}$ is a codimension 1 linear subspace. Since $\dim N_{\mathbb{R}} = 2$ in this case, we have that $|\lk_\Sigma(\rho)| \subset N_{\mathbb{R}}$ actually forms a line. This means that $\lk_\Sigma(\rho)$ consists of a pair of antipodal rays $\alpha$ and $-\alpha$ in $N_{\mathbb{R}}$. If we start with $\alpha$ in place of $\rho$, the fact that $-D_\alpha|_{D_\alpha}$ implies that $|\lk_\Sigma(\alpha)| \subset N_{\mathbb{R}}$ is the line formed by the rays $\rho$ and $-\rho$. \\

		If $\dim \Sigma = 3$, we can make use of the dimension 2 case that yields a recursion for analyzing higher dimensions in general. Fix a ray $\rho \in \Sigma(1)$. Since $-D_\rho|_{D_\rho} = 0$, we have that $|\lk_\Sigma(\rho)| \subset N_{\mathbb{R}}$ is a codimension 1 linear subspace (i.e. a 2-plane). Consider a ray $\alpha \in \Sigma(1)$ with $\alpha \ne \rho$ and $|\lk_\Sigma(\alpha)| \ne |\lk_\Sigma(\rho)|$. Then, we have that $|\lk_\Sigma(\alpha)| \subset N_{\mathbb{R}}$ is a codimension 1 linear subspace generated by a codimension 1 linear subspace of $|\lk_\Sigma(\rho)|$ and $\rho$. We note that $\lk_\Sigma(\alpha)$ contains a ray $\beta$ that is on the opposite side of $|\lk_\Sigma(\rho)|$ from $\rho$. \\
		
		As observed above, we have that $|\lk_\Sigma(\beta)| = |\lk_\Sigma(\rho)|$. Note that the containment part of the equality holds since the intersection of unions is the union of intersections and cones of a fan intersect on faces (Definition 3.1.2 on p. 106 of \cite{CLS}). Since $|\lk_\Sigma(\alpha)| \ne |\lk_\Sigma(\rho)|$, we note that $|\lk_\Sigma(\beta)| \cap |\lk_\Sigma(\alpha)| = |\lk_\Sigma(\beta) \cap \lk_\Sigma(\alpha)|$ is a codimension 1 linear subspace of $|\lk_\Sigma(\alpha)|$. The dimension 2 case then implies that $\beta = -\rho$. We can repeat this set of arguments with each of the rays $\alpha \in \lk_\Sigma(\rho)$ replacing $\rho$ itself to see that the cones of $\Sigma$ are formed by picking one of the antipodal pairs in 3 axes like a cross polytope. \\

		We can repeat this argument for higher dimensions. Suppose that having $|\lk_\Gamma(\rho)|$ is a codimension 1 vector subspace of the support of a complete fan $\Gamma$ implies that the unique ray of $\Gamma$ on the opposite side of $|\lk_\Sigma(\rho)|$ from $\rho$ in the support $\Gamma$ itself is $-\rho$ in the support of $\Gamma$ if every conormal bundle of $\Gamma$ is trivial (i.e. $|\lk_\Gamma(\delta)| \subset |\Gamma|$ a codimension 1 vector subspace for all rays $\delta \in \Gamma(1)$). Now consider a complete fan $\Sigma$ such that $-D_\rho|_{D_\rho} = 0$ for all rays $\rho \in \Sigma(1)$. Fix a ray $\rho \in \Sigma(1)$ and take a ray $\alpha \ne \rho$ with $|\lk_\Sigma(\alpha)| \ne |\lk_\Sigma(\rho)|$. We note that $|\lk_{\lk_\Sigma(\alpha)}(\delta)| = |\lk_\Sigma(\alpha) \cap \lk_\Sigma(\delta)| = |\lk_\Sigma(\alpha)| \cap |\lk_\Sigma(\delta)| \subset |\lk_\Sigma(\alpha)|$ is a codimension 1 vector subspace for each $\delta \in \lk_\Sigma(\alpha)$. As mentioned in the $\dim \Sigma = 3$ case, the second equality uses the fact that cones of a fan intersect on faces and that these faces still lie on the fan (Definition 3.1.2 on p. 106 of \cite{CLS}). Then, $|\lk_\Sigma(\alpha)| \subset N_{\mathbb{R}}$ is a codimension 1 vector subspace and $\lk_\Sigma(\alpha)$ must contain a ray $\beta$ that is on the opposite side of $|\lk_\Sigma(\rho)|$ from $\rho$. Since $\beta \notin \lk_\Sigma(\rho)$, this implies that $|\lk_\Sigma(\beta)| = |\lk_\Sigma(\rho)|$. Since $|\lk_\Sigma(\alpha)| \ne |\lk_\Sigma(\rho)|$, this implies that $|\lk_\Sigma(\beta)| \cap |\lk_\Sigma(\alpha)| = |\lk_\Sigma(\beta) \cap \lk_\Sigma(\alpha)|$ is a codimension 1 vector subspace of $|\lk_\Sigma(\alpha)|$. Within $\lk_\Sigma(\alpha)$, the ray $\beta$ is the unique ray of $\lk_\Sigma(\alpha)$ on the opposite side of $|\lk_\Sigma(\alpha)| \cap |\lk_\Sigma(\rho)|$ from $\beta$. The induction assumption then implies that $\beta = -\rho$. \\

	\end{proof}
	
	The reasoning in Proposition \ref{vanconormcross} implies that the vanishing conormal bundle restrictions in Corollary \ref{mixvol0toconorm0} yield a cross polytope-like structure involving suspensions. We will focus on the combinatorial structure since this is what we use in subsequent structural results and we are only considering one conormal bundle restriction center at a time. \\

	\begin{prop} \textbf{(Vanishing conormal bundle restrictions and suspension-like structures) \\}
		\label{conormrestrsusp}
		
		Suppose that $-D_{j_r}|_{D_{j_1} \cap \cdots \cap D_{j_m}} = 0$ and that there is another ray $\alpha \in \lk_\Sigma(\sigma_{j_1, \ldots, \widehat{j_r}, \ldots, j_m})$ \emph{not} equal to $\rho_{j_r}$ such that $-D_\alpha|_{D_\alpha \cap D_{j_1} \cap \cdots \cap \widehat{D_{j_r}} \cap \cdots D_{j_m}} = 0$. Note that this latter condition is equivalent to $|\lk_\Sigma(\alpha, \sigma_{j_1, \ldots, \widehat{j_r}, \ldots, j_m})| \subset N(\sigma_{j_1, \ldots, \widehat{j_r}, \ldots, j_m})_{\mathbb{R}}$ being a codimension 1 vector subspace. There is at most one such ray $\beta \in \lk_\Sigma(\sigma_{j_1, \ldots, \widehat{j_r}, \ldots, j_m})$ \emph{not} in $\lk_\Sigma(\rho_{j_r})$ (equivalently not in $\lk_\Sigma(\sigma_{j_1, \ldots, j_m})$ under our conditions) such that $-D_\beta|_{D_\beta \cap D_{j_1} \cap \cdots \cap \widehat{D_{j_r}} \cap \cdots D_{j_m}} = 0$ and its link inside $\lk_\Sigma(\sigma_{j_1, \ldots, \widehat{j_r}, \ldots, j_m})$ must have support $|\lk_\Sigma(\beta, \sigma_{j_1, \ldots, \widehat{j_r}, \ldots, j_m})| = |\lk_\Sigma(\sigma_{j_1, \ldots, j_m})|$.

	\end{prop}
	
	\begin{proof}

		\vspace{3mm}
		
		The existence of a ray $\alpha \ne \rho_{j_r}$ such that $-D_\alpha|_{D_\alpha \cap D_{j_1} \cap \cdots \cap \widehat{D_{j_r}} \cap \cdots D_{j_m}} = 0$ is equivalent to $|\lk_\Sigma(\alpha, \sigma_{j_1, \ldots, \widehat{j_r}, \ldots, j_m})| \subset N(\sigma_{j_1, \ldots, j_m})_{\mathbb{R}}$ being a codimension 1 vector subspace. Then, we have that $\alpha \in \lk_\Sigma(\sigma_{j_1, \ldots, \widehat{j_r}, \ldots, j_m})$. If $|\lk_\Sigma(\alpha, \sigma_{j_1, \ldots, \widehat{j_r}, \ldots, j_m})| \ne |\lk_\Sigma(\sigma_{j_1, \ldots, j_m})|$, then $|\lk_\Sigma(\alpha, \sigma_{j_1, \ldots, \widehat{j_r}, \ldots, j_m})|$ contains a ray $x$ on the same side of $|\lk_\Sigma(\sigma_{j_1, \ldots, j_m})|$ as $\rho_{j_r}$. This would force $x = \rho_{j_r}$ and $\rho_{j_r} \in \lk_\Sigma(\alpha, \sigma_{j_1, \ldots, \widehat{j_r}, \ldots, j_m})$. Equivalently, we have that $\alpha \in \lk_\Sigma(\sigma_{j_1, \ldots, j_m}) = \lk_\Sigma(\rho_{j_r}) \cap \lk_\Sigma(\sigma_{j_1, \ldots, \widehat{j_r}, \ldots, j_m})$ (using flagness of $\Sigma$ for the second expression). \\
		
		If $|\lk_\Sigma(\alpha, \sigma_{j_1, \ldots, \widehat{j_r}, \ldots, j_m})| = |\lk_\Sigma(\sigma_{j_1, \ldots, j_m})|$, then we actually have \[ \lk_\Sigma(\alpha, \sigma_{j_1, \ldots, \widehat{j_r}, \ldots, j_m}) = \lk_\Sigma(\sigma_{j_1, \ldots, j_m}) \] since cones of a fan intersect on boundaries (think about 2-cones formed by adjacent rays and wall crossing arguments on p. 265 -- 266 of \cite{CLS}). In particular, this would imply that there are no rays of $\lk_\Sigma(\alpha, \sigma_{j_1, \ldots, \widehat{j_r}, \ldots, j_m})$ on the opposite side of $\lk_\Sigma(\sigma_{j_1, \ldots, j_m})$ from $\rho_{j_r}$ other than $\alpha$. This indicates the uniqueness of a choice of $\alpha$ that is \emph{not} adjacent to $\rho_{j_r}$. \\

		Suppose that there is a ray $\beta \notin \lk_\Sigma(\rho_{j_r})$ with $\beta \in \lk_\Sigma(\sigma_{j_1, \ldots, \widehat{j_r}, \ldots j_m})$, $\beta \ne \rho_{j_r}$, and $-D_\beta|_{D_\beta \cap D_{j_1} \cap \cdots \cap \widehat{D_{j_r}} \cap \cdots \cap D_{j_m}} = 0$ that is adjacent to a ray $\alpha \in \lk_\Sigma(\rho_{j_r})$ such that $-D_\alpha|_{D_\alpha \cap D_{j_1} \cap \cdots \cap \widehat{D_{j_r}} \cap \cdots \cap D_{j_m}} = 0$. This would mean that $|\lk_\Sigma(\beta, \sigma_{j_1, \ldots, \widehat{j_r}, \ldots, j_m})| = |\lk_\Sigma(\sigma_{j_1, \ldots, j_m})|$. \\
	\end{proof}

	\begin{rem} \textbf{(Some comments on antipodal points and possible connections to algebraic structures) \\}
		The reasoning in the proof of Proposition \ref{vanconormcross} seems to indicate that we also have some kind of antipodal ray structure. However, we need to be careful since we only look at a single ray at the ``center'' of the conormal bundle restriction at a time and do not make any adjacency assumptions. The cases where Proposition \ref{vanconormcross} does not apply directly start to occur when $\dim \Sigma = 3$. If $-D_{j_1}|_{D_{j_1} \cap D_{j_2}} = 0$, then $|\lk_\Sigma(\sigma_{j_1, j_2})| \subset N(\rho_{j_1})_{\mathbb{R}}$ is a codimension 1 vector subspace. Since $\dim N(\rho_{j_1})_{\mathbb{R}} = 2$, this implies that $|\lk_\Sigma(\sigma_{j_1, j_2})|$ is a line. In addition, a ray $\alpha \in \lk_\Sigma(\sigma_{j_1, j_2})$ such that $-D_\alpha|_{D_\alpha \cap D_{j_2}} = 0$ would yield a link $\lk_\Sigma(\alpha, \rho_{j_2})$ whose support is a line that is formed by $\rho_{j_1}$ and a ray whose ``residue'' in $\dim N(\rho_{j_1})_{\mathbb{R}}$ is $-\rho_{j_1}$. \\
		
		More generally, the ray $\beta$ must be on the opposite side of $|\lk_\Sigma(\sigma_{j_1, \ldots, j_m})|$ from $\rho_{j_r}$. In order to carry out an induction similar to Proposition \ref{vanconormcross}, we need $\beta$ to be adjacent to a ray $\alpha \in \lk_\Sigma(\sigma_{j_1, \ldots, j_m})$ such that $-D_\alpha|_{D_\alpha \cap D_{j_1} \cap \cdots \cap \widehat{D_{j_r}} \cap \cdots \cap D_{j_m} } = 0$ to provide a ``flat ambient space'' to work in where it is known that we are forced to have an antipodal point structure. \\ 
		
		Instead, we may try considering possible connections to algebraic structures and common rays inducing a flat wall crossing in Corollary \ref{suspalg} and Remark \ref{suspspecnon}. \\

	\end{rem}
	
	\color{black}
	
	We will combine Proposition \ref{conormrestrsusp} with Corollary \ref{mixvol0toconorm0} to look at the structure of special and non-special rays. \\

	\begin{cor} \textbf{(Suspension structures and special/non-special rays) \\} \label{suspspecnon}
		\vspace{-3mm}
		\begin{enumerate}
			
			\item Fix a ray $\gamma \in \lk_\Sigma(\rho)$. All but possibly one of the rays $\beta \ne \rho$ such $\gamma \in \lk_\Sigma(\beta)$ and is a special ray for $\beta$ are in $\lk_\Sigma(\rho)$. \\

			\item Fix a ray $\rho \in \Sigma(1)$ and a non-special ray $\delta$ of $\lk_\Sigma(\rho)$. All but possibly one of the non-special rays $\mu \in \lk_\Sigma(\rho)$ with $\mu \ne \delta$ are in $\lk_\Sigma(\delta)$. \\

			\item The statements of Part 1 and Part 2 hold with $\rho$ replaced by $\sigma_{i_1, \ldots, i_p}$ and replacing being special/non-special with respect to a particular ray with being special/non-special with respect to $\rho_{i_j}$ for all $j \in A$ for some given $A \subset [p]$. \\
		\end{enumerate}
	\end{cor}
	
	\begin{proof}
		\begin{enumerate}
			\item By Corollary \ref{mixvol0toconorm0}, we have that $-D_\rho|_{D_\gamma \cap D_\rho} = 0$ if $\gamma \in \lk_\Sigma(\rho)$ is a special ray of $\rho$. Proposition \ref{conormrestrsusp} then implies that all but possibly one ray $\beta$ such that $-D_\beta|_{D_\gamma \cap D_\beta} = 0$ (e.g. a ray $\beta$ such that $\gamma \in \lk_\Sigma(\beta)$ and is special with respect to $\beta$) satisfy $\beta \in \lk_\Sigma(\rho)$. \\
			
			\item We start by restricting to non-special rays of $\lk_\Sigma(\rho)$ attached to a fixed special ray $\gamma \in \lk_\Sigma(\rho)$. We will take a special ray $\gamma \in \lk_\Sigma(\rho) \cap \lk_\Sigma(\delta)$ of $\lk_\Sigma(\rho)$. By Corollary \ref{mixvol0toconorm0}, we have that $-D_\delta|_{D_\delta \cap D_\gamma \cap D_\rho} = 0$. Then, Proposition \ref{conormrestrsusp} implies that all but possibly one ray $\mu \in \lk_\Sigma(\gamma) \cap \lk_\Sigma(\rho)$ such that $-D_\mu|_{D_\mu \cap D_\gamma \cap D_\mu} = 0$ (e.g. a non-special ray of $\lk_\Sigma(\rho)$ that is adjacent to $\gamma$) such that $\mu \ne \delta$ is in $\lk_\Sigma(\delta)$. \\
			
			\noindent Suppose that the codimension of the divisor polytope associated to $-D_\rho|_{D_\rho}$ is greater than or equal to 2. Equivalently, we are assuming that any maximal cone of $\lk_\Sigma(\rho)$ contains 2 or more special rays. Now consider a wall crossing between maximal cones of $\lk_\Sigma(\rho)$ using $\delta$ as an off-wall ray (i.e. the ray that is replaced with a new ray). Let $\sigma \in \lk_\Sigma(\rho)$ be the starting maximal cone. The proof of Corollary \ref{mixvol0toconorm0} indicates that any replacement $\delta'$ of $\delta$ is again a non-special ray of $\lk_\Sigma(\rho)$. Since $\delta'$ and $\delta$ are \emph{not} adjacent (see Part 3 of proof of Corollary \ref{mixvol0toconorm0}), $\delta'$ is the unique ray that is a non-special ray of $\lk_\Sigma(\rho)$ that is adjacent to $\gamma$ and \emph{not} adjacent to $\delta$ by Proposition \ref{conormrestrsusp}. Since all of the special rays of $\lk_\Sigma(\rho)$ in the starting maximal cone $\sigma$ stay on the wall, the same statement holds with $\gamma$ replaced by \emph{any} special ray $\widetilde{\gamma} \ne \gamma$ of $\lk_\Sigma(\rho)$ lying in $\sigma$. Since our choice of maximal cone $\sigma \in \lk_\Sigma(\rho)$ is arbitrary, the same statement actually holds for \emph{any} special ray $\widetilde{\gamma}$ of $\lk_\Sigma(\rho)$ adjacent to $\gamma$ replacing $\gamma$. \\
			
			\noindent Consider a wall crossing starting with a maximal cone $\sigma \in \lk_\Sigma(\rho)$ containing $\delta$ where a special ray $\gamma \in \lk_\Sigma(\rho)$ used as the starting off-wall ray. Label the special rays of $\sigma$ as $\gamma_1, \ldots, \gamma_r$ (with $r \ge 2$ since we assumed that the codimension of $P_{-D_\rho}^{D_\rho}$ is greater than or equal to 2). Suppose that the special ray used as the off-wall ray is $\gamma_1$. Then, it is sent to another special ray $\gamma'_1$ of $\lk_\Sigma(\rho)$. Note that $\delta$ and $\gamma_i$ for $2 \le i \le r$ stay on the wall. If we follow this wall crossing by a wall crossing using $\delta$ as a the starting off-wall ray to be replaced, the fact that $\gamma_2, \ldots, \gamma_r$ are from the original maximal cone $\sigma$ mean that the unique special ray $\delta'$ \emph{not} in $\lk_\Sigma(\delta)$ that is adjacent to $\gamma_i$ is the same for each $2 \le i \le r$. Our previous argument regarding adjacent special rays implies that the same is true for $\gamma'_1$ as well since $\gamma'_1$ is adjacent to each of the special rays $\gamma_2, \ldots, \gamma_r$ of $\lk_\Sigma(\rho)$. These steps can be repeated for any wall crossing using a special ray of $\lk_\Sigma(\rho)$ as an off-wall ray. Note that the maximal cones of $\lk_\Sigma(\rho)$ containing $\delta$ can be traversed via wall crossings with $\delta$ on the wall. This means that the unique ray $\delta'$ \emph{not} adjacent to $\delta$ such that $\delta'$ is a special ray of $\lk_\Sigma(\rho)$ adjacent to $\gamma$ is the same for \emph{any} special ray $\gamma$ of $\lk_\Sigma(\rho)$ when the codimension of $P_{-D_\rho}^{D_\rho}$ is greater than or equal to 2. \\
			
			\noindent Now suppose that the codimension of the divisor polytope associated to $-D_\rho|_{D_\rho}$ is 1. Then, every maximal cone of $\lk_\Sigma(\rho)$ contains exactly 1 special ray. Since the special rays form a linear subspace of $|\lk_\Sigma(\rho)| \subset N_{\mathbb{R}}$ and it is 1-dimensional in this case, there is exactly 1 other special ray $-\gamma$ of \emph{any} maximal cone of $\lk_\Sigma(\rho)$. Since a wall crossing using a special ray as the initial off-wall ray replaces it with another special ray, any of $\lk_\Sigma(\rho)$ adjacent to $\gamma$ must also be adjacent to $-\gamma$ and $ \lk_\Sigma(\rho)$ is the suspension of the subfan consisting of cones generated by the non-special rays over $\gamma$ and $-\gamma$. This means that \emph{any} non-special ray $\delta$ of $\lk_\Sigma(\rho)$ is adjacent to \emph{both} $\gamma$ and $-\gamma$. In particular, being adjacent to a fixed special ray $\gamma$ or $-\gamma$ is \emph{not} an additional condition. Given a non-special ray $\delta$ of $\lk_\Sigma(\rho)$, this means that there is at most one non-special ray of $\lk_\Sigma(\rho)$ that is \emph{not} adjacent to $\delta$. \\
			
			\item This follows from the same reasoning as the proof of Part 1 and Part 2 with $\rho$ replaced by $\sigma_{i_1, \ldots, i_p}$ and noting that uniqueness of a condition implies that the existence of a case where a stronger condition holds must also imply that case is a unique example. 
		\end{enumerate}
	\end{proof}

	We can combine the discussion in Remark \ref{sign0background} and the structural results in Corollary \ref{suspspecnon} to describe a certain kind of ``cover'' of $\Sigma$ in toric varieties associated to locally convex fans whose signatures are 0. \\

	\begin{prop} \textbf{(Signature 0, link suspension structures, and special ray covers) \\} \label{sign0specialcov}
		Suppose that $\Sigma$ is locally convex and the signature $\sigma(X_\Sigma) = 0$. \\
		\begin{enumerate}
			\item For every ray $\rho \in \Sigma(1)$, the support $|\lk_\Sigma(\rho)| \subset N_{\mathbb{R}}$ comes from the maximal vector subspace contained in $|\lk_\Sigma(\rho)| \subset N_{\mathbb{R}}$ (say of dimension $r \ge 1$, generated by special rays) suspended by $d - 1 - r$ common pairs of (non-special) rays. \\
			
			\item We can look at how the links in Part 1 are connected around special rays using a cone or suspension structure. 
			\begin{enumerate}
				\item There is a ``cover'' of rays of $\Sigma$ centered around rays $\gamma$ that are special as rays in the link $\lk_\Sigma(\rho)$ of some ray $\rho$ with the ``surrounding rays'' given by such rays $\rho$. Part 1 of Corollary \ref{suspspecnon} indicates that there is a cone or a repeated suspension yielding such $\rho$. \\
				
				\item  If two ``mutually non-special'' adjacent rays $\rho_{i_1}$ and $\rho_{i_2}$ are \emph{not} in the same special ray block, then they do \emph{not} share a special ray. \\
				
				\item Coinciding ``centers'' of special ray covers involve different rays that are special with respect to a common set of rays (which may involve a partial overlap). \\
			\end{enumerate}

		\end{enumerate}
	\end{prop}
	
	\begin{proof}
		\begin{enumerate}
			\item This follows from the fact that $D_\rho^d = D_\rho \cdot (-D_\rho)^{d - 1} = 0$ for all $\rho \in \Sigma(1)$ and Part 2 of Corollary \ref{suspspecnon}. \\
			
			\item This is an application of Part 1 of Corollary \ref{suspspecnon}. 
		\end{enumerate}
	\end{proof}
	
	There are analogues of the results above for links over $p$-cones. \\
	
	\begin{prop} \textbf{(Special ray adjacency and extension of signature 0 special ray cover to $p$-cones) \\} \label{specialcovpconext}
		\begin{enumerate}
			\item The set of special rays of $\lk_\Sigma(\sigma_{i_1, \ldots, i_p})$ with respect to $\rho_{i_j}$ contains the set of special rays of $\lk_\Sigma(\rho_{i_j})$ with respect to $\rho_{i_j}$. \\
			
			\item If $\Sigma$ is locally convex and the signature $\sigma(X_\Sigma) = 0$, Part 1 implies that the special ray cover from Proposition \ref{sign0specialcov} applies with $\Sigma$ replaced by $\lk_\Sigma(\sigma_{i_1, \ldots, i_p})$ and replacing $N_{\mathbb{R}}$ by $N(\sigma_{i_1, \ldots, i_p})_{\mathbb{R}}$ (for some choice of codimension $p$ vector subspace of $N_{\mathbb{R}}$ \emph{not} containing $\rho_{i_j}$ for any $j \in [p]$). \\
		\end{enumerate}
	\end{prop}
	
	\begin{proof}
		\begin{enumerate}
			\item Going back to Proposition \ref{altpdperp}, the set of special rays of $\lk_\Sigma(\sigma_{i_1, \ldots, i_p})$ with respect to $\rho_{i_j}$ comes from the intersection of spans of a subcollection of the walls intersected to obtain the set of special rays of $\lk_\Sigma(\rho_{i_j})$ with respect to $\rho_{i_j}$. \\
			
			\item Since $D_{i_j}^d = D_{i_j} \cdot (-D_{i_j})^{d - 1} = 0$ for all $j \in [p]$, this is a reflection of Part 3 of Corollary \ref{suspspecnon} applying to $p$ variables and the structural property it is based on (Proposition \ref{conormrestrsusp}) coming from $p$-cone properties. 
		\end{enumerate}
	\end{proof}

	So far, the structural information related to possible fan structures yielding a combination of linear subspaces and suspensions made use of the vanishing monomials from Remark \ref{sign0background} that are supported on a single variable. The information coming from monomials supported on $2 \le p \le \frac{d}{2}$ (``globally'' of even degrees and ``locally '' of odd degrees on each variable) will yield information on linear dependence and containment properties among maximal linear subspaces contained in realizations of links over designated (choices of) quotient spaces $N(\omega)_{\mathbb{R}}$ of $N_{\mathbb{R}}$ (e.g. see Remark \ref{flatambvar}). In order to do this we will use the fact the conormal bundle restrictions are nef divisors and move to the general setting of polytopes and odd tuple exponents yielding mixed volumes equal to 0 before connecting back to the context of conormal bundle restrictions above. \\
	
	\section{Exponents, (non)zero mixed volumes, and intersection patterns of maximal linear subspaces} \label{expmixedvol}

	Recall that we are interested in (simplicial) locally convex fans $\Sigma$ yielding signature 0 toric varieties $X_\Sigma$ (Remark \ref{sign0background}). This depends on showing that all top degree monomials supported on at most $\frac{d}{2}$ variables with even exponents on each ray divisor are 0 in the Chow ring (see Lemma 3.1 on p. 262 and the proof of Theorem 1.2(i) below Lemma 3.2 on p. 264 of \cite{LR}). We will work on the restrictions $D_{i_1} \cap \cdots \cap D_{i_p}$ in order to make use of the local convexity property, which implies global generation of the associated line bundles. On these restrictions, this means top degree exponents with \emph{odd} exponents on the basepoint free/nef conormal bundle restrictions $-D_{i_j}|_{D_{i_1} \cap \cdots \cap D_{i_p}}$ are equal to 0 (equation (5) on p. 264 of \cite{LR}). So far, we have mainly focused on $D_\rho^d = D_\rho \cdot (-D_\rho)^{d - 1} = 0$ for all $\rho \in \Sigma(1)$. \\
	
	Taking a closer look at the $p = 2$ and $p = \frac{d}{2}$ cases gives some information on suspension structures that is also relevant for higher $p$. In order to do this, we will treat the (non)vanishing of mixed volumes from the perspective of the \emph{exponents}. We start by covering related background in Section \ref{expmvind}. Note that the $p = 3$ case behaves similarly to the $p = 2$ case. There are is also a ``dichotomy'' in the behavior around $p$-cones for higher $p$ similar to implications of $p = 2$ identities. This relates to links whose realizations in $N_{\mathbb{R}}$ contain a positive-dimensional linear subspace and cases where every ray of the link is non-special with respect to some ray (which forces a suspension-like structure). These ideas are covered in Section \ref{p2higherpext}. \\
	
	Finally, we will also consider containments of linear subspaces of realizations from the $p = \frac{d}{2}$ case, linear dependence conditions from equality cases of submodularity inequalities, and how higher $p$ generally interpolates between them in Section \ref{higherpid}. This involves a more direct use of the identities coming from higher values of $1 \le p \le \frac{d}{2}$. It gives linear dependence and containment information on the linear subspaces we suspend around in Section \ref{wallfan}. Specific results involved are summarized in Part 3 and Part 4 of Theorem \ref{sig0fanstr} and described in more detail in the subsections mentioned above. \\
	
	\subsection{Exponents as indicator functions for mixed volumes} 
	\label{expmvind} 
	
	We can convert this directly into questions about polytopes via the conversion between products of basepoint free/nef divisors and mixed volumes of the associated polytopes (p. 1116 and p. 68 -- 69 of \cite{Ful}). Since we are working with (locally convex) simplicial fans, we will use equivalences between the globally generated and basepoint free conditions (Corollary 6.0.25 on p. 258 of \cite{CLS}) along with local convexity for those involving nef divisors for toric varieties associated to full-dimensional fans with convex support (Theorem 6.3.12 and Proposition 6.3.11 on p. 291 of \cite{CLS}). Thus, we will look for conditions on the starting polytopes which imply that all mixed volumes of a given degree (thought of as the top degree) with odd exponents on each polytope are equal to 0. \\
	
	We will start by studying this property in detail for the $p = 2$ case. This involves the behavior of the nonzero mixed volume region in general before focusing on odd exponents relevant to signature 0 toric varieties with locally convex fans. \\
	
	Here are some of the main tools used: \\

	\begin{thm} (Theorem 5.1.8 on p. 283 of \cite{Sch}) \label{nonzeromixvol} \\
		
		Given convex polytopes $K_1, \ldots, K_N$ lying in $\mathbb{R}^N$, the following are equivalent: \\
		
		\begin{enumerate}
			\item $V(K_1, \ldots, K_N) > 0$ \\
			
			\item There are line segments $S_i \subset K_i$ ($i = 1, \ldots, N$) with linearly independent directions. \\
			
			\item $\dim (K_{i_1} + \ldots + K_{i_k}) \ge k$ for each choice of indices $1 \le i_1 < \cdots < i_k \le N$ and for all $k \in \{ 1, \ldots, N \}$. \\
		\end{enumerate}

		Note that these conditions implicitly assume that $\dim K_i \ge 1$ for all $1 \le i \le N$, which is actually necessary to have $V(K_1, \ldots, K_N) > 0$ by multilinearity and nonnegativity properties of the mixed volume. A concrete interpretation of mixed volumes involving Minkowski sums of faces of dimensions given by exponents/multiplicities from distinct polytopes involved is given in Schneider's summation formula (Theorem 15.2.2 and discussion following it on p. 263 -- 264 of \cite{HRGZ}). \\
		
	\end{thm}

	\begin{thm} (Version 2 of Theorem 5.1.8 on p. 283 of \cite{Sch}) \label{zerovolcrit} \\
		
		Given convex polytopes $K_1, \ldots, K_N$ lying in $\mathbb{R}^N$ ($N \ge 1$), the following are equivalent: \\
		
		\begin{enumerate}
			\item $V(K_1, \ldots, K_N) = 0$ \\

			\item We either have that some $K_i$ is a point or $\dim K_i \ge 1$ for each $1 \le i \le N$ and all line segments $S_i \subset K_i$ ($i = 1, \ldots, N$) have linearly dependent directions. \\
			
			\item We either have that some $K_i$ is a point or $\dim K_i \ge 1$ for each $1 \le i \le N$ and there is a choice of indices $1 \le i_1 < \cdots < i_k \le N$ for some $k \in \{ 1, \ldots, N \}$ such that $\dim (K_{i_1} + \ldots + K_{i_k}) < k$. \\
		\end{enumerate}
		
		The case where $K_i$ is a point for some $i$ actually shows up for some relevant examples where the conormal bundle restrictions arising from the rays is actually 0 such as self-products of $\mathbb{P}^1$ and cross polytopes. \\
		
	\end{thm}

	\begin{thm} (Exponent version of Theorem 5.1.8 on p. 283 of \cite{Sch}) \label{expmixvol} \\
		Fix a collection of \emph{distinct} polytopes $Q_1, \ldots, Q_p$ of dimension $\ge 1$. Then, the mixed volume $V(Q_1, a_1; \ldots ; Q_p, a_p)$ with exponents $a_k \ge 0$ assigned to $Q_k$ for each $k \in [p]$ is nonzero if and only if $a_{i_1} + \ldots + a_{i_\ell} \le \dim(Q_{i_1} + \ldots + Q_{i_\ell})$ for all subsets $\{ i_1, \ldots, i_\ell \} \subset [p]$. \\
	\end{thm}
	
	\begin{proof}
		Focusing on the role of multiplicities in Part 3 of Theorem \ref{nonzeromixvol}, the variable $k$ records the sum of the number of times each polytope is used in the Minkowski sum. To study the effect of repeats in polytopes used in mixed volumes, we note that $\dim(c_1 P_1 + \ldots + c_r P_r) = \dim(P_1 + \ldots + P_r)$ if $c_i > 0$ and $P_i$ are distinct polytopes. Since the total number of polytopes used in \emph{any} such Minkowski sum involving a collection of distinct polytopes $Q_{i_1}, \ldots, Q_{i_\ell}$ (including the maximal value $a_{i_1} + \ldots + a_{i_\ell}$) is bounded above by $\dim(Q_{i_1} + \ldots + Q_{i_\ell})$ in Part 3 of Theorem \ref{nonzeromixvol}, the nonzero mixed volume condition is equivalent to $a_{i_1} + \ldots + a_{i_\ell} \le \dim(Q_{i_1} + \ldots + Q_{i_\ell})$ for any subset $\{ i_1, \ldots, i_\ell \} \subset [N]$.  \\
	\end{proof}
	
	\begin{rem} \textbf{(Dimensions of Minkowski sums of polytopes and generalized permutohedra) \\} \label{permnot}
		\vspace{-3mm}
		
		\begin{enumerate}
			\item \textbf{(Generalized permutohedra and ``nonzero mixed volume inequalities'') \\}
			\vspace{0.5mm}
			When we intersect with a hyperplane $a_1 + \ldots + a_p = C$, the defining equation and inequalities look similar to those of generalized permutohedra (e.g. see p. 1041 of \cite{Pos}). In fact, setting \[ b_A \coloneq \dim \left( \sum_{j \in A} Q_j \right) \] for some initial collection of polytopes $Q_1, \ldots, Q_N$ yields a collection of parameters respecting the \textbf{submodularity inequalities} $b_S + b_T \ge b_{S \cup T} + b_{S \cap T}$ (e.g. see Theorem 3.9 on p. 1993 of \cite{CL}) and we \emph{do} obtain generalized permutohedra (e.g. compare with linear polymatroids on p. 30 of \cite{F1}). \\
			
			\noindent That being said, we will often call the defining inequalities \textbf{nonzero mixed volume inequalities} since our main motivation is to use the exponents as indicator functions for (non)vanishing of mixed volumes of a given collection of polytopes. A specialization of the submodularity inequality we will often end up using is that $b_A - b_{A \setminus k} \ge b_C - b_{C \setminus k}$ if $C \subset A$ and $k \in C$ (see Proposition \ref{minbaselift}). \\
			
			\item \textbf{(Assumptions relevant to our setting) \\}
			\vspace{0.5mm}
			In our setting, we will take the $b_A$ to be positive integers. We will often assume that $b_A \ge |A|$ and $R \subset S \Longrightarrow b_R \le b_S$. The first condition is necessary in order for odd tuples of exponents inducing nonzero mixed volumes to exist. Also, we note that the second condition always holds in the setting of the problem where \[ b_A \coloneq \dim \left( \sum_{j \in A} Q_j \right) \] for some initial collection of polytopes $Q_1, \ldots, Q_N$ (taken to be divisor polytopes of conormal bundle restrictions in our setting).  \\ 
			
			\item \textbf{(Some comments on notation) \\}
			\vspace{0.5mm}
			In general, parts that apply to arbitrary generalized permutohedra will use $N$ in place of $p$ in the locally convex fan setting. In addition, we will usually put $b_A = \dim \left( \sum_{j \in A} Q_j \right)$ for an initial collection of polytopes $Q_1, \ldots, Q_N$ although a large part of the material applies to arbitrary generalized permutohedra. \\ 
			
		\end{enumerate}

	\end{rem}

	We will list some properties of the objects in Remark \ref{permnot} that may be relevant in following subsections. \\
	
	\begin{rem} \textbf{(Defining inequalities and other properties of generalized permutohedra) \\} \label{initmvgenperm}
		\vspace{-3mm}
		
		\begin{enumerate}
			
			\item \textbf{(Initial conditions and effects on defining inequalities)} \\
			Since we will work with monomials of a fixed degree, we will often intersect the nonzero mixed volume inequalities for the exponents with a hyperplane of the form $a_1 + \ldots + a_p = C$ for a constant $C$ (e.g. Proposition \ref{oddalgoutput}). In our setting, we will usually have $C \ge \dim(Q_1 + \ldots + Q_p)$ since we will mainly study polytopes associated to basepoint-free/nef divisors on toric varieties. However, we will also consider smaller values of $C$ while studying recursive properties of the intersection of a hyperplane of this form with nonzero mixed volume inequalities. We will consider the effect of the intersection of the nonzero mixed volume region with such hyperplanes on the defining inequalities and regions where they are nontrivial. \\
			
			For example, the defining inequalities for the intersection of the nonzero mixed volume region with the hyperplane $a_1 + \ldots + a_p = \dim(Q_1 + \ldots + Q_p)$ is \[ \dim(Q_1 + \ldots + Q_p) - \dim \left( \sum_{j \in [p] \setminus J} Q_j \right) \le \sum_{j \in J} a_j \le \dim \left( \sum_{j \in J} Q_j \right) \] for all subsets $J \subset [p]$. On the hyperplane $a_1 + \ldots + a_p = \dim(Q_1 + \ldots + Q_p)$, the inequality above for $J$ and the inequality \[ \dim(Q_1 + \ldots + Q_p) - \dim \left( \sum_{j \in J} Q_j \right) \le \sum_{j \in [p] \setminus J} a_j \le \dim \left( \sum_{j \in [p] \setminus J} Q_j \right) \] for $[p] \setminus J$ are equivalent to each other. Set $J = \{ k \}$. This yields \[ \dim(Q_1 + \ldots + Q_p) - \dim(Q_1 + \ldots + \widehat{Q_k} + \ldots + Q_p) \le a_k \le \dim Q_k, \] which is equivalent to \[ \dim(Q_1 + \ldots + Q_p) - \dim Q_k \le a_1 + \ldots + \widehat{a_k} + \ldots + a_p \le \dim(Q_1 + \ldots + \widehat{Q_k} + \ldots + Q_p) \] from $J = [p] \setminus k$ on the hyperplane $a_1 + \ldots + a_p = \dim(Q_1 + \ldots + Q_p)$. \\
			
			\item \textbf{(Special properties of dimensions and lattice points)} \\

			\noindent In special cases where $\dim(Q_{i_1} + \ldots + Q_{i_\ell}) = \dim Q_{i_1} + \ldots + \dim Q_{i_\ell}$ for all subsets $\{ i_1, \ldots, i_\ell \} \subset [p]$ (i.e. the affine spaces spanned by vertices of the $Q_{i_j}$ translate to a direct sum of vector spaces), we seem to get permutohedra (Proposition 2.5 on p. 1032 of \cite{Pos}). \\
			
			\noindent Since we are studying lattice points on generalized permutohedra, we mention that a special case related to the $p = 3$ case of the inequalities from Part 1 yields slices of prisms/boxes whose Ehrhart theory was studied recently in \cite{FM}. However, we generally need to check if the defining constants given by $\dim(Q_{i_1} + \ldots + Q_{i_\ell})$ for subsets $\{ i_1, \ldots, i_\ell \} \subset [p]$ are compatible with the submodular condition (see Theorem 3.9 on p. 1993 of \cite{CL}) required to lie in the deformation cone of the braid fan (which would be needed for a generalized permutohedron). \\

			\item \textbf{((Signed) volumes and connections to wall crossings) \\}
			
			In addition, generalized permutohedra are generally \emph{signed} Minkowski sums of simplices (Proposition 2.3 on p. 843 of \cite{ABD}) and that volume formulas from Corollary 9.4 on p. 1058 of \cite{Pos} for usual Minkowski sums of simplices generalize to arbitrary generalized permutohedra since they also apply to signed Minkowski sums (Proposition 3.2 on p. 847 of \cite{ABD}). For example, this may be an issue for ``degenerate'' settings where $\dim(Q_{i_1} + \ldots + Q_{i_\ell}) = \max(\dim Q_{i_1}, \ldots, \dim Q_{i_\ell})$ for all subsets $\{ i_1, \ldots, i_p \} \subset [p]$ when the dimensions of the polytopes are \emph{not} all equal to each other. Finally, there is a survey of some background on submodularity/deformations of the braid fan on p. 1975 -- 1977 of \cite{CL} and a wall crossing perspective on deformation cones in general is discussed in \cite{ACEP} (see Lemma 2.11 on p. 9, Proposition 2.12 on p. 9 -- 10, and Lemma 2.14 on p. 10 -- 11 of \cite{ACEP}). \\
			
		\end{enumerate}
	\end{rem}
	
	\subsection{Implications of $p = 2$ identities and analogous behavior involving $p$-cones for higher $p$} \label{p2higherpext} 
	\vspace{2mm} 
	
	Returning to the $p = 2$ case, we have the following:
	
	\begin{prop}  \textbf{(Nonzero mixed volume region and odd exponent tuples for the $p = 2$ case)} \label{p2mvodd} \\
		\vspace{-3mm}
		\begin{enumerate}
			\item \textbf{(Nonzero mixed volume regions) } 
			
			Fix polytopes $Q_1$ and $Q_2$ of dimension $\ge 1$. The exponents $a_1$ and $a_2$ for $Q_1$ and $Q_2$ that yield a nonzero mixed volume are the lattice points $(a_1, a_2)$ in the region of the box $[0, \dim Q_1] \times [0, \dim Q_2]$ on or \emph{below} the line $a_1 + a_2 = \dim(Q_1 + Q_2)$. Note that we use points on or below the line since $\dim(Q_1 + Q_2) \le \dim Q_1 + \dim Q_2$. \\
			
			\item \textbf{(Extremal cases) }
			\begin{enumerate}
				\item The ``extremal'' cases in Part 1 occur when $\dim(Q_1 + Q_2) = \max(\dim Q_1, \dim Q_2)$ and $\dim(Q_1 + Q_2) = \dim Q_1 + \dim Q_2$. The former is like a maximal possible truncation and the latter does not induce a truncation at all and the nonzero mixed volume region takes up the entire box bounded by the points $(0, 0)$, $(\dim Q_1, 0)$, $(0, \dim Q_2)$, and $(\dim Q_1, \dim Q_2)$. In addition, the intersection of the latter case with the line $a_1 + a_2 = \dim(Q_1 + Q_2)$ is a single point $(\dim Q_1, \dim Q_2)$. \\
				
				\item  Assume without loss of generality that $\dim Q_1 \ge \dim Q_2$. Suppose that $0 \in \Int(Q_1) \cap \Int(Q_2)$. The maximal truncation/minimal $\dim(Q_1 + Q_2)$ occurs when $\Span(Q_2) \subset \Span(Q_1)$ and the minimal truncation/maximal $\dim(Q_1 + Q_2)$ occurs when $\Span(Q_1) \cap \Span(Q_2) = 0$. The former is ``flat'' and the latter involves a direct sum of vector spaces. To remove the assumption $0 \in \Int(Q_1) \cap \Int(Q_2)$, we can replace vector spaces spanned by $Q_1$ and $Q_2$ by translations of linear subspaces given by affine hulls of $Q_1$ and $Q_2$ (see p. 3 of \cite{Zie}). \\
			\end{enumerate}
			
			\item \textbf{(Intersections with lines $a_1 + a_2 = C$ for even $C$ and odd pairs in nonzero mixed volume regions) } 
			
			Let $C \ge 2$ be an even positive integer less than or equal to $\dim(Q_1 + Q_2)$. If $C < \dim Q_1 + \dim Q_2$, then the intersection of the line $a_1 + a_2 = C$ with the nonzero mixed volume region contains a pair $(a_1, a_2)$ where $a_1$ and $a_2$ are both odd. If $C = \dim Q_1 + \dim Q_2$ (which requires $\dim(Q_1 + Q_2) = \dim Q_1 + \dim Q_2$), this is possible if and only if $\dim Q_1$ and $\dim Q_2$ are both odd. The parity assumption on $C$ is motivated by considering degree $d - 2$ terms on $D_{\rho_{i_1}} \cap D_{\rho_{i_2}}$ for even $d$ when $p = 2$. \\
		\end{enumerate}
	\end{prop}
	
	\begin{proof}
		\begin{enumerate}
			\item

			By Theorem \ref{expmixvol}, the nonzero mixed volume condition on exponents $a_1$ and $a_2$ associated to polytopes $Q_1$ and $Q_2$ is determined by the inequalities $a_1 \le \dim Q_1$, $a_2 \le \dim Q_2$, and $a_1 + a_2 \le \dim(Q_1 + Q_2)$. As part of our initial conditions, we will assume that $a_1, a_2 \ge 0$ (although we will focus on positive integers later on). Also, we can assume without loss of generality that $\dim Q_1 \ge \dim Q_2$. \\

			\noindent To study the region where the exponents yield a nonzero mixed volume, we start by putting the region inside the box formed by $0 \le a_1 \le \dim Q_1$ and $0 \le a_2 \le \dim Q_2$. The portion of the box taken up by the nonzero mixed volume region depends on where the inequality $a_1 + a_2 \le \dim(Q_1 + Q_2)$ behaves ``nontrivially''. We can rewrite the inequality as $a_1 \le \dim(Q_1 + Q_2) - a_2$. Since $a_2 \le \dim Q_2$, the right hand side is greater than or equal to $\dim(Q_1 + Q_2) - \dim Q_2$. This implies that the inequality can only impose a nontrivial condition if $a_1 \ge \dim(Q_1 + Q_2) - \dim Q_2$. If $a_1 < \dim(Q_1 + Q_2) - \dim Q_2$, we already have $a_1 < \dim(Q_1 + Q_2) - \dim Q_2 < \dim(Q_1 + Q_2) - a_2$ for any $0 \le a_2 \le \dim Q_2$. We can repeat a similar statement for $a_2$ that the inequality $a_1 + a_2 \le \dim(Q_1 + Q_2)$ only imposes a nontrivial condition on $a_2 \ge \dim(Q_1 + Q_2) - \dim Q_1$. \\ 
			
			\noindent The boundary of the nonzero mixed volume region then consists of the ``horizontal'' line $a_2 = \dim Q_2$ for $0 \le a_1 \le \dim(Q_1 + Q_2) - \dim Q_2$, a ``diagonal'' bounded by $a_1 + a_2 = \dim(Q_1 + Q_2)$ for $\dim(Q_1 + Q_2) - \dim Q_2 \le a_1 \le \dim Q_1$, and a ``vertical'' line $a_1 = \dim Q_1$ for $0 \le a_2 \le \dim(Q_1 + Q_2) - \dim Q_1$. On the diagonal $a_1 + a_2 = \dim(Q_1 + Q_2)$, the nontriviality condition $\dim(Q_1 + Q_2) - \dim Q_2 \le a_1 \le \dim Q_1$ for $a_1$ is equivalent that of $a_2$ (which is given by $\dim(Q_1 + Q_2) - \dim Q_1 \le a_2 \le \dim Q_2$). In summary, the box formed by $0 \le a_1 \le \dim Q_1$ and $0 \le a_2 \le \dim Q_2$ is truncated by the line $a_1 + a_2 = \dim(Q_1 + Q_2)$. Here, we use the word ``truncate'' since $\dim(Q_1 + Q_2) \le \dim Q_1 + \dim Q_2$. \\

			\item Since Part 2(a) is a direct consequence of the description of the graph of the nonzero mixed volume region in Part 1, we focus on Part 2(b). The ``extremal'' cases come from situations where the minimum and maximal possible values of $\dim(Q_1 + Q_2)$ for given fixed values of $\dim Q_1$ and $\dim Q_2$ are attained. We note that $\dim(Q_1 + Q_2) \ge \max(\dim Q_1, \dim Q_2)$ and $\dim(Q_1 + Q_2) \le \dim Q_1 + \dim Q_2$. The first inequality follows from the fact that $Q_1 + Q_2$ contains both $Q_1$ and $Q_2$. For the second inequality, we use the definition of the dimension of a polytope as the affine hull of its points (p. 3 of \cite{Zie}) and the assumption $0 \in \Int(Q_1) \cap \Int(Q_2)$ to apply the identity $\dim(V + W) = \dim V + \dim W - \dim(V \cap W)$ for vector spaces $V$ and $W$. Without loss of generality, we can assume that $\dim Q_1 \ge \dim Q_2$. The equality $\dim(Q_1 + Q_2) = \max(\dim Q_1, \dim Q_2) = \dim Q_1$ is attained exactly when $\Span(Q_2) \subset \Span(Q_1)$. On the other hand, the maximal value $\dim(Q_1 + Q_2) = \dim Q_1 + \dim Q_2$ occurs when $\dim(Q_1 + Q_2) = \dim Q_1 + \dim Q_2$. This is exactly when $\Span (Q_1) \cap \Span(Q_2) = 0$. \\

			\item The point is to study where the line $a_1 + a_2 = C$ enters and leaves the nonzero mixed volume region (which was described for $p = 2$ in Part 1). Note that the initial condition $C \le \dim(Q_1 + Q_2)$ was necessary in order for the line $a_1 + a_2 = C$ to have a nonempty intersection with the nonzero mixed volume region. \\
			
			\noindent Since $a_1, a_2 \ge 0$, the defining inequalities of the intersection of the region $a_1 + a_2 \le C$ with the nonzero mixed volume region are $a_1 \le \min(C, \dim Q_1)$, $a_2 \le \min(C, \dim Q_2)$, and $a_1 + a_2 \le \min(C, \dim(Q_1 + Q_2))$. The intersection of the line $a_1 + a_2 = C$ with the nonzero mixed volume region occurs exactly where the inequality $a_1 + a_2 \le \min(C, \dim(Q_1 + Q_2))$ is satisfied nontrivially (in the sense of the proof of Part 1) assuming we are working with pairs $(a_1, a_2)$ satisfying the inequalities $a_1 \le \min(C, \dim Q_1)$ and $a_2 \le \min(C, \dim Q_2)$. This nontriviality region is given by $a_1 \ge \min(C, \dim(Q_1 + Q_2)) - \min(C, \dim Q_2)$ and $a_2 \ge \min(C, \dim(Q_1 + Q_2)) - \min(C, \dim Q_1)$. \\
			
			\noindent We will now split into cases according to the size of $C$ relative to $\dim Q_1$ and $\dim Q_2$. Without loss of generality, we may assume that $\dim Q_1 \ge \dim Q_2$. \\

			\textbf{Case 1: $C < \dim Q_2$ \\}
			
			\noindent In this case, the nontriviality region is $a_1 \ge C - C = 0$ and $a_2 \ge C - C = 0$. This implies that the entire segment of the line $a_1 + a_2 = C$ with $a_1, a_2 \ge 0$ is contained in the nonzero mixed volume region. Since $C \ge 2$ and $C$ is even, the point $(C - 1, 1)$ is a valid odd pair in the intersection of $a_1 + a_2 = C$ and the nonzero mixed volume region. \\

			\textbf{Case 2: $\dim Q_2 \le C \le \dim Q_1$ \\}
			
			\noindent In this case, the nontriviality region is $a_1 \ge C - \dim Q_2$ and $a_2 \ge C - C = 0$. This means that the line $a_1 + a_2 = C$ enters the nonzero mixed volume region at $(C - \dim Q_2, \dim Q_2)$ and leaves it at $(C, 0)$. Since $\dim Q_2 \ge 1$, the point $(C - 1, 1)$ lies in between $(C - \dim Q_2, \dim Q_2)$ and $(C, 0)$ on the line $a_1 + a_2 = C$. This gives a valid odd pair since $C$ is assumed to be even and $C \ge 2$. \\

			\textbf{Case 3: $C > \dim Q_1$ \\}
			
			\noindent In this case, the nontriviality region is $a_1 \ge C - \dim Q_2$ and $a_2 \ge C - \dim Q_1$. This means that the line $a_1 + a_2 = C$ enters the nonzero mixed volume region at $(C - \dim Q_2, \dim Q_2)$ and leaves it at $(\dim Q_1, C - \dim Q_1)$. If these points are distinct from each other, the portion of the line $a_1 + a_2 = C$ bounded by these contains at least two distinct lattice points and one of these must be an odd tuple by a parity argument. The points $(C - \dim Q_2, \dim Q_2)$ and $(\dim Q_1, C - \dim Q_1)$ are distinct if and only if $\dim Q_1 > C - \dim Q_2$, which is equivalent to having $C < \dim Q_1 + \dim Q_2$. Thus, the intersection of $a_1 + a_2 = C$ and the nonzero mixed volume region contains an odd pair if $C < \dim Q_1 + \dim Q_2$. If $C = \dim Q_1 + \dim Q_2$ (which requires $\dim(Q_1 + Q_2) = \dim Q_1 + \dim Q_2$ since $C \le \dim(Q_1 + Q_2)$), the only intersection of $a_1 + a_2 = C$ with the nonzero mixed volume region is the point $(\dim Q_1, \dim Q_2)$. This is an odd pair if and only if $\dim Q_1$ and $\dim Q_2$ are both odd. \\

		\end{enumerate}
	\end{proof}

	In general, we will be interested in odd tuples of points $(a_1, \ldots, a_p)$ in the intersection of the hyperplane $a_1 + \ldots + a_N = \dim(Q_1 + \ldots + Q_p)$ and the nonzero mixed volume region. The recursive properties that we consider may also involve intersections of hyperplanes of the form $a_1 + \ldots + a_p = C$ with the nonzero mixed volume region for $C \le \dim(Q_1 + \ldots + Q_p)$. Some of these cases may involve $C \ge p$ where $C$ and $p$ have the same parity. \\
	
	In the locally convex fan context, Corollary \ref{mixvol0toconorm0} indicates that the submodularity inequality conditions are only nontrivial for $p$-tuples of adjacent rays that are pairwise non-special. We keep this in mind in the statements of Corollary \ref{p2conorm} below and Remark \ref{linkmove}. \\

	\begin{cor} \textbf{($p = 2$ case and $\sigma(X_\Sigma) = 0$) \\}
		\label{p2conorm}
		Suppose that we are in the setting of Remark \ref{sign0background}. For each pair of adjacent rays $\rho_{i_1}$ and $\rho_{i_2}$, at least one of the following must hold:
		\begin{enumerate}
			\item $-D_{i_1}|_{D_{i_1} \cap D_{i_2}} = 0$ or $-D_{i_2}|_{D_{i_1} \cap D_{i_2}} = 0$ \\
			
			\item $\dim P_{(-D_{i_1}) + (-D_{i_2})}^{D_{i_1} \cap D_{i_2}} \le d - 3$. In other words, the support of $\lk_\Sigma(\sigma_{i_1, i_2})$ in $N_{\mathbb{R}}$ is a vector subspace of codimension 2. This means that the realization of $\lk_\Sigma(\sigma_{i_1, i_2})$ in $N_{\mathbb{R}}$ contains a linear subspace of dimension $\ge 1$. For mutually non-special adjacent rays $\rho_{i_1}$ and $\rho_{i_2}$, this is like belonging to a common ``special ray block'' in Proposition \ref{sign0specialcov}. \\

			\item Suppose that $\dim P_{(-D_{i_1}) + (-D_{i_2})}^{D_{i_1} \cap D_{i_2}} = d - 2$ (i.e. no positive dimensional linear subspaces are contained in the support of $\lk_\Sigma(\sigma_{i_1, i_2})$ in $N_{\mathbb{R}}$). Then, $b_{i_1} + b_{i_2} = b_{i_1, i_2}$ and $b_{i_1}$ and $b_{i_2}$ are \emph{not} both odd (i.e. at least one of them is even). We can imagine having two copies of versions of the figure on p. 9 of \cite{Ful} ``centered'' around different vector spaces spanning the ``ambient space'' in place of the line. \\
			
			\noindent In Proposition \ref{sign0specialcov}, this means that any ray of $\lk_\Sigma(\sigma_{i_1, i_2})$ is a special ray of some kind (so a ``center'') but $\rho_{i_1}$ and $\rho_{i_2}$ cannot both lie in the same special ray block (i.e. do not have a common special ray). Every ray of $\lk_\Sigma(\sigma_{i_1, i_2})$ being non-special with respect to one of $\rho_{i_1}$ or $\rho_{i_2}$ seems to imply a cross polytope-like structure from suspensions by Corollary \ref{conormrestrsusp}. A rough picture is two distinct linear subspaces spanning the entire ambient space that each support fans formed by repeated suspensions. \\
		\end{enumerate}
	\end{cor}

	\begin{rem} \textbf{(Moving between ambient spaces and maximal linear subspaces vs. links) \\} \label{linkmove}
		Suppose that we are working on a $p$-tuple of adjacent rays that are pairwise non-special. Thinking about flat links from Remark \ref{flatambvar}, it looks like we may be able to think of Part 2 and Part 3 in terms of Corollary \ref{p2conorm} in terms of common extensions to maximal sets of adjacent non-special rays with respect to $\rho_{i_1}$ and $\rho_{i_2}$ in the context of $\lk_\Sigma(\rho_{i_1})$ and $\lk_\Sigma(\rho_{i_2})$. This seems to be related to thinking about linear subspaces lying in the intersection of $N(\rho_{i_1})_{\mathbb{R}}$ and $N(\rho_{i_2})_{\mathbb{R}}$ (designated choices of codimension 1 vector subspaces of $N_{\mathbb{R}}$ not containing $\rho_{i_1}$ and $\rho_{i_2}$ respectively) as codimension 2 linear subspaces of $N_{\mathbb{R}}$ not containing $\rho_{i_1}$ or $\rho_{i_2}$. \\

		Consider a collection of maximal adjacent non-special ray sets $M_\ell$ of $\lk_\Sigma(\sigma_{i_1, \ldots, i_p})$ with respect to $A_\ell \subset [p]$ such that $|\lk_\Sigma(\sigma_{i_1, \ldots, i_p}, M_\ell)| \subset N(\rho_{i_j} : j \in C_\ell)_{\mathbb{R}}$ ($C_\ell \coloneq [p] \setminus A_\ell$) is a vector subspace with $\bigcup_{\ell \in \mathcal{A}} C_\ell = [p]$ (equivalent to $\bigcap_{\ell \in \mathcal{A}} A_\ell = \emptyset$). We seem to get $|\lk_\Sigma(\sigma_{i_1, \ldots, i_p}, M_\ell : \ell \in \mathcal{A})| \subset N(\sigma_{i_1, \ldots, i_p})_{\mathbb{R}}$ a vector subspace, which seems to correspond to a ``flat'' lift to $N_{\mathbb{R}}$. More generally, removing the condition $\bigcup_{\ell \in \mathcal{A}} C_\ell = [p]$ seems to yield a containment in $N\left( \bigcup_{\ell \in \mathcal{A}} C_\ell \right)_{\mathbb{R}}$. In the context above, this is similar to a ``flat lift'' to $N \left( [p] \setminus \bigcup_{\ell \in \mathcal{A}} C_\ell \right)_{\mathbb{R}} = N \left( \bigcap_{\ell \in \mathcal{A}} A_\ell \right)_{\mathbb{R}}$. \\
	\end{rem}
	
	Since special ray covers also exist for higher $p$ by Proposition \ref{specialcovpconext}, methods similar to those used in the $p = 2$ can also be used for related structural information on fans. \\

	\begin{cor} \textbf{(Implications on $p$-cones based on $p = 2$ methods) \\} \label{pconelikep2}
		Suppose that $\rho_{i_1}, \ldots, \rho_{i_p}$ form a $p$-tuple of pairwise non-special rays. If $2 \le p \le \frac{d}{2}$, then one of the following hold:
		
		\begin{enumerate}
			\item $-D_{i_j}|_{D_{i_1} \cap \cdots \cap D_{i_p}} = 0$ for some $j \in [p]$ \\
			
			\item $\dim P_{(-D_{i_1}) + \ldots + (-D_{i_p})}^{D_{i_1} \cap \cdots \cap D_{i_p}} \le d - p - 1$. This is equivalent to the support of $\lk_\Sigma(\sigma_{i_1, \ldots, i_p})$ in $N_{\mathbb{R}}$ containing a positive dimensional linear subspace. This is equivalent to being contained in a common ``special ray block'' in Proposition \ref{specialcovpconext}. \\
			
			\item Every ray of $\lk_\Sigma(\sigma_{i_1, \ldots, i_p})$ is non-special with respect to \emph{some} $\rho_{i_j}$. Corollary \ref{conormrestrsusp} then implies that $\lk_\Sigma(\sigma_{i_1, \ldots, i_p})$ is obtained via cones or repeated suspensions of them. 
		\end{enumerate}
		
	\end{cor}
	
	\subsection{Linear subspace containments from $p = \frac{d}{2}$, linear dependence via submodularity inequality equality cases, and interpolations in higher $p$ identities}  \label{higherpid}

	We return to implications of identities from higher $p$ related to nonexistence of $p$-tuples of odd exponents yielding nonzero mixed volumes. The $p = \frac{d}{2}$ case looks like containments/dependencies involving which $(-\rho_{i_j})$-heights are 0 (see p. 421 of \cite{Mat} and p. 282 -- 283 of \cite{LR}). \\

	\begin{prop} \textbf{($p = \frac{d}{2}$ case and containments/dependencies in fan structure) \\} \label{pdover2}
		For any $\frac{d}{2}$-tuple of adjacent rays $\rho_{i_1}, \ldots, \rho_{i_{\frac{d}{2}}}$, there is a $k$-tuple (say indexed by $\{ i_{j_1}, \ldots, i_{j_k} \}$) such that \[ \dim P_{\sum_{\ell = 1}^k (-D_{i_{j_\ell}})}^{ D_{i_1} \cap \cdots \cap D_{i_{\frac{d}{2}} }} \le k - 1. \] After some threshold (following an ordering of the variables), there is an instance where $R \subset P$ with $P$ being the Minkowski sum of polytopes considered so far and $R$ being the new polytope we are taking the Minkowski sum with. In theory, we may have an instance where the upper bound does \emph{not} after the threshold such as a sequence of dimensions of successive Minkowski sums $(1, 2, 2, 4, 5)$. \\
		
		These can be understood in terms of containments of non-special rays (Definition \ref{specnonspecdef} and Remark \ref{extaffequiv}). In terms of orthogonal complements (which relate to special rays and linear subspaces contained in link realizations), we have $\omega_P \subset \omega_R$. Implicitly, we have been using the implication $D, E$ nef meaning that $P_{D + E} = P_D + P_E$ (p. 69 of \cite{Ful}). \\
	\end{prop}
	
	\pagebreak 
	
	\begin{rem} \textbf{(Degeneracy of the $p = \frac{d}{2}$ case) \\}
		In some sense, the $p = \frac{d}{2}$ case is the ``most degenerate''. This is because $D_{i_1} \cdots D_{i_{\frac{d}{2}}} \cdot D_{i_1} \cdots D_{i_{\frac{d}{2}}} = 0$ and we fail the condition $b_A \ge |A|$, which is necessary for an odd tuple to exist in the nonzero mixed volume region. \\
	\end{rem}
	
	\color{black}
	
	The $p = 2$ and $p = \frac{d}{2}$ cases both illustrate properties that will be observed for general $2 \le p \le \frac{d}{2}$. We will first reframe $p = 2$ in the context of equality cases of submodularity inequalities. \\

	\begin{exmp} \textbf{($p = 2$ and submodularity inequality equality cases in general) \\} \label{submodineqeqgeom}
		Part 3 of the $p = 2$ case in Proposition \ref{p2mvodd} can be thought of an equality case of the submodularity inequality. Recall that the submodularity inequality states that $b_S + b_T \ge b_{S \cup T} + b_{S \cap T}$ for $S, T \subset [p]$ (Theorem 3.9 on p. 1993 of \cite{CL}). In the $p = 2$, we had $S = \{ i_1 \}$ and $T = \{ i_2 \}$. \\
		
		This also specializes to $C \subset A$ and $k \in C$ implies that $b_A - b_{A \setminus k} \le b_C - b_{C \setminus k}$ when we take $S = C$ and $T = A \setminus k$. In the $p = 2$ case, we have $A = \{ i_1 \}$, $C = \{ i_1, i_2 \}$, and $k = i_1$. \\

		Thinking about dimensions of Minkowski sums of polytopes, the general shape is like considering polytopes $P \supset Q$ (in terms of spans) and cases where $\dim(P + R) - \dim P = \dim(Q + R) - \dim Q$. We note that $\dim(P + R) - \dim P \le \dim(Q + R) - \dim Q$ in general. The example we have in mind takes $R$ as a single new polytope and $P$ as a Minkowski sum of a collection of polytopes containing those whose Minkowski sum is $Q$. In the framework above, $R$ is labeled by $k$. This is equivalent to $P \cap R \subset Q$, which can be phrased in terms of orthogonal complements as $\omega_Q \subset \omega_P + \omega_R$, which is a linear dependence relation. \\

	\end{exmp}

	Suppose that $\sum_{j \in [N]} = a_j = b_{[N]}$. On the restriction to this hyperplane, we note that the point $(a_1, \ldots, a_N)$ yields a point of the generalized permutohedron (e.g. yields nonzero mixed volumes) if and only if \[ b_{[N]} - b_{[N] \setminus A} \le \sum_{j \in A} a_j \le b_A \] for all $A \subset [N]$. For example, we have that \[ b_{[N]} - b_{[N] \setminus k} \le a_k \le b_k \] for all $k \in [N]$. \\
	
	In particular, appending $a_k = b_{[N]} - b_{[N] \setminus k}$ to a tuple $(a_1, \ldots, \widehat{a_k}, \ldots, a_N)$ satisfying the nonzero mixed volume inequalities with respect to $[N] \setminus k$ yields a point $(a_1, \ldots, a_N)$ satisfying the nonzero mixed volume inequalities with respect to $[N]$ by the submodularity inequality. \\

	\begin{prop} \label{minbaselift}
		If a point $(a_1, \ldots, \widehat{a_k}, \ldots, a_N)$ satisfies the nonzero mixed volume inequalities for $[N] \setminus k$, then appending $a_k = b_{[N]} - b_{[N] \setminus k}$ produces a point $(a_1, \ldots, a_N)$ respecting the nonzero mixed volume inequalities for $[N]$. \\
	\end{prop}
	
	\begin{proof}
		This is a consequence of the submodularity inequality $b_S + b_T \ge b_{S \cup T} + b_{S \cap T}$. Consider a subset $A \subset [N]$ such that $A \ni k$. Setting $S = A$ and $T = [N] \setminus k$, we have that $S \cup T = [N]$ and $S \cap T = A \setminus k$. Substituting this back into the submodularity inequality, we have that $b_A + b_{[N] \setminus k} \ge b_{[N]} + b_{A \setminus k}$. This is equivalent to having $b_A - b_{A \setminus k} \ge b_{[N]} - b_{[N] \setminus k}$. \\
		
		For $(a_1, \ldots, a_N)$ and the nonzero mixed volume inequalities for $[N]$, we only need to consider subsets $A \subset [N]$ such that $A \ni k$ since $(a_1, \ldots, \widehat{a_k}, \ldots, a_N)$ satisfies the nonzero mixed volume inequalities for subsets of $[N] \setminus k$. Taking the sum of the coordinates indexed by elements of $A$, we have
		
		\begin{align*}
			\sum_{j \in A} a_j &= \sum_{j \in A \setminus k} a_j + a_k \\
			&= \sum_{j \in A \setminus k} a_j + (b_{[N]} - b_{[N] \setminus k}) \\
			&\le b_{A \setminus k} + (b_{[N]} - b_{[N] \setminus k}) \\
			&\le b_{A \setminus k} + (b_A - b_{A \setminus k}) \\
			&= b_A.
		\end{align*}
		
		The second inequality follows from the form of the submodularity inequality written out in the first paragraph. \\
	\end{proof}

	The first point one might come up with after repeatedly using such points (e.g. in an attempt to produce a point using induction on dimension) is related to previous optimization-related work on generalized permutohedra/submodular functions. \\
	
	\begin{rem} \textbf{(Optimization-related comments) \\}
		\label{optpov}
		\vspace{-3mm}
		\begin{enumerate}
			
			\item If we optimize (e.g. minimize) a linear form on the restriction of the intersection of nonzero mixed volume region and the top hyperplane $a_1 + \ldots + a_N = b_{[N]}$ (a generalized permutohedron, also called the base polyhedron of a submodular function $f$ with $f(A) = b_A$) to odd tuples, the optimal point on the restriction to odd tuples is either an optimal point overall or is $e_i - e_j$ ($i \ne j$) away from the boundary of the generalized permutohedron. Note that the converse might not necessarily hold (i.e. such points on the boundary aren't necessarily more optimal than $x$). Heuristically, one may expect to have something like an optimal point among odd tuples. We may have a point that looks locally optimal, but being globally optimal via a greedy algorithm seems to generally use the submodularity property (which is ``close'' to holding but not necessarily exactly due to parity restrictions/modifications). This is based on a modification of the proof of Theorem 3.15 on p. 61 -- 62 of \cite{F1}. \\

			\item Fix a permutation $\sigma \in S_N$. On each turn, we choose the minimal point possible (see Corollary 3.3 on p. 313 -- 314 of \cite{FT} and Theorem 3.15 on p. 61 of \cite{F1}). This means that the inequalities involving terms of index $\sigma(\ell)$ for $\ell \le N - 1$ are the smallest possible (i.e. seem to need ``at least'' these coordinates to work if odd tuples satisfying the nonzero mixed volume inequality exist). For checking whether the output of the algorithm is compatible with the nonzero mixed volume inequalities, we seem to have a collection of potential discrepancy/strengthening of the submodularity inequality (pieces broke into segments added together). Working recursively, this seems to be related to a comparison between $(b_{\sigma([i_\ell])} - b_{\sigma([i_\ell - 1])}) + \mu_\ell$ (with $\mu_\ell$ being the ``discrepancy'' added) and $(b_{ \{ \sigma(i_1,), \ldots, \sigma(i_\ell) \} } - b_{ \{ \sigma(i_1,), \ldots, \sigma(i_{\ell - 1}) \} })$. We note that $b_{\sigma([i_\ell])} - b_{\sigma([i_\ell - 1])} \le b_{ \{ \sigma(i_1,), \ldots, \sigma(i_\ell) \} } - b_{ \{ \sigma(i_1,), \ldots, \sigma(i_{\ell - 1}) \} }$ in general since $b_C - b_{C \setminus k} \ge b_A - b_{A \setminus k}$ for $C \subset A$ and $k \in C$ by the submodularity inequality. This seems to be related to an induction on the size $\ell$ of the index sets considered. \\
			
		\end{enumerate}
	\end{rem}

	If we would like to produce points preserving the nonzero mixed volume property while keeping the coordinates $a_j$ odd (a sort of expected extreme point among odd tuples), we can produce an algorithm for doing this. \\

	Fix an ordering of the variables $a_1, \ldots, a_N$ coming from a permutation $\pi \in S_N$. This essentially involves choosing $b_{\pi(1)}$ or $b_{\pi(1)} - 1$ for the value of $a_{\pi(1)}$ at the beginning (whichever one is odd) and then choosing whatever is the minimal possible choice of odd $a_{\pi(2)}$ that could possibly produce a point $(a_{\pi(1)}, a_{\pi(2)})$ respecting the nonzero mixed volume inequalities. We continue this as we increase the input $i$ in $\pi(i)$ until $i = N$. Note that this seems to list points in the opposite order from the optimization literature such as \cite{FT} and \cite{F1} (which seem to use the ``opposite'' $\sigma \in S_N$ defined by $\pi(q) = \sigma(N + 1 - q)$, which implies that $[N] \setminus \pi([q]) = \sigma([N - q])$. \\
	
	Before listing the general algorithm, we illustrate how it would work in the $p = 3$ case and how it relates to equality cases of the submodularity inequality as in the $p = 2$ case. As mentioned in Part 3 of Remark \ref{permnot}, we will use $N$ in place of $p$ for the number of variables in results on arbitrary generalized permutohedra. We will assume that $b_{[N]}$ and $N$ have the same parity and that $b_A \ge |A|$ for all $A \subset [N]$. The second assumption is necessary in order for odd tuples to exist in the nonzero mixed volume region in the first place. In addition, we will assume that $b_R \le b_S$ if $R \subset S$. This is because we will take $b_A \coloneq \dim \left( \sum_{j \in A} Q_j \right)$ for a given initial collection of polytopes $Q_1, \ldots, Q_N$.  \\

	\begin{prop} \textbf{($N = 3$ case and equality cases of the submodularity inequality) \\} \label{p3algsubmodex}
		
		\vspace{-3mm}
		Suppose that $b_{[3]}$ is odd. Fix a permutation $\sigma \in S_3$. \\
		
		We will assume that $b_{[3]} - b_{[3] \setminus \ell}$ is even for all $\ell \in [3]$. Otherwise, we can set $a_\ell = b_{[3]} - b_{[3] \setminus \ell}$ and use Proposition \ref{minbaselift} to reduce to the 2 variable case (which was covered in Proposition \ref{p2mvodd} and Corollary \ref{p2conorm}). Note that it would suffice to satisfy the condition $\sum_{j \in [3] \setminus \ell} a_j < b_{[3] \setminus \ell}$ in that setting. 
		\vspace{0.5mm}
		\begin{enumerate}
			\item If $b_{\sigma(1)}$ is even, the odd tuple algorithm based on successively choosing minimal elements after the first one above produces the point $(a_{\sigma(1)}, a_{\sigma(2)}, a_{\sigma(3)}) = (b_{\sigma(1)} - 1, b_{[3] \setminus \sigma(3)} - b_{\sigma(1)}, b_{[3]} - b_{[3] \setminus \sigma(3)} + 1)$. This point is compatible with the nonzero mixed volume inequalities if and only if $b_{[3]} - b_{\sigma(2), \sigma(3)} = b_{[3]} - b_{[3] \setminus \sigma(1)} \le b_{\sigma(1)} - 1$ and $b_{[3]} - b_{[3] \setminus \sigma(3)} \le b_{\sigma(3)} - 1$. These show that certain specializations of the submodularity inequality (see version on Proposition \ref{minbaselift}) are \emph{strict} inequalities. \\
			
			\item If $b_{\sigma(1)}$ is odd, it produces the point $(a_{\sigma(1)}, a_{\sigma(2)}, a_{\sigma(3)}) = (b_{\sigma(1)}, b_{[3] \setminus \sigma(3)} - b_{\sigma(1)} - 1, b_{[3]} - b_{[3] \setminus \sigma(3)} + 1)$. This point is compatible with the nonzero mixed volume inequalities if and only if $b_{[3]} - b_{\sigma(2), \sigma(3)} = b_{[3]} - b_{[3] \setminus \sigma(1)} \le b_{\sigma(1)} - 1$ and $b_{[3]} - b_{[3] \setminus \sigma(3)} \le b_{\sigma(3)} - 1$, which are also state that certain specializations of the submodularity inequality (version on Proposition \ref{minbaselift}) are \emph{strict} inequalities. \\
			
			Note that a stronger condition is required in order to produce a valid point producing a nonzero mixed volume when $b_{\sigma(1)}$ is odd than in the case where $b_{\sigma(1)}$ is even since $b_{\sigma(3)} \ge b_{\sigma(1), \sigma(3)} - b_{\sigma(1)}$ by the submodularity inequality. This means that a smaller upper bound when $b_{\sigma(1)}$ is odd. \\
		\end{enumerate}
		
		In particular, the statements above show that the instances where the algorithm fails to produce an odd tuple are those where the submodularity inequality is an equality for certain specializations. This is one similarity to the $p = 2$ case discussed in Proposition \ref{p2mvodd} and Corollary \ref{p2conorm}. The geometry related to these equality cases is also discussed in Example \ref{submodineqeqgeom}. \\
		
	\end{prop}
	
	\begin{proof}
		\begin{enumerate}
			\item We note that the $a_{\sigma(1)} + a_{\sigma(3)} \le b_{\sigma(1), \sigma(3)}$ follows from the submodularity inequality $b_S + b_T \ge b_{S \cup T} + b_{S \cap T}$ (Theorem 3.9 on p. 1993 of \cite{CL}) with $S = \{ \sigma(1), \sigma(3) \}$ and $T = \{ \sigma(1), \sigma(2) \}$. Also, the submodularity inequality with $S = \{ \sigma(1) \}$ and $T = \{ \sigma(2) \}$ implies that $a_{\sigma(2)} = b_{[3] \setminus \sigma(3)} - b_{\sigma(1)} = b_{\sigma(1), \sigma(2)} - b_{\sigma(1)} \le b_{\sigma(2)}$. Apart from this, we have that $a_{\sigma(1)} = b_{\sigma(1)} - 1 < b_{\sigma(1)}$, $a_{\sigma(1)} + a_{\sigma(2)} = b_{\sigma(1), \sigma(2)} - 1 < b_{\sigma(1), \sigma(2)}$, and $a_{\sigma(1)} + a_{\sigma(2)} + a_{\sigma(3)} = b_{[3]}$. It remains to check when it is possible to have $a_{\sigma(2)} + a_{\sigma(3)} \le b_{\sigma(1), \sigma(3)}$ and $a_{\sigma(3)} \le b_{\sigma(3)}$. In other words, we need $b_{[3]} - b_{\sigma(1)} + 1 \le b_{\sigma(2), \sigma(3)}$ and $b_{[3]} - b_{[3] \setminus \sigma(3)} + 1 \le b_{\sigma(3)}$. These conditions are equivalent to $b_{[3]} - b_{\sigma(2), \sigma(3)} = b_{[3]} - b_{[3] \setminus \sigma(1)} \le b_{\sigma(1)} - 1$ and $b_{[3]} - b_{[3] \setminus \sigma(3)} \le b_{\sigma(3)} - 1$, which state that the version of the submodularity inequality  of the form $b_A - b_{A \setminus k} \le b_C - b_{C \setminus k}$ for $C \subset A$ and $k \in C$ (discussed in Proposition \ref{minbaselift}) is a \emph{strict} inequality for the sets used. \\
			
			\item The reasoning is similar to that of Part 1. For singletons, the only inequality that does not follow automatically from a direct comparison of values or the submodularity inequality is $a_{\sigma(3)} \le b_{\sigma(3)}$, which is the condition $b_{[3]} - b_{[3] \setminus \sigma(3)} \le b_{\sigma(3)} - 1$ from above. Since $a_{\sigma(1)} + a_{\sigma(2)} = b_{[3]} - b_{\sigma(1)}$, the submodularity inequality $b_S + b_T \ge b_{S \cup T} + b_{S \cap T}$ with $S = \{ \sigma(1) \}$ and $T = [3] \setminus \sigma(1)$ implies that $a_{\sigma(1)} + a_{\sigma(2)} \le b_{\sigma(1), \sigma(2)}$. We also have $a_{\sigma(1)} + a_{\sigma(2)} = b_{[3] \setminus \sigma(1)} - 1 = b_{\sigma(1), \sigma(2)} - 1 < b_{\sigma(1), \sigma(2)}$. However, satisfying the inequality $a_{\sigma(1)} + a_{\sigma(3)} = b_{\sigma(1)}  + b_{[3]} - b_{[3] \setminus \sigma(3)} + 1 \le b_{\sigma(1), \sigma(3)}$ requires a submodularity inequality to be a \emph{strict} inequality since we need $b_{[3]} - b_{[3] \setminus \sigma(3)} \le b_{[3] \setminus \sigma(2)} - b_{\sigma(1)} - 1 = b_{\sigma(1), \sigma(3)} - b_{\sigma(1)} - 1$. \\
		\end{enumerate}
	\end{proof}

	For general $p$, the cases where the odd tuple modification of the algorithm in Remark \ref{optpov} fails to produce a valid odd tuple compatible with the nonzero mixed volume inequalities (i.e. staying in the generalized permutohedron) are those where certain specializations of the submodularity inequality attain equality as in the proof of the $p = 3$ case in Proposition \ref{p3algsubmodex}. In order to study when this occurs in general, we first recall that we can prove that the extreme point discussed in Remark \ref{optpov} (and motivated earlier by Proposition \ref{minbaselift}) is compatible with the nonzero mixed volume inequalities via repeated applications of the inequality $b_{\sigma([i_\ell])} - b_{\sigma([i_\ell - 1])} \le b_{\sigma(i_1), \ldots, \sigma(i_\ell)} - b_{\sigma(i_1), \ldots, \sigma(i_{\ell - 1})}$ (which itself is a specialization of the submodularity inequality discussed in Proposition \ref{minbaselift}). \\

	We now describe our main strategy for analyzing higher $p$ identities. \\
	
	\begin{rem} \textbf{(Higher $p$ identities and running totals) \\}
		\label{runtot}
		To measure when the output of the algorithm fails to yield an odd tuple that actually satisfies the nonzero mixed volume inequalities, we will use a ``running total'' comparison applied to the indices selected for the subset $A \subset [N]$ under consideration (with elements listed in the order determined by the permutation $\pi \in S_N$). Suppose that we have $A = \{ \pi(q_1), \ldots, \pi(q_m) \}$ with $q_1 < \cdots < q_m$. We have an initial difference between the upper bound and lower bound given by $b_{\pi(q_1)} - a_{\pi(q_1)}$. This serves as a ``cushion'' for the rest of the term comparisons. In general, we can measure the \emph{change} in the difference between the desired upper bound and the current total sum of the $a_{\pi(q_j)}$ by taking the difference between the change in the upper bound and the change in the sum. \\
	\end{rem}

	The algorithm below must fail to yield a valid output respecting the nonzero mixed volume inequalities if an odd tuple does \emph{not} exist. Its output is listed below and is an odd tuple analogue of the optimal points on generalized permutohedra discussed above. \\
	
	\begin{prop} \textbf{(Odd tuple algorithm output and associated parity conditions) \\} \label{oddalgoutput}
		Let $b_A$ for $A \subset [N]$ be a collection of positive integers satisfying the submodularity inequalities $b_S + b_T \ge b_{S \cup T} + b_{S \cap T}$. Suppose that $a_1 + \ldots + a_N = b_{[N]}$, $b_A \ge |A|$ for all $A \subset [N]$, and $b_R \le b_S$ for all $R \subset S$. \\
		
		For each permutation $\pi \in S_N$, the output of the odd tuple algorithm consists of ``blocks'' of terms subdividing $\pi(1), \ldots, \pi(N)$ with each block starting with indices $\pi(w_i)$ and any other element of the same block being set equal to 1 (i.e. $a_{\pi(q)} = 1$ if it is such an element). These elements where 1 is assigned by default are those where the ``running sum'' of the variables just before removal is less than or equal to $b_{[N] \setminus \pi([q])}$. Each of the ``transition points'' starting these blocks have values 
		
		\begin{align*}
			\hspace{-20mm} a_{\pi(w_i)} &= T_{w_i} - T_{w_i + 1} \\
			&= (T_{w_{i - 1} + 1} - (w_i - w_{i - 1} - 1)) - \begin{cases} 
				b_{[N] \setminus \pi([w_i])} & \text{ if } b_{[N] \setminus \pi([w_i])} \text{ has parity } N - w_i \\
				b_{[N] \setminus \pi([w_i])} - 1 & \text{ if } b_{[N] \setminus \pi([w_i])} \text{ has parity } N - w_i - 1
			\end{cases} \\
			&= \begin{cases} 
				b_{[N] \setminus \pi([w_{i - 1}])} - b_{[N] \setminus \pi([w_i])} & \text{ if } b_{[N] \setminus \pi([w_{i - 1}])} \text{ and } b_{[N] \setminus \pi([w_i])}  \\
				& \text{ have parities }  N - w_{i - 1} \text{ and } N - w_i  \\
				& \text{ or } N - w_{i - 1} - 1 \text{ and } N - w_i - 1 \\
				b_{[N] \setminus \pi([w_{i - 1}])} - b_{[N] \setminus \pi([w_i])} - 1 & \text{ if } b_{[N] \setminus \pi([w_{i - 1}])} \text{ and } b_{[N] \setminus \pi([w_i])} \\ 
				&\text{ have parities } N - w_{i - 1} - 1 \text{ and } N - w_i \\
				b_{[N] \setminus \pi([w_{i - 1}])} - b_{[N] \setminus \pi([w_i])} + 1 & \text{ if } b_{[N] \setminus \pi([w_{i - 1}])} \text{ and } b_{[N] \setminus \pi([w_i])} \\ 
				&\text{ have parities } N - w_{i - 1} \text{ and } N - w_i - 1 
			\end{cases} \\
			&- (w_i - w_{i - 1} - 1).
		\end{align*}
		By  ``$\alpha$ has parity $\beta$'', we mean that $\alpha$ and $\beta$ have the same parity. \\
	\end{prop}
	
	\begin{rem} \textbf{(Permutation choice compared to optimization context) \\}
		In comparison to permutations $\sigma \in S_N$ with $\sigma([k]) = S_i$ in the Corollary 3.3 on p. 313 -- 314 of \cite{FT}, we have that $\pi(a) = \sigma(N + 1 - a)$ and $[N] \setminus \pi([q]) = \sigma([N - q])$. In other words, it lists the elements of $[N]$ in reverse order compared to $\sigma$. We choose to use $\pi$ since it seems to be easier to think about remaining terms and what are ``minimal possible values'' of coordinates compatible with the nonzero mixed volume inequalities among those lying on that particular hyperplane (for the sum of the remaining variables) after a value is chosen (which was subtracted from the previous total sum). \\
	\end{rem}
	
	\begin{proof}
		
		Fix a permutation $\pi\in S_N$. \\

		We will start with Step 0. Let $T_1 = b_{[N]}$. After Step $q - 1$, we will set $T_q \coloneq \sum_{j \in [N] \setminus \pi([q - 1])} a_j$. The general idea is to try to choose $a_{\pi(q)}$ to be the minimal possible value compatible with the nonzero mixed volume relations among those on the given sum $T_q$ (i.e. just before $a_{\pi(q)}$ is removed) while keeping coordinates nonnegative and assigning odd values to them. \\ 
		
		In order to keep terms nonnegative, we will consider $c_A \coloneq \min(C, b_A)$. We will take $C = T_q$ at the start of Step $q$ in our setting (i.e. in the setup above). Note that the submodularity inequalities $c_S + c_T \ge c_{S \cup T} + c_{S \cap T}$ follow from $b_S + b_T \ge b_{S \cup T} + b_{S \cap T}$ for $b_A$ with $A \subset [N]$ since we assumed that $R \subset S$ implies $b_R \le b_S$. Since this will not be used for the remainder of the proof, we will state its proof separately after this one (Remark \ref{minsubmod}). \\

		In addition, we have that 
		
		\[ c_A - c_{A \setminus k} = \begin{cases} 
			b_A - b_{A \setminus k} & \text{ if $C \ge b_A$} \\
			C - b_{A \setminus k} & \text{ if $b_{A \setminus k} \le C \le b_A$} \\
			0 & \text{ if $C \le b_{A \setminus k}$}
		\end{cases} \]
		
		This means that we will take $a_{\pi(q)} = 1$ if $T_q \le b_{[N] \setminus \pi([q])}$. In such cases, we have that $T_{q + 1} = T_q - 1$. Here, our ``ground set'' will be $A = [N] \setminus \pi([q - 1])$. \\ 
		
		We also have that 
		
		\[ \hspace{-20mm} T_{q + 1} = \begin{cases} 
			T_q - (T_q - \min(T_q, b_{[N] \setminus \pi([q]) })) = \min(T_q, b_{[N] \setminus \pi([q]) }) & \text{ if } T_q - \min(T_q, b_{[N] \setminus \pi([q]) }) \\
			&\text{ is odd} \\
			T_q - (T_q - \min(T_q, b_{[N] \setminus \pi([q]) }) + 1) = \min(T_q, b_{[N] \setminus \pi([q]) }) - 1 & \text{ if } T_q - \min(T_q, b_{[N] \setminus \pi([q]) }) \\ 
			&\text{ is even.}
		\end{cases} \]

		As mentioned above, we are trying to subtract the minimal value of $a_{\pi(q)}$ possible to form $T_{q + 1}$ from $T_q$ while keeping terms nonnegative. This was the reason for subtracting $\min(T_q, b_{[N] \setminus \pi([q])})$ from $T_q$ and adding 1 if necessary to make parity modifications. \\
		
		\color{black}

		In order to consider the indices yielding the case with $c_A - c_{A \setminus k} = 0$ separately, we will study where the values of $\pi([N] \setminus \pi([q]))$ strictly decrease or stay the same as we increase $q$ (i.e. omit more elements). Among the inputs in $[N]$, we will denote these ``transition steps'' $w_1 + 1, \ldots, w_\ell + 1$ where $T_{w_i} > b_{[N] \setminus \pi([w_i])}$. Such inputs yield the first element of a ``block'' consisting of a point that is \emph{not} assigned 1 by default followed by ones that are (if they exist). Given $0 \le m \le w_i - w_{i - 1} - 1$, we have that $T_{w_{i - 1} + 1 + m} = T_{w_{i - 1} + 1} - m$. For these transition points, we have 
		
		\[ T_{w_i + 1} = \begin{cases} 
			b_{[N] \setminus \pi([w_i])} & \text{ if } b_{[N] \setminus \pi([w_i])} \text{ has parity } N - w_i \\
			b_{[N] \setminus \pi([w_i])} - 1 & \text{ if } b_{[N] \setminus \pi([w_i])} \text{ has parity } N - w_i - 1.
		\end{cases} \]   
		
		\vspace{-2mm}
		
		This implies that
		\begin{align*}
			\hspace{-15mm} a_{\pi(w_i)} &= T_{w_i} - T_{w_i + 1} \\
			&= (T_{w_{i - 1} + 1} - (w_i - w_{i - 1} - 1)) - \begin{cases} 
				b_{[N] \setminus \pi([w_i])} & \text{ if } b_{[N] \setminus \pi([w_i])} \text{ has parity } N - w_i \\
				b_{[N] \setminus \pi([w_i])} - 1 & \text{ if } b_{[N] \setminus \pi([w_i])} \text{ has parity } N - w_i - 1
			\end{cases} \\
			&= \begin{cases} 
				b_{[N] \setminus \pi([w_{i - 1}])} - b_{[N] \setminus \pi([w_i])} & \text{ if } b_{[N] \setminus \pi([w_{i - 1}])} \text{ and } b_{[N] \setminus \pi([w_i])} \\ 
				&\text{ have parities } N - w_{i - 1} \text{ and } N - w_i  \\
				& \text{ or } N - w_{i - 1} - 1 \text{ and } N - w_i - 1 \\
				b_{[N] \setminus \pi([w_{i - 1}])} - b_{[N] \setminus \pi([w_i])} - 1 & \text{ if } b_{[N] \setminus \pi([w_{i - 1}])} \text{ and } b_{[N] \setminus \pi([w_i])} \\ 
				&\text{ have parities } N - w_{i - 1} - 1 \text{ and } N - w_i \\
				b_{[N] \setminus \pi([w_{i - 1}])} - b_{[N] \setminus \pi([w_i])} + 1 & \text{ if } b_{[N] \setminus \pi([w_{i - 1}])} \text{ and } b_{[N] \setminus \pi([w_i])} \\ 
				&\text{ have parities } N - w_{i - 1} \text{ and } N - w_i - 1 
			\end{cases} \\
			&- (w_i - w_{i - 1} - 1).
		\end{align*}

	\end{proof}

	\begin{rem} \textbf{(Taking minima and preserving submodularity) \\} \label{minsubmod}
		
		We will write down the claim from the proof of Proposition \ref{oddalgoutput} that $c_A \coloneq \min(C, b_A)$ produces terms satisfying the submodularity inequality $c_S + c_T \ge c_{S \cup T} + c_{S \cap T}$ if $b_S + b_T \ge b_{S \cup T} + b_{S \cap T}$ and $R \subset S$ implies that $b_R \le b_S$. Note that this uses the assumption that $R \subset S$ implies $b_R \le b_S$ and involves casework based on the size of $C$ relative to the terms of the inequality. Without loss of generality, suppose that $b_S \le b_T$. It is clear that the inequality $c_S + c_T \ge c_{S \cup T} + c_{S \cap T}$ is satisfied if $C \le b_{S \cap T}$ or $C \ge b_{S \cup T}$ since both sides are equal to $2C$ in the former case and the inequality is the original submodularity inequality $b_S + b_T \ge b_{S \cup T} + b_{S \cap T}$ in the latter case. Suppose that $b_{S \cap T} \le C \le b_S$. Then, the inequality $c_S + c_T \ge c_{S \cup T} + c_{S \cap T}$ simplifies to $C \ge b_{S \cap T}$, which is true by our assumption. If $b_S \le C \le b_T$, the inequality $c_S + c_T \ge c_{S \cup T} + c_{S \cap T}$ reduces to $b_S \ge b_{S \cap T}$. This follows from our assumption that $R \subset S$ implies $b_R \le b_S$. Finally, we consider the case where $b_T \le C \le b_{S \cup T}$. The inequality $c_S + c_T \ge c_{S \cup T} + c_{S \cap T}$ reduces to $b_S + b_T \ge C + b_{S \cup T}$. This is actually implied by the original submodularity inequality since $b_S + b_T \ge b_{S \cup T} + b_{S \cap T}$ and we assumed that $C \le b_{S \cup T}$. \\
	\end{rem}

	\color{black}

	We can compare the ``transition points'' from the algorithm in Proposition \ref{oddalgoutput} to the entries of the ``usual'' extreme points of the submodular function referred to in Remark \ref{optpov} and Proposition \ref{minbaselift}.  \\

	\begin{cor} \textbf{(Comparison between odd tuple algorithm output and ``usual'' optimal/extreme points) \\} \label{oddalgcomp}
		In Proposition \ref{oddalgoutput}, we have that \[ a_{\pi(w_i)} - \alpha_i \le b_{[N] \setminus \pi([w_i - 1])} - b_{[N] \setminus \pi([w_i])}, \] where
		
		\[ \alpha_i = \begin{cases} 
			0 & \text{ if } b_{[N] \setminus \pi([w_{i - 1}])}  \text{ has parity } N - w_{i - 1}   \\
			1 & \text{ if } b_{[N] \setminus \pi([w_{i - 1}])}  \text{ has parity } N - w_{i - 1} - 1.
		\end{cases}  \]
		
	\end{cor}
	
	\begin{proof}
		
		Since subtracting $w_i - w_{i - 1} - 1$ involves indices that still lie within the block started by $w_{i - 1}$, we have that $T_{w_{i - 1} + 1} - (w_i - w_{i - 1} - 1) \le b_{[N] \setminus \pi([w_i - 1])}$. We can then apply
		\[ T_{w_{i - 1} + 1} = \begin{cases} 
			b_{[N] \setminus \pi([w_{i - 1}])} & \text{ if } b_{[N] \setminus \pi([w_{i - 1}])} \text{ has parity } N - w_{i - 1} \\
			b_{[N] \setminus \pi([w_{i - 1}])} - 1 & \text{ if } b_{[N] \setminus \pi([w_{i - 1}])} \text{ has parity } N - w_{i - 1} - 1
		\end{cases} \]   
		
		from the proof of Proposition \ref{oddalgoutput}. Note that \[ T_{w_{i - 1} + 1 + m_{i - 1}} = T_{w_{i - 1} + 1} - m_{i - 1} \le b_{[N] \setminus \pi([w_{i - 1} + 1 + m_{i - 1}])} \] for all $0 \le m_{i - 1} \le w_i - w_{i - 1} - 1$. Combining this with the parity conditions above gives the claimed bound. 
		
	\end{proof}
	
	We will now keep track of the ``running totals'' mentioned in Remark \ref{runtot}. This will yield conditions describing when the outputs of Proposition \ref{oddalgoutput} are (in)compatible with the nonzero mixed volume inequality for a particular subset $S \subset [N]$. \\

	\begin{cor} \textbf{(Running total classification and submodularity inequality equality cases) \\} \label{oddtupzerovoltot}
		\begin{enumerate}
			\item \textbf{(Comparing desired upper bounds with actual sums) \\}
			Consider a subset $S \subset [N]$. We will run through the elements of $S$ in increasing order of inputs into the given permutation $\pi \in S_N$. We can write $S = \{ \pi(q_1), \ldots, \pi(q_{|S|}) \}$ for some $q_1 > \cdots > q_{|S|}$. The point is to check if $\sum_{j \in S} a_j \le b_S$ for the output of Proposition \ref{oddalgoutput} and determine when the inequality fails while we run through the variables in this order if it does not. \\
			\begin{itemize}
				\item The \textbf{initial difference} between the upper bound and the sum (which is a single term) is either of the form $b_{\pi(j)} - 1$ or 
				\vspace{-1mm}
				\begin{align*}
					\hspace{5mm} b_{\pi(w_i)} - a_{\pi(w_i)} &= b_{\pi(w_i)} - (b_{[N] \setminus \pi([w_{i -1}])} - (w_i - w_{i - 1} - 1) - b_{[N] \setminus \pi([w_i])} + \mu_i ),
				\end{align*}

				where $\mu_i \in \{ 0, -1, 1 \}$ is the constant from Proposition \ref{oddalgoutput}. \\
				
				We note that the term we are subtracting is either bounded above by $b_{[N] \setminus \pi([w_i - 1])} - b_{[N] \setminus \pi([w_i])} + \mu_i$ or $b_{[N] \setminus \pi([w_i - 1])} - b_{[N] \setminus \pi([w_i])} + \mu_i + 1$ depending if $b_{[N] \setminus \pi([w_i])}$ has parity $N - w_i$ or $N - w_i - 1$. \\
				
				\item Let $A \subset S$ be the set of elements of $S$ we have run through so far. \\
				
				If we add 1 after this, the \textbf{change in the difference} between the desired upper bound and the actual (running) sum is $b_{\pi(A \cup j)} - b_{\pi(A)} - 1$. \\
				
				If we add a transition point $a_{\pi(w_i)}$ after this, the \textbf{change in the difference} between the desired upper bound and the actual (running) sum is 
				
				\begin{align*}
					\hspace{7mm} (b_{\pi(A \cup w_i)} - b_{\pi(A)}) - a_{\pi(w_i)}
					&= (b_{\pi(A \cup w_i)} - b_{\pi(A)}) - (b_{[N] \setminus \pi([w_i - 1])} - b_{[N] \setminus \pi([w_i])}) \\
					&+ (b_{[N] \setminus \pi([w_i - 1])} - b_{[N] \setminus \pi([w_i])}) - a_{\pi(w_i)}. \\
				\end{align*}

				Since $A \subset [N] \setminus \pi([w_i - 1])$, the difference on the first line is nonnegative by the submodularity inequality (version from proof of Proposition \ref{minbaselift}). This inclusion holds since are running through $S$ so that the inputs into $\pi$ are decreasing and we did not hit $\pi(w_i)$ yet after running through $A$. Then, the inclusion $A \subset [N] \setminus \pi([w_i - 1])$ holds since $[N] \setminus \pi([w_i - 1])$ consists of the elements $\pi(j)$ with $j \ge w_i$. \\
				
				We note that the difference on the second line was studied in Corollary \ref{oddalgcomp}. Recall that \[ a_{\pi(w_i)} \le b_{[N] \setminus \pi([w_i - 1])} - b_{[N] \setminus \pi([w_i])}\] or \[ a_{\pi(w_i)} - 1 \le b_{[N] \setminus \pi([w_i - 1])} - b_{[N] \setminus \pi([w_i])} , \] where we use $a_{\pi(w_i)}$ or $a_{\pi(w_i)} - 1$ on the left hand side of the inequality if $b_{[N] \setminus \pi([w_{i - 1}])}$ has parity $N - w_{i - 1}$ or $N - w_{i - 1} - 1$ respectively. \\

			\end{itemize}

			\item \textbf{(Possible terms yielding nonnegative and strictly negative changes in upper bound and sum differences) \\}
			
			\begin{itemize}
				\item \textbf{Nonnegative changes}
				\begin{itemize}
					\item If $a_{\pi(j)} = 1$ is added, instances where $b_{\pi(A \cup j)} - b_{\pi(A)} \ge 1$. \\
					
					\item If $a_{\pi(w_i)}$ is added, instances where $b_{[N] \setminus \pi([w_{i - 1}])}$ has parity $N - w_{i - 1}$. 
				\end{itemize}
				
				\item \textbf{Possible negative changes}
				\begin{itemize}
					\item If $a_{\pi(j)} = 1$ added, instances where $b_{\pi(A \cup j)} - b_{\pi(A)} = 0$ (i.e. $b_{\pi(A \cup j)} = b_{\pi(A)}$). This would contribute a change of $-1$ to the difference. \\

					\item If $a_{\pi(w_i)}$ is added, instances where $b_{[N] \setminus \pi([w_{i - 1}])}$ has parity $N - w_{i - 1} - 1$. If this is negative, this would contribute a change of $-1$ to the difference. Note that being negative would require $b_{[N] \setminus \pi([w_{i - 1}])} - (w_i - w_{i - 1} - 1) + 1 = b_{[N] \setminus \pi([w_{i - 1}])}$. \\
				\end{itemize}
			\end{itemize}
			
			\item \textbf{(Compatibility with nonzero mixed volume inequalities) \\}
			
			\noindent Combining the possible changes in differences between desired upper bounds and actual running sums (Part 2) with the initial differences (Part 1), the subset $S \subset [N]$ can induce a failure of the output of the odd tuple algorithm (Proposition \ref{oddalgoutput}) to satisfy the nonzero mixed volume inequalities only if the following condition hold: \\
			
			\begin{itemize}
				\item The number of terms satisfying the possible negative changes (Part 2) from the second term onwards is strictly larger than the sums of the following terms: \\
				\begin{itemize}
					\item $(b_{\pi(A \cup j)} - b_{\pi(A)} - 1)$ and differences of the form $(b_{\pi(A \cup w_i)} - b_{\pi(A)}) - (b_{[N] \setminus \pi([w_i - 1])} - b_{[N] \setminus \pi([w_i])})$ from the positive change terms from Part 2. Note that terms of the second type are differences of the two sides of a specialization of the submodularity inequality. \\
					
					\item The initial chosen difference $b_{\pi(j)} - 1$ or $b_{\pi(w_i)} - a_{\pi(w_i)}$ from Part 1. \\
				\end{itemize}

			\end{itemize}

		\end{enumerate}
	\end{cor}
	
	\section{Minimality with respect to restricted blowups and constructing induced 4-cycles} \label{4cycleminwall}

	Wall crossings also provide a method of constructing the 4-cycles present in flag spheres that are minimal with respect to blowups coming from edge subdivisions of simplicial complex giving the structure of the cones of the fan (Theorem 1.1 of \cite{LN}). For example, every ray is contained in an induced 4-cycle if $\sigma(X_\Sigma) = 0$ and $\Sigma$ is a simplicial complete locally convex fan (Corollary \ref{sig04cycle}). This involves some geometric information on conormal bundle restrictions that are nef and not ample which may explain why it may be natural to see suspensions in fan structures in earlier combinatorial constructions from Section \ref{wallfan}. Applying the analysis of Lemma \ref{zeroconorm4cycle} to Proposition \ref{conormrestrsusp} and Corollary \ref{suspspecnon} or Corollary \ref{4cycleconstralg} in Corollary \ref{suspalg} yields possible algebraic structures connected to suspension structures of fans studied in Section \ref{wallfan}. \\
	
	In particular, the constructions of 4-cycles here are applications of rational equivalence in a codimension 2 cones/2-dimensional orbit closures instead of a codimension 1 cones/1-dimensional orbit closures as in most combinatorial applications (e.g. use of the wall relations and deformation cones in \cite{ACEP}). They are connected to the contraction theorem for toric varieties (Corollary 14-1-9 on p. 422 -- 423 of \cite{Mat}). \\  
	
	\begin{defn} \label{ind4cycle}
		An \textbf{induced 4-cycle} in a simplicial complex $\Delta$ is a collection of vertices $v_1, v_2, v_3, v_4$ where $(v_i, v_{i + 1}) \in \Delta$ for $1 \le i \le 4$ (with $v_5 \coloneq v_1$) and $(v_1, v_3), (v_2, v_4) \notin \Delta$. In a fan, we use the same term for the simplicial complex structure on the cones where the $k$-dimensional cones give the $(k - 1)$-dimensional faces of the simplicial complex. \\
	\end{defn}
	
	\begin{lem} \label{zeroconorm4cycle}
		Suppose that $\Sigma$ is a simplicial and complete fan of dimension $d \ge 4$ such that its support $|\Sigma| \subset N_{\mathbb{R}}$ is locally convex. Let $\tau = \sigma \cap \sigma'$ be a wall of $\Sigma$. If $D_\alpha \cdot V_\Sigma(\tau) = 0$ for some ray $\alpha \in \tau(1)$, then there is a ray $\alpha' \in \Sigma(1)$ satisfying the following properties: \\
		
		\begin{enumerate} 
			\item Substituting $\alpha'$ in place of $\alpha$ in $\sigma$ and $\sigma'$ yields cones of $\Sigma$. \\
			
			\item $V_\Sigma(\tau') \in \mathbb{R}_+[V_\Sigma(\tau)]$, where $\tau' = \Cone(\tau/\alpha, \alpha')$.  \\
			
			\item The rays $\alpha$, $\beta \coloneq \sigma/\tau$, $\alpha'$, and $\beta' \coloneq \sigma'/\tau$ form an induced 4-cycle in $\Sigma$. 
		\end{enumerate}

	\end{lem}
	
	\begin{rem} 
		Since $\Sigma$ is simplicial, the assumption that $\alpha \in \tau(1)$ is necessary in order to have $D_\alpha \cdot V_\Sigma(\tau) = 0$. Also, the $d \ge 4$ case accounts for all nontrivial cases in the setting of our problem where $d$ is even since $d = 2$ yields cones whose bases are polygons with $\ge 4$ sides. \\

	\end{rem}

	\begin{proof}
		Since Part 1 and Part 2 are an application of Proposition 14-1-5 on p. 419 and Lemma 14-1-1 on p. 414 of \cite{Mat}, we will focus on Part 3. Note that $\tau$ yields an extremal ray of $\overline{NE}(X_{\lk_\Sigma(\omega)})$ when $\tau \supset \omega$ since $D_\alpha \cdot V_\Sigma(\tau) = 0$. The first two parts imply that $(\alpha, \beta), (\beta, \alpha'), (\alpha', \beta'), (\beta', \alpha) \in \Sigma(2)$. It remains to show that $(\alpha, \alpha'), (\beta, \beta') \notin \Sigma$. We first show that $(\beta, \beta') \notin \Sigma$. The wall relation from rational equivalence (p. 301 of \cite{CLS}) implies that \[ \sum_{ \gamma \in \tau(1)} (D_\gamma \cdot V_\Sigma(\tau)) u_\gamma + (D_\beta \cdot V_\Sigma(\tau)) u_\beta + (D_{\beta'} \cdot V_\Sigma(\tau)) u_{\beta'} = 0. \]
		
		Since $\Sigma$ is simplicial, we have that $D_\beta \cdot V_\Sigma(\tau), D_{\beta'} \cdot V_\Sigma(\tau) > 0$. Now recall that the restriction of $-D_\gamma$ to $D_\gamma$ (the conormal bundle) is globally generated by the local convexity assumption (Lemma 3.2 on p. 264 of \cite{LR}). Since we are working with a full-dimensional fan that has convex support (Theorem 6.3.12 on p. 291 of \cite{CLS}), this implies that $-D_\gamma \cdot V_\Sigma(\tau) \ge 0$. In this context, we can rewrite the wall relation as \[  (D_\beta \cdot V_\Sigma(\tau)) u_\beta + (D_{\beta'} \cdot V_\Sigma(\tau)) u_{\beta'} = \sum_{ \gamma \in \tau(1)} (-D_\gamma \cdot V_\Sigma(\tau)) u_\gamma. \]

		Then, each side of the relation gives a nonnegative linear combination of rays. It is clear that the right hand side gives an element of the wall $\tau \in \Sigma(d - 1)$. If $(\beta, \beta') \in \Sigma$, the strict positivity of the coefficients implies that the left hand side gives an element in the \emph{interior} of $(\beta, \beta')$. However, the fact that the cones of $\Sigma$ form a fan means that they can only intersect on boundaries (faces from Definition 3.1.2 on p. 106 and Theorem 2.3.2 on p. 77 of \cite{CLS}). This is not possible unless the cone generated by $\beta$ and $\beta'$ lies inside the wall $\tau$, which is impossible since $\Sigma$ is simplicial. Thus, we have that $(\beta, \beta') \notin \Sigma$. \\
		
		The proof that $(\alpha, \alpha') \notin \Sigma$ is an application of $V_\Sigma(\tau') \in \mathbb{R}_+[V_\Sigma(\tau)]$ and the flagness of the simplicial complex of cones from a locally convex fan (Proposition 5.3 on p. 279 -- 280 of \cite{LR}). Recall that $D_\alpha \cdot V_\Sigma(\tau) = 0$. Combining this with $V_\Sigma(\tau') \in \mathbb{R}_+[V_\Sigma(\tau)]$ implies that $D_\alpha \cdot V_\Sigma(\tau') = 0$. Since $\alpha \notin \tau'(1)$ and $\Sigma$ is simplicial, this implies that $(\alpha, \tau') \notin \Sigma$. The flagness of $\Sigma$ implies that there is \emph{some} pair of rays among $\alpha$ and $\tau'(1)$ which do \emph{not} form a cone in $\Sigma$. The only possible candidate is $(\alpha, \alpha')$, which means that $(\alpha, \alpha') \notin \Sigma$. \\ 
		
	\end{proof}

	\begin{rem} \textbf{(Pictures and local convexity) \\} \label{picloconv}
		Rewriting the inequality $-D_\gamma \cdot V_\Sigma(\tau) \ge 0$ from the proof of Lemma \ref{zeroconorm4cycle} as $D_\gamma \cdot V_\Sigma(\tau) \le 0$, the nef property of the conormal bundle has a direct connection to the ``convex'' case of a wall crossing moving between the off-wall rays with respect to a particular on-wall ray from Lemma 14-1-7 on p. 421 of \cite{Mat}. The expansion of $\beta$ with respect to the $\sigma'(1)$-basis and the expansion of $\beta'$ with respect to the $\sigma(1)$-basis both have nonnegative coefficients of $u_\gamma$. \\
		
		\noindent In this context, being ``flat'' with respect to the ray is that we can move between the off-wall rays using the span of a subcone of the wall that doesn't contain the ray we are considering. The coefficient of $u_\gamma$ in these same expansions would be 0. This is relevant to visualizing the affine equivalence of the graph of the support function of the conormal bundle $-D_\rho|_{D_\rho}$ and the lift of the support of $\lk_\Sigma(\rho)$ from $N(\rho)_{\mathbb{R}}$ to $N_{\mathbb{R}}$ in Proposition A.1 on p. 282 -- 283 of \cite{LR}. Note that $N(\rho)_{\mathbb{R}}$ implicitly makes a choice of a codimension 1 vector space of $N_{\mathbb{R}}$ that does not contain $\rho$. \\
	\end{rem}

	The results above also give a possible algebraic structure on the combinatorial suspension structures studied in Section \ref{wallfan}. \\

	\begin{cor} \label{suspalg}
		Suppose that we are in the setting of Lemma \ref{zeroconorm4cycle}. The unique ray $\alpha'$ with the properties listed in Lemma \ref{zeroconorm4cycle} has $D_{\alpha'} \cdot V_\Sigma(\tau') = 0$. Note that $\tau$ and $\tau'$ have the same off-wall rays inducing flat wall crossings with respect to $\alpha$ and $\alpha'$ (Lemma 14-1-7 on p. 421 of \cite{Mat}) using the maximal cones containing them. \\
		
	\end{cor}
	
	\begin{proof}
		Since $(\alpha, \alpha') \notin \Sigma(2)$ by Lemma \ref{zeroconorm4cycle}, we have that $D_{\alpha'} \cdot V_\Sigma(\tau) = 0$. Then $V_\Sigma(\tau') \in \mathbb{R}_+[V_\Sigma(\tau)]$ implies that we also have $D_{\alpha'} \cdot V_\Sigma(\tau') = 0$. \\
	\end{proof}

	\begin{rem} \textbf{(Connection to material in Section \ref{wallfan}) \\} \label{wallsuspalg}
		Applying this for $\alpha = \rho_{i_j}$ for $j \in A \subset [p]$ and $\alpha = \delta$ for a non-special ray $\delta$ of $\lk_\Sigma(\sigma_{i_1, \ldots, i_p})$ with respect to $A$ or rays $\rho_{i_j}$ inducing the special property on a ray $\gamma$. This may be related to the ``other side'' of the suspension structures (i.e. the only possible element not in the link of the ``first'' one) inducing conormal bundle restrictions equal to 0 in Proposition \ref{conormrestrsusp} and Corollary \ref{suspspecnon}. However, we note that the construction of $\alpha'$ itself depends on the wall $\tau$ that we started with. \\
	\end{rem}

	A more concise description of the main result of Lemma \ref{zeroconorm4cycle} is rewritten below. \\

	\begin{cor}
		Conormal bundle restrictions equal to 0 on toric varieties $X_\Sigma$ associated to simplicial fans with locally convex fan $\Sigma$ yield rays and 2-cones contained in induced 4-cycles. This includes any ray divisors yielding conormal bundles that are nef and not big. If the signature of $X_\Sigma$ is 0, then any ray is contained in an induced 4-cycle. \\
	\end{cor}
	
	\begin{proof}
		The first part follows directly from Lemma \ref{zeroconorm4cycle}. If the restriction of $-D_\rho$ to $D_\rho$ is nef and not big for some ray $\rho \in \Sigma(1)$, it is \emph{not} ample. This is because $D$ is big if and only if $\dim P_D = \dim X_\Sigma$ (p. 427 of \cite{CLS}) and $\dim P_D = \dim X_\Sigma$ if $D$ is ample (p. 272 of \cite{CLS}). Then, the toric version of Kleiman's criterion (Theorem 6.3.13 on p. 292 of \cite{CLS}) to $V_\Sigma(\rho)$ and rescaling by a multiplicity implies that there is some wall $\tau \in \Sigma(d - 1)$ containing $\rho$ such that $D_\rho \cdot V_\Sigma(\tau) = 0$. Note that $\tau \ni \rho$ is necessary since $\Sigma$ is assumed to be simplicial. This yields a conormal bundle restriction that is equal to 0 and Lemma \ref{zeroconorm4cycle} implies that $\rho$ is contained in an induced 4-cycle. If the signature of $X_\Sigma$ is 0, this applies to \emph{any} ray $\rho \in \Sigma(1)$ (Lemma 3.2(iii) on p. 264 and proof of Theorem 1.2(iii) on p. 266 of \cite{LR}). \\
	\end{proof}
	
	We can give some additional context based on what conormal bundles are known to restrict to 0 based on the mixed volume 0 conditions in Theorem \ref{zerovolcrit} applied to monomials that vanish when the signature of $X_\Sigma$ is $0$. \\
	
	\begin{cor} \label{sig04cycle}
		Every ray of a complete simplicial locally convex fan yielding a toric variety with signature 0 is contained in an induced 4-cycle. Also, every 2-cone given by a special ray and a non-special ray (Definition \ref{specnonspecdef}, Corollary \ref{mixvol0toconorm0}) of the link over a fixed $p$-cone ($1 \le p \le \frac{d}{2}$) or a minimal monomial ray paired with a special ray of a $p$-cone containing it is contained in an induced 4-cycle. \\
	\end{cor}
	
	\color{black}

	We take a closer look at where the induced 4-cycles come from and what the fourth vertex of the induced 4-cycle looks like when we have an extremal wall. \\
	
	\begin{defn} (p. xix -- xx of \cite{Sch} and p. 27 of \cite{CLS}) \\ \label{relint}
		The \textbf{relative interior} of a set $A \subset \mathbb{R}^n$ is the interior relative to its affine hull, which is the set of all affine combinations of elements of $A$. In the case of a cone $\sigma$, this means linear combinations of rays in $\sigma(1)$ where each ray has a strictly positive coefficient. \\
	\end{defn}
	
	\begin{lem} (proof of Proposition 14-1-5 on p. 421 -- 422 of \cite{Mat}) \label{relintuniq} \\
		Let $\Sigma$ be a complete simplicial fan of dimension $d$. Given $\alpha \in \tau(1)$ such that $D_\alpha \cdot V_\Sigma(\tau) \ge 0$, $\tau$ is the only $(d - 1)$-dimensional cone in $\Sigma$ containing $\tau \setminus \alpha$ whose relative interior is contained in the half-space $\{ z \in N_{\mathbb{R}} : \langle \alpha^*, z \rangle > 0 \}$, where $\alpha^*$ denotes the dual of $\alpha$ with respect to the basis given by the rays of one of the two maximal cones containing $\tau$. \\
	\end{lem}

	\begin{rem} \textbf{(Conditions in lemma vs. our setting) \\}
		Since we are considering nef restrictions of conormal bundles $-D_\alpha|_{D_\alpha}$, the relevant intersection conditions involve walls $\tau$ containing $\alpha$ such that $D_\alpha \cdot V_\Sigma(\tau) \le 0$ (e.g. see Lemma 14-1-7 on p. 421 of \cite{Mat}). So, the only relevant case here is technically $D_\alpha \cdot V_\Sigma(\tau) = 0$. However, we are including the $D_\alpha \cdot V_\Sigma(\tau) > 0$ case mentioned in the reference for completeness. \\
	\end{rem}

	\begin{proof}
		A wall containing $\tau \setminus \alpha$ can be written as $\eta = (\tau \setminus \alpha, \beta)$ for some ray $\beta$. We can write

		\begin{align*}
			\beta &= a_\gamma u_\gamma + \sum_{\omega \in \tau(1)} a_\omega u_\omega \\
			&= b_{\gamma'} u_{\gamma'} + \sum_{\omega \in \tau(1)} b_\omega u_\omega 
		\end{align*}
		
		with respect to the basis expansions from the rays of the two maximal cones $\sigma$ and $\sigma'$ containing $\tau$ (see Part 3 of proof of Corollary \ref{mixvol0toconorm0}). Note that $a_\gamma, b_{\gamma'} \ne 0$ since $\Sigma$ is simplicial and that they have opposite signs. We will take the dual $\alpha^*$ with respect to the second basis from $\sigma'$. \\
		
		This means that an element in relative interior of $\eta$ can be written as
		
		\begin{align*}
			\sum_{\omega \in \tau(1) \setminus \alpha} c_\omega u_\omega + c_\beta u_\beta &=  \sum_{\omega \in \tau(1) \setminus \alpha} c_\omega u_\omega + c_\beta \left( a_\gamma u_\gamma + \sum_{\omega \in \tau(1)} a_\omega u_\omega \right) \\ 
			&= \sum_{\omega \in \tau(1) \setminus \alpha} c_\omega u_\omega + c_\beta \left( b_{\gamma'} u_{\gamma'} + \sum_{\omega \in \tau(1)} b_\omega u_\omega \right).
		\end{align*}
		
		using strictly positive coefficients $c_\omega, c_\beta > 0$ for $\omega \in \tau(1)$. \\

		Before writing down the entire proof, we will give an overview of the main ideas. Since we are considering $\eta = (\tau \setminus \alpha, \beta)$, assigning large values of $c_\omega$ for $\omega \in \tau(1) \setminus \alpha$ compared to $c_\beta$ will force positive coefficients of $u_\omega$ for $\omega \in \tau(1) \setminus \alpha$. Since we are working with the half-space $\{ z \in N_{\mathbb{R}} : \langle \alpha^*, z \rangle > 0 \}$, we have that $b_\alpha > 0$. We also note that $a_\gamma > 0 \Longleftrightarrow b_{\gamma'} < 0$ and $a_\gamma < 0 \Longleftrightarrow b_{\gamma'} > 0$. We will focus on whichever basis expansion has the positive coefficient for the off-wall ray. If $D_\alpha \cdot V_\Sigma(\tau) = 0$, the wall relations would imply that $a_\alpha = b_\alpha$ (from $a_\alpha - b_\alpha = 0$ -- see Part 3 of the proof of Corollary \ref{mixvol0toconorm0}). In particular, this would mean that $a_\alpha > 0$. Now suppose that $D_\alpha \cdot V_\Sigma(\tau) > 0$. If $a_\gamma > 0$, the wall relations would imply that $a_\alpha - b_\alpha > 0 \Longrightarrow a_\alpha > b_\alpha > 0$. If $a_\gamma < 0$, it suffices to consider $b_\alpha > 0$. Putting these together, we always end up with strictly positive coefficients for all the coefficients in the basis expansion and $\beta$ ends up lying in the interior of $\sigma$ or $\sigma'$ respectively. This would contradict the fact that cones of a fan intersect on faces (which lie in the boundary). \\

		We now start writing down the arguments outlined above in terms of explicit basis elements. Since the dual is taken with respect to the basis given by the rays of $\sigma'$ and minimal rays are positive multiples of the original ray, we have that $\langle \alpha^*, u_\omega \rangle = 0$ for all rays $\omega \in \tau(1) \setminus \alpha$. This means that  
		
		\begin{align*}
			\left\langle \alpha^*, \sum_{\omega \in \tau(1) \setminus \alpha} c_\omega u_\omega + c_\beta u_\beta \right\rangle  &= c_\beta (a_\gamma \langle \alpha^*, u_\gamma \rangle + a_\alpha \langle \alpha^*, u_\alpha \rangle). 
		\end{align*}
		
		Note that $c_\beta$ doesn't affect the sign since $c_\beta > 0$ and that $\langle \alpha^*, u_\alpha \rangle > 0$ since $u_\alpha$ is a positive multiple of $\alpha$. Since $D_\alpha \cdot V_\Sigma(\tau)$ has the same sign as $-\langle \alpha^*, \gamma \rangle$ and $D_\alpha \cdot V_\Sigma(\tau) \ge 0$ and $u_\gamma$ is a positive multiple of $\gamma$, we have that $\langle \alpha^*, u_\gamma \rangle \le 0$. We split this into two cases. \\
		
		\textbf{Case 1: $D_\alpha \cdot V_\Sigma(\tau) = 0$ (which is equivalent to $-\langle \alpha^*, \gamma \rangle = 0$) \\}
		
		In this case, we have that
		
		\begin{align*}
			\left\langle \alpha^*, \sum_{\omega \in \tau(1) \setminus \alpha} c_\omega u_\omega + c_\beta u_\beta \right\rangle  &= c_\beta (a_\gamma \langle \alpha^*, u_\gamma \rangle + a_\alpha \langle \alpha^*, u_\alpha \rangle) \\
			&= c_\beta a_\alpha \langle \alpha^*, u_\alpha \rangle.
		\end{align*} 
		
		Since $c_\beta > 0$ and $\langle \alpha^*, u_\alpha \rangle > 0$, this inner product is strictly positive if and only if $a_\alpha > 0$. It remains to consider the other coefficients in the expansion 
		
		\begin{align*}
			\beta &= a_\gamma u_\gamma + \sum_{\omega \in \tau(1)} a_\omega u_\omega \\
			&= b_{\gamma'} u_{\gamma'} + \sum_{\omega \in \tau(1)} b_\omega u_\omega. 
		\end{align*}
		
		We focus on the expansion with respect to the first basis. Since $\eta \supset \tau \setminus \alpha$, setting $c_\omega > 0$ sufficiently large for $\omega \in \tau(1) \setminus \alpha$ would mean that the coefficient of $\omega$ in \[ \sum_{\omega \in \tau(1) \setminus \alpha} c_\omega u_\omega + c_\beta u_\beta =  \sum_{\omega \in \tau(1) \setminus \alpha} c_\omega u_\omega + c_\beta \left( a_\gamma u_\gamma + \sum_{\omega \in \tau(1)} a_\omega u_\omega \right) \] is strictly positive regardless of the sign of the coefficient $a_\omega$. Since $a_\alpha > 0$ and $c_\beta > 0$, having $a_\gamma > 0$ would mean that we have a point in the relative interior of $\sigma$. This would be a point of intersection of the relative interior of $\eta = (\tau \setminus \alpha, \beta)$ and that of $\sigma$. However, this is impossible since cones belonging to a fan can only intersect on boundaries (specifically subcones/faces). \\
		
		Suppose that $a_\gamma < 0$. We will compare the two basis expansions 
		
		\begin{align*}
			\beta &= a_\gamma u_\gamma + \sum_{\omega \in \tau(1)} a_\omega u_\omega \\
			&= b_{\gamma'} u_{\gamma'} + \sum_{\omega \in \tau(1)} b_\omega u_\omega
		\end{align*}
		
		and use the reasoning involving the wall relations from Part 3 of the proof of Corollary \ref{mixvol0toconorm0}. Since $D_\alpha \cdot V_\Sigma(\tau) = 0$, the wall relations imply that $a_\alpha = b_\alpha$. We have already shown that $b_\alpha = a_\alpha > 0$. They also imply that $b_{\gamma'} > 0$ if $a_\gamma < 0$. If we set a take $c_\omega > 0$ to be sufficiently large for $\omega \in \tau(1) \setminus \alpha$, then 
		
		\[ \sum_{\omega \in \tau(1) \setminus \alpha} c_\omega u_\omega + c_\beta u_\beta = \sum_{\omega \in \tau(1) \setminus \alpha} c_\omega u_\omega + c_\beta \left( b_{\gamma'} u_{\gamma'} + \sum_{\omega \in \tau(1)} b_\omega u_\omega \right) \]
		
		is a point in the relative interior of $\sigma'$ which is also in the relative interior of $\eta = (\tau \setminus \alpha, \beta)$. Since this is impossible, we cannot have $a_\gamma < 0$. \\
		
		This means that $a_\gamma = 0$, which is equivalent to $b_{\gamma'} = 0$. Then, the ray $\beta$ is in the span of the rays of $\tau$. Suppose that $\eta \ne \tau$. Then, there are some $\ell \in \tau \setminus \alpha$ such that $a_\ell \ne 0$. Then, setting $c_\ell > 0$ sufficiently large for such $\ell \in \tau(1)$ in \[ \sum_{\omega \in \tau(1) \setminus \alpha} c_\omega u_\omega + c_\beta u_\beta =  \sum_{\omega \in \tau(1) \setminus \alpha} c_\omega u_\omega + c_\beta \left(\sum_{\omega \in \tau(1)} a_\omega u_\omega \right) \]
		
		would yield a point in the relative interior of $\tau$ \emph{and} that of $\eta$, which is not possible for two distinct cones of a fan. Thus, $\eta = \tau$ and $\tau$ is the only wall containing $\tau \setminus \alpha$ whose relative interior is contained in $\{ z \in N_{\mathbb{R}} : \langle \alpha^*, z \rangle > 0 \}$. \\

		\textbf{Case 2: $D_\alpha \cdot V_\Sigma(\tau) > 0$  (which is equivalent to $-\langle \alpha^*, \gamma \rangle > 0$ since $D_\alpha \cdot V_\Sigma(\tau) $ is a strictly positive multiple of $-\langle \alpha^*, \gamma \rangle$) \\}

		Recall that 
		
		\begin{align*}
			\left\langle \alpha^*, \sum_{\omega \in \tau(1) \setminus \alpha} c_\omega u_\omega + c_\beta u_\beta \right\rangle  &= c_\beta (a_\gamma \langle \alpha^*, u_\gamma \rangle + a_\alpha \langle \alpha^*, u_\alpha \rangle). 
		\end{align*}
		
		Our assumption above implies that $\langle \alpha^*, \gamma \rangle < 0$. The inner product above is strictly positive if and only if \[ a_\alpha > a_\gamma \frac{-\langle \alpha^*, \gamma \rangle}{\langle \alpha^*, u_\alpha \rangle}. \] 
		
		If $a_\gamma > 0$, this implies that $a_\alpha > 0$. As in Case 1, setting $c_\omega > 0$ sufficiently large for $\omega \in \tau(1) \setminus \alpha$ in \[ \sum_{\omega \in \tau(1) \setminus \alpha} c_\omega u_\omega + c_\beta u_\beta =  \sum_{\omega \in \tau(1) \setminus \alpha} c_\omega u_\omega + c_\beta \left( a_\gamma u_\gamma + \sum_{\omega \in \tau(1)} a_\omega u_\omega \right) \] would produce a point of $\eta$ in the relative interior of $\sigma$, which cannot be contained in a cone of the fan $\Sigma$ other than $\sigma$ itself. \\
		
		Suppose that $a_\gamma < 0$. Then, the lower bound in the inequality \[ a_\alpha > a_\gamma \frac{-\langle \alpha^*, \gamma \rangle}{\langle \alpha^*, u_\alpha \rangle} \] is negative. Recall that we have two basis expansions

		\begin{align*}
			\beta &= a_\gamma u_\gamma + \sum_{\omega \in \tau(1)} a_\omega u_\omega \\
			&= b_{\gamma'} u_{\gamma'} + \sum_{\omega \in \tau(1)} b_\omega u_\omega
		\end{align*}
		
		from the maximal cones containing $\tau$. Note that $b_{\gamma'} > 0$ since $a_\gamma < 0$. \\
		
		In order for the relative interior of $\eta = (\tau \setminus \alpha, \beta)$ to be contained in $\{ z \in N_{\mathbb{R}} : \langle \alpha^*, z \rangle > 0 \}$, we need to have $\langle \alpha^*, \beta \rangle > 0$. This means that $\langle \alpha^*, \beta \rangle = b_\alpha > 0$ since the dual is taken with respect to the second basis. Then, setting $c_\omega > 0$ sufficiently large for $\omega \in \tau(1) \setminus \alpha$ in the strictly positive linear combination \[  \sum_{\omega \in \tau(1) \setminus \alpha} c_\omega u_\omega + c_\beta u_\beta = \sum_{\omega \in \tau(1) \setminus \alpha} c_\omega u_\omega + c_\beta \left( b_{\gamma'} u_{\gamma'} + \sum_{\omega \in \tau(1)} b_\omega u_\omega \right) \] yields a point of $\eta$ in the relative interior of $\sigma'$. This is not possible for a cone other than $\sigma'$ itself. \\
		
		This means that  $a_\gamma = 0$, which is equivalent to having $b_{\gamma'} = 0$. The same reasoning as the end of the proof of Case 1 (after it was shown that $a_\gamma = 0$) then implies that $\eta = \tau$. \\

	\end{proof}
	
	If $D_\alpha \cdot V_\Sigma(\tau) = 0$, there is a stronger condition than what is claimed on p. 421 -- 422 of \cite{Mat}. All sums in the rational equivalence relation modification other than $\tau, \sigma \setminus \alpha$, and $\sigma' \setminus \alpha$ have strictly positive coefficients and all other walls of $\Sigma$ containing $\tau \setminus \alpha$ are contained in the half-space on the ``negative'' side of $\alpha$. In addition, flagness implies that there is only one wall on the negative side of $\alpha^*$ and that it is the unique wall adjacent to $\gamma = \sigma \setminus \tau$ and $\gamma' = \sigma' \setminus \tau$ containing $\tau \setminus \alpha$ on this side of $\alpha^*$. This yields a source of a suspension structure using rational equivalence relations and intersections with 2-dimensional torus orbits in addition to the explicit combinatorial analysis of the fans from Section \ref{wallfan}. \\ 
	
	\begin{cor} \label{4cycleconstralg}
		Let $\Sigma$ be a complete simplicial fan of dimension $d$. Consider a ray $\alpha \in \tau(1)$ such that $D_\alpha \cdot V_\Sigma(\tau) = 0$. Let $\alpha^*$ be the dual of $\alpha$ with respect to a basis formed by the rays of one of the two maximal cones containing $\tau$. 
		
		\begin{enumerate}
			\item If $\eta$ is a wall of $\Sigma$ containing $\tau \setminus \alpha$, all points of the relative interior of $\eta$ are either on one side of the hyperplane defined by $\alpha^*$ or contained in the hyperplane itself. The former refers to one of the half-spaces defined by $\alpha^*$. \\

			\item Suppose that $\eta$ is a wall containing $\tau \setminus \alpha$ that is \emph{not} equal to any of $\tau$, $\sigma \setminus \alpha = (\tau \setminus \alpha, \gamma)$, or $\sigma' \setminus \alpha = (\tau \setminus \alpha, \gamma')$. Note that such a wall exists by completeness. Then, we have that $\eta \subset \{ z \in N_{\mathbb{R}} : \langle \alpha^*, z \rangle < 0 \}$ and $[V_\Sigma(\eta)] \in \mathbb{R}_{\ge 0}[V_\Sigma(\tau)]$. \\
			
			\item Suppose that we are in the setting of Part 2. If the cones of $\Sigma$ form a flag simplicial complex (e.g. if $\Sigma$ is locally convex), then $\eta$ is the unique wall containing $\tau \setminus \alpha$ that is adjacent to $\gamma = \sigma \setminus \tau$ and $\gamma' = \sigma' \setminus \tau$ whose relative interior is on the negative side of the half-space defined by $\alpha^*$. The restrictions of the ray divisors to $V_\Sigma(\tau \setminus \alpha)$ satisfy the following properties: \\
			\begin{itemize}
				\item $D_\beta|_{V_\Sigma(\tau \setminus \alpha)} \in \mathbb{R}_{\ge 0} [V_\Sigma(\tau)]$ \\
				
				\item  $D_\alpha|_{V_\Sigma(\tau \setminus \alpha)} \in \mathbb{R}_{\ge 0} [V_\Sigma(\tau)]$ is \emph{not} a scalar multiple of $D_\gamma|_{V_\Sigma(\tau \setminus \alpha)}$ or $D_{\gamma'}|_{V_\Sigma(\tau \setminus \alpha)}$. \\
				
				\item $D_{\gamma'}|_{V_\Sigma(\tau \setminus \alpha)} \in \Span (D_\alpha|_{V_\Sigma(\tau \setminus \alpha)}, D_\gamma|_{V_\Sigma(\tau \setminus \alpha)})$ and the coefficient of $D_\gamma|_{V_\Sigma(\tau \setminus \alpha)})$ in the linear combination is negative. \\
			\end{itemize} 
		\end{enumerate}

	\end{cor}
	
	\begin{proof}
		
		\begin{enumerate}
			\item Let $\eta = (\tau \setminus \alpha, \beta)$ be a wall of $\Sigma$ containing $\tau \setminus \alpha$. As in the proof of Lemma \ref{relintuniq}, we can write

			\begin{align*}
				\beta &= a_\gamma u_\gamma + \sum_{\omega \in \tau(1)} a_\omega u_\omega \\
				&= b_{\gamma'} u_{\gamma'} + \sum_{\omega \in \tau(1)} b_\omega u_\omega 
			\end{align*}
			
			with respect to the basis expansions from the rays of the two maximal cones $\sigma$ and $\sigma'$ containing $\tau$ from Part 3 of proof of Corollary \ref{mixvol0toconorm0}. \\
			
			We will also take the dual basis with respect to the second basis and denote the dual element of $\alpha$ by $\alpha^*$. In this notation, we have $D_\alpha \cdot V_\Sigma(\tau) = -\langle \alpha^*, \gamma \rangle = 0$. Since $\langle \alpha^*, \rho \rangle = 0$ for any ray $\rho \in \sigma'(1) \setminus \alpha$, we have that $\langle \alpha^*, \rho \rangle = 0$ for any ray $\rho \in \tau \setminus \alpha$. This means that the signs of $\langle \alpha^*, z \rangle$ for points $z$ in the relative interior of $\eta$ are all equal to the sign of $\langle \alpha^*, \beta \rangle = b_\alpha$. In particular, this means that the sign of $\langle \alpha^*, z \rangle$ is the same for all points $z$ of the relative interior of $\eta$. \\

			\item Let $\eta = (\tau \setminus \alpha, \beta)$ be a wall containing $\tau \setminus \alpha$ that is \emph{not} equal to $\tau$, $\sigma \setminus \alpha = (\tau \setminus \alpha, \gamma)$, or $\sigma' \setminus \alpha = (\tau \setminus \alpha, \gamma')$. Suppose that $\langle \alpha^*, \beta \rangle = b_\alpha = 0$. Then, we have that $b_{\gamma'} \ne 0$ since $\beta \notin \tau(1)$ and $\Sigma$ is simplicial. If $b_{\gamma'} > 0$, then setting $c_\omega > 0$ sufficiently large for $\omega \in \tau(1) \setminus \alpha$ in the strictly positive linear combination \[  \sum_{\omega \in \tau(1) \setminus \alpha} c_\omega u_\omega + c_\beta u_\beta = \sum_{\omega \in \tau(1) \setminus \alpha} c_\omega u_\omega + c_\beta \left( b_{\gamma'} u_{\gamma'} + \sum_{\omega \in \tau(1) \setminus \alpha} b_\omega u_\omega \right) \] yields a point of $\eta = (\tau \setminus \alpha, \beta)$ in the relative interior of $\sigma' \setminus \alpha$. This is not possible for a cone other than $\sigma' \setminus \alpha$ itself. \\

			If $b_{\gamma'} < 0$, then $a_\gamma > 0$. The assumption that $D_\alpha \cdot V_\Sigma(\tau) = 0$ implies that $a_\alpha = b_\alpha = 0$ since \[ b_\alpha = a_\alpha - \frac{D_\alpha \cdot V_\Sigma(\tau)}{D_\gamma \cdot V_\Sigma(\tau)} a_\gamma = a_\alpha \] by the wall relation (see Part 3 of proof of Corollary \ref{mixvol0toconorm0}). Then, setting $c_\omega > 0$ sufficiently large for $\omega \in \tau(1) \setminus \alpha$ in the strictly positive linear combination \[ \sum_{\omega \in \tau(1) \setminus \alpha} c_\omega u_\omega + c_\beta u_\beta =  \sum_{\omega \in \tau(1) \setminus \alpha} c_\omega u_\omega + c_\beta \left( a_\gamma u_\gamma + \sum_{\omega \in \tau(1) \setminus \alpha} a_\omega u_\omega \right) \] would give a point of $\eta$ in the relative interior of $\sigma \setminus \alpha$. Again, this is not possible since cones of a fan only intersect on boundaries (specifically faces). This means that $\langle \alpha^*, \beta \rangle = b_\alpha \ne 0$ and $\eta$ must be contained in the interior of one of the half-spaces defined by $\alpha^*$. If $b_\alpha > 0$, then $\eta = (\tau \setminus \alpha, \beta)$ would contain an interior point of $\tau$. This is impossible unless $\eta = \tau$. Thus, we have that $b_\alpha < 0$ and $\eta \subset \{ z \in N_{\mathbb{R}} : \langle \alpha^*, z \rangle < 0 \}$. Such a wall containing $\tau \setminus \alpha$ exists by completeness. \\

			Finally, the statement that $[V_\Sigma(\eta)] \in \mathbb{R}_{\ge 0} [V_\Sigma(\tau)]$ follows from multiplying the rational equivalence relation specialized to inner products with $\alpha^*$, multiplying by the 2-dimensional closed orbit $V_\Sigma(\tau \setminus \alpha)$, and using the assumption that $[V_\Sigma(\tau)]$ is extremal as in p. 421 -- 422 of \cite{Mat}. \\

			\item  If the cones in $\Sigma$ form a flag simplicial complex, then the subcollection in $\lk_\Sigma(\tau \setminus \alpha)$ yields a 2-dimensional simplicial fan/1-dimensional flag sphere. This must be an $n$-gon for some $n \ge 4$. There is a wall containing $\tau \setminus \alpha$ that is adjacent to $\gamma$ and $\gamma'$ while being on the negative side of the half-space defined by $\alpha^*$ (Part(ii) of Proposition 14-1-5 on p. 419 -- 420 of \cite{Mat}). The ray $\beta \in \lk_\Sigma(\tau \setminus \alpha)$ inducing such a wall is adjacent to both $\gamma$ and $\gamma'$. Such a ray can only exist in $\lk_\Sigma(\tau \setminus \alpha)$ if $n = 4$. \\

			As for the restrictions of classes, we start with $(f_0, f_1) = (n, n)$. The Dehn--Sommerville relations imply that the Betti numbers are $(h_0, h_1, h_2) = (1, n - 2, 1)$ (Corollary 5.1.9 on p. 213 of \cite{BH}). It was already known that $D_\beta|_{V_\Sigma(\tau \setminus \alpha)} \in \mathbb{R}_{\ge 0} [V_\Sigma(\tau)]$ (Part(ii) of Proposition 14-1-5 on p. 419 -- 420 of \cite{Mat}). We take a closer look structure of the classes beyond the dimension. First of all, the classes $D_\gamma|_{V_\Sigma(\tau \setminus \alpha)}$ and $D_{\gamma'}|_{V_\Sigma(\tau \setminus \alpha)}$ are \emph{not} scalar multiples of $D_\alpha|_{V_\Sigma(\tau \setminus \alpha)}$ since $D_\alpha \cdot V_\Sigma(\tau) = 0$ while $D_\gamma \cdot V_\Sigma(\tau) > 0$ and $D_{\gamma'} \cdot V_\Sigma(\tau) > 0$. If the earlier restrictions were equal, then multiplying by $D_\alpha$ on both sides of each pair would have yielded equalities. Since $n = 4$, the Betti numbers are given by $(h_0, h_1, h_2) = (1, 2, 1)$. This means that $D_{\gamma'}|_{V_\Sigma(\tau \setminus \alpha)} \in \Span (D_\alpha|_{V_\Sigma(\tau \setminus \alpha)}, D_\gamma|_{V_\Sigma(\tau \setminus \alpha)})$. In addition, the coefficient of $D_\gamma|_{V_\Sigma(\tau \setminus \alpha)})$ in the linear combination is negative since $V_\Sigma(\tau) \cong \mathbb{P}^1$. \\

		\end{enumerate}

	\end{proof}
	
	In the context of intersection patterns of induced 4-cycles, we note that another way of thinking about the mixed volume conditions from Corollary \ref{mixvol0toconorm0} is to view it as a result about intersection patterns of extremal walls of links of $p$-cones of $\Sigma$ for $1 \le p \le \frac{d}{2}$ coming from top degree monomials with these support sizes whose ray divisors have even global exponents/odd local exponents on the restrictions. \\

	\color{black} 
			
\end{document}